\documentclass[reqno]{amsart}
\usepackage{amsfonts}
\usepackage{mathrsfs}
\usepackage{color}
\usepackage{cite}
\usepackage{amscd}
\usepackage{latexsym}
\usepackage{amsfonts}
\usepackage{graphicx}
\usepackage{CJK,indentfirst,amsmath,amsfonts,amssymb,amsthm,cite,cases, subeqnarray,setspace}
\allowdisplaybreaks[4]

\begin{document}
\title[Modulation approximation for the non-isentropic Euler-Poisson system]
{Modulation approximation for the non-isentropic\\ Euler-Poisson system}
\author{Huimin Liu and Xueke Pu*}

\address{Huimin Liu \newline
Faculty of Applied Mathematics, Shanxi University of Finance and Economics, Taiyuan 030006, P.R. China}
\email{hmliucqu@163.com}

\address{Xueke Pu \newline
Department of Applied Mathematics, Guangzhou University, Guangzhou 510006, P.R. China}
\email{xuekepu@gzhu.edu.cn}

\thanks{*Corresponding author: Xueke Pu, xuekepu@gzhu.edu.cn\\
The first author H. Liu was supported by the Basic Research Project of Shanxi Province of China under 202403021211156 and 202303021211140, the second author X. Pu was supported by NSFC (Grant 12471220) and the Natural Science Foundation of Guangdong Province of China under 2019A1515012000.}

\begin{abstract}
As a formal approximation, the nonlinear Schr\"{o}dinger (NLS) equation can be derived to describe the evolution of the envelopes of small oscillating wave packets-like solutions to the Euler-Poisson system. In this paper we rigorously justify that the wave packets for the non-isentropic Euler-Poisson system can be approximated  by solutions of the NLS equation over a physically relevant $\mathcal{O}(\epsilon^{-2})$ time scale. Besides the difficulties such as resonances at $k=0$ and $k=\pm k_0$ and loss of derivatives arising in the modulation approximation problem in the isentropic Euler-Poisson system, new difficulties arise in the non-isentropic case. In the non-isentropic Euler-Poisson system, new resonances at wave number $k=\pm 2k_0$ appear which necessitate rescaling the correction to the modulation approximation differently for different wave numbers. In addition, it is more difficult to obtain the uniform estimates for the error $(R_{0},R_{1},R_{-1})$ between the real solutions and the approximate solutions, due to the extra interactions with the temperature. To overcome the difficulties aroused by resonances and loss of derivatives, we find several important structural identities between the diagonalized unknowns and apply a series of normal-form transforms, to obtain uniform estimates for the error over the desired $\mathcal{O}(\epsilon^{-2})$ long time scale.
\end{abstract}

\maketitle \numberwithin{equation}{section}
\newtheorem{proposition}{Proposition}[section]
\newtheorem{theorem}{Theorem}[section]
\newtheorem{lemma}[theorem]{Lemma}
\newtheorem{remark}[theorem]{Remark}
\newtheorem{hypothesis}[theorem]{Hypothesis}
\newtheorem{definition}{Definition}[section]
\newtheorem{corollary}{Corollary}[section]
\newtheorem{assumption}{Assumption}[section]

\tableofcontents{}

\section{\textbf{Introduction}}
\setcounter{section}{1}\setcounter{equation}{0}

In one-dimensional space, the full Euler-Poisson system for ions in a plasma is governed by
\begin{subequations}\label{fEP}
\begin{numcases}{}
\partial_{t}n+\partial_{x}(nu)=0,\\
m\partial_{t}(nu)+\partial_{x}(mnu^{2}+p)+en\partial_{x}\phi=0,\\
\partial_{t}w+\partial_{x}(w u+pu)+enu\partial_{x}\phi=0,\\
-\partial_{x}^{2}\phi+4\pi e\bar{n}e^{e\phi/T_{e}}=4\pi en,
\end{numcases}
\end{subequations}
where $n(x,t)$, $u(x,t)$, $w(x,t)$ and $\phi(x,t)$ are the number density, the velocity, the total energy and the electric potential of the ions at time $t$ and position $x$, respectively. This is a widely used important model to depict the dynamics of the ions in a plasma.  In kinetic theory, the Euler-Poisson system \eqref{fEP} can be formally derived from Vlasov-Poisson-Boltzmann system for the ion-acoustic flow through the macro-micro decomposition around local Maxwellians \cite{DL15}. In above, $\gamma>1$ is the ratio of specific heat, the pressure function $p(x,t)$ satisfies the state equation $p=RnT$ and the total energy $w(x,t)$ is defined by
\begin{align*}
w=\frac{1}{2}mnu^{2}+\frac{p}{\gamma-1},
\end{align*}
where $T(x,t)$ represents the temperature of ions and $R$ is the gas constant. The constants $m$, $e$, $T_{e}$ and $\bar{n}$ represent the mass of ions, the electron charge, the temperature of electron and the equilibrium density of electrons, respectively. In the unknowns $(n,u,T,\phi)$, we have the following non-isentropic Euler-Poisson system
\begin{subequations}\label{EP}
\begin{numcases}{}
\partial_{t}n+\partial_{x}(nu)=0,\\
m\partial_{t}u+mu\partial_{x} u+R\partial_{x}T+R\frac{\partial_{x}n}{n}T+e\partial_{x}\phi=0,\\
\partial_{t}T+u\partial_{x}T+(\gamma-1)T\partial_{x} u=0,\\
-\partial_{x}^{2}\phi+4\pi e\bar{n}e^{e\phi/T_{e}}=4\pi en.
\end{numcases}
\end{subequations}
For the sake of clarity, we set all the parameters $m,R,4\pi e, \bar{n}$ and $e/T_{e}$ to be $1$ in this paper.

The Euler-Poisson system was extensively studied in the past two decades. The list could be very long, so for readers' convenience, we only introduce some results mainly concerning the global well-posedness and the singular limits, in particular the long wavelength limits, for the Cauchy problem. For other interesting aspects, such as derivation of the model, stability in various situations, quasineutral limits and so on,  that are not mentioned here, the interested readers can refer to references cited therein. Guo \cite{Guo98} firstly obtained the global irrotational solutions with small initial data for the 3D electron Euler-Poisson system. The 2D global smooth irrotational small solutions are constructed for the electron Euler-Poisson system \cite{IP13, Jang12, JLZ14, LW14}. Furthermore, Guo, Han and Zhang \cite{GHZ17} proved that no shocks form for the 1D electron Euler-Poisson system. For the ion Euler-Poisson system, Guo and Pausader \cite{GP11} established the global irrotational solution by using the scattering theory. For the long wavelength limit for the ion Euler-Poisson system, the formal derivation of the KdV equation was conducted early in \cite{SG69}. Guo and Pu \cite{GP14} proved  rigorously the mathematical justification in the cases of hot and cold ions. This result was extended further to Kadomtsev-Petviashvili limit and Zakharov-Kuznetsov limit in high-dimensional spaces \cite{P13,LLS} and to Vlasov-Poisson system in \cite{HK13}. Later on, Liu and Yang \cite{LY20} extended the single directional result to the bidirectional long wavelength limit for the one dimensional Euler-Poisson system. A modified KdV limit result at critical densities can also be obtained \cite{PX21}. In addition, authors \cite{LP16, LP23} obtained the quantum KdV and KP limit in 1D and 2D spaces for a reduced two-fluid Euler-Poisson system. For the non-isentropic Euler-Poisson system \eqref{EP}, there are relatively few results in both physics and mathematics. The existence and the asymptotic stability of stationary solutions for the full Euler-Poisson system in half-line was investigated by Duan \emph{et al} \cite{DYZ21}. 
Convergence for all time of the non-isentropic Euler-Poisson system was studied by Liu and Peng \cite{LP18}. The long wavelength limit for the non-isentropic Euler-Poisson system \eqref{EP} was proved rigorously in \cite{ZY23} recently.

The modulation approximation is an important topic that was extensively studied in recent years, in particular for many important dispersive PDE systems. It is asked how well a solution of a system can be approximated by modulation over desired long timespan. In order not to deviate too much from the main topic, we will give a short remark on the state of the art for this issue after the statement of the main Theorem \ref{Thm1}. For now, we only mention the modulation approximation for the isentropic Euler-Poisson system by the authors \cite{LP19}, where a rigorous justification of the NLS equation was established. However, for the non-isentropic Euler-Poisson system \eqref{EP}, there is currently no relevant discussion on the modulation approximation. This will be conducted in the present paper.

Let us consider the following approximation around the constant solution $(n,u,T)=(1,0,1)$
\begin{equation}
\begin{split}\label{fom}
\begin{pmatrix} n-1 \\ u \\ T-1 \end{pmatrix}=\epsilon\Psi_{NLS}+\mathcal{O}(\epsilon^{2}),
\end{split}
\end{equation}
with
\begin{equation}
\begin{split}\label{LS}
\epsilon\Psi_{NLS}=\epsilon A(\epsilon(x-c_{g}t),\epsilon^{2}t) e^{i(k_{0}x-\omega_{0}t)}\varrho(k_{0})+c.c..
\end{split}
\end{equation}
Then the NLS equation can be derived for the complex amplitude $A$,
\begin{equation}
\begin{split}\label{A}
\partial_{T}A=i\mu_{1}\partial_{X}^{2}A+i\mu_{2}A|A|^{2},
\end{split}
\end{equation}
where $0<\epsilon\ll1$ is a small perturbation parameter, $T=\epsilon^{2}t\in\mathbb{R}$ is the slow time scale and $X=\epsilon(x-c_{g}t)\in\mathbb{R}$ is the slow spatial scale and coefficients $\mu_{j}=\mu_{j}(k_{0})\in \mathbb{R}$ with $j\in\{1,2\}$. Such an approximation is often called \emph{modulation approximation} in literature (see e.g. \cite{TW12} for similar approximations in water wave problems). In the above modulation approximation, $\omega_{0}>0$ is the basic temporal wave number associated to the basic spatial wave number $k_{0}>0$ of the underlying temporally and spatially oscillating wave train $e^{i(k_{0}x-\omega_{0}t)}$, $c_{g}$ is the group velocity and `c.c.' denotes the complex conjugate. The NLS is derived in order to describe the slow modulations in time and in space of the wave train $e^{i(k_{0}x-\omega_{0}t)}$ and the time and space scales of the modulations are $\mathcal{O}(1/\epsilon^{2})$ and $\mathcal{O}(1/\epsilon)$, respectively. For the non-isentropic Euler-Poisson system \eqref{EP}, the basic spatial wave number $k=k_{0}$ and the basic temporal wave number $\omega=\omega_{0}$ are related via the following linear \emph{dispersion relation}
\begin{equation}
\begin{split}\label{equation3}
\omega(k)=k\sqrt{\gamma+\frac{1}{1+k^{2}}}=k\widehat{q}(k), \ \ \ \ \widehat{q}(k)=\sqrt{\gamma+\frac{1}{1+k^{2}}},
\end{split}
\end{equation}
where $\gamma>1$ is the ratio of specific heat. From this dispersion relation, $\varrho(k_{0})=(1, \widehat{q}(k_{0}), \gamma-1)^{T}$ in \eqref{LS} and the group velocity $c_{g}=\omega'(k_{0})$ of the wave packet can be computed. This ansatz \eqref{fom} leads to waves moving to the right. To obtain waves moving to the left, $-\omega_{0}$ and $c_{g}$ have to be replaced by $\omega_{0}$ and $-c_{g}$, respectively.

Such modulation approximation is important in describing slow modulations in time and space of potential spatial and temporal oscillation wave packet in dispersive systems \cite{AS81}. In 1968, Zakharov \cite{Z68} provided the first formal derivation of the modulation approximation for the 2D water wave problem. Although the modulation approximation is very successful in many areas, such as \cite{L01, SS99}, it may also yield false predictions for the behavior of the original system. In fact, there are counterexamples where the modulation approximation fails \cite{S05, SSZ15}. Thus, it is necessary to rigorously justify in mathematics the validity of the modulation approximation for a given system by giving uniform error estimates over a relevant long timescale. There are many attempts in the past decades, especially for the modulation approximation to water wave problems. The interested reader may refer to \cite{D17,D21,D,DS06,DSW16,H22,K88,S98,S05,S1998,S11,SSZ15,SW11,TW12} for rigourous mathematically justification of such approximation in various situations.

%
%

There recently has been some progress on the modulation approximation of the isentropic Euler-Poisson system; see \cite{LP19,LBP24} for instance. However, for the non-isentropic Euler-Poisson system, to the best of our knowledge, there are few results on the same issue. In this paper, we will study the modulation approximation to the one-dimensional non-isentropic Euler-Poisson system \eqref{EP}, or equivalently for \eqref{fEP} and provide uniform error estimates in Sobolev spaces over desired $\mathcal{O}(\epsilon^{-2})$ long timescales, thus rigorously justify the modulation approximation for the non-isentropic Euler-Poisson system. The main result is stated in the following theorem.
\begin{theorem}\label{Thm1}
Fix $s_{A}\geq6$. Then for all $k_{0}\neq0$ and for all $C_{1}, \ T_{0}>0$ there exist $C_{2}>0$ and $\epsilon_{0}>0$ such that for all solutions $A\in C([0,T_{0}],H^{s_{A}}(\mathbb{R},\mathbb{C}))$ of the NLS equation \eqref{A} with
\begin{equation*}
\begin{split}
\sup_{T\in[0,T_{0}]}\|A(\cdot,T)\|_{H^{s_{A}}}(\mathbb{R},\mathbb{C})\leq C_{1},
\end{split}
\end{equation*}
the following holds. For all $\epsilon\in(0,\epsilon_{0})$, there are solutions
\begin{equation*}
\begin{split}
\begin{pmatrix} n-1 \\ u \\ T-1 \end{pmatrix}\in \big(C([0,T_{0}/\epsilon^{2}],H^{s_{A}}(\mathbb{R},\mathbb{R}))\big)^{3}
\end{split}
\end{equation*}
of the non-isentropic Euler-Poisson system \eqref{EP} that satisfy
\begin{equation*}
\begin{split}
\sup_{t\in[0,T_{0}/\epsilon^{2}]}\left\|\begin{pmatrix} n-1 \\ v \\ T-1 \end{pmatrix}-\epsilon\Psi_{NLS}(\cdot,t)\right\|_{(H^{s_{A}}(\mathbb{R},\mathbb{R}))^{3}}
\leq C_{2}\epsilon^{3/2}.
\end{split}
\end{equation*}
\end{theorem}
Before we proceed for the formal derivation and rigorous justification of the modulation approximation to the non-isentropic Euler-Poisson system in subsequent sections, we give some remarks on the literature of this issue. Besides the inherent difficulties, the proof of the main Theorem \ref{Thm1} shares similar difficulties in modulation approximation problems with various systems, such as the water wave system, the Klein-Gordon and the Boussinesq system.

For semilinear systems without quadratic terms, the modulation approximation was justified over $\mathcal{O}(\epsilon^{-2})$ timescales by directly using Gronwall's inequality, as in \cite{KS}. However for systems with nonlinear quadratic terms, such a result is not trivial since direct application of the Gronwall's inequality only yields uniform bounds for the error over $\mathcal{O}(\epsilon^{-1})$ timescales, due to the $\mathcal{O}(\epsilon)$ terms in the error system aroused by the quadratic terms. By a near identity change of variables, the so-called normal-form transform, these $\mathcal{O}(\epsilon)$ terms can be eliminated if there are no resonances or only trivial resonances occur \cite{K88} (`trivial' will be made clear below).  The normal-form transform method was also developed to make it applicable to systems with non-trivial resonances at $k=\pm k_{0}$; refer to \cite{DS06,S98,S1998,S05}.

For quasilinear systems with quadratic nonlinearities, justification of such modulation approximation becomes more difficult, due to the loss of derivatives caused by the nonlinearities. 
For example, for the 2D water wave problem without surface tension and in a canal of finite depth in Lagrangian coordinates, where the quadratic terms lose only half a derivative, the modulation approximation is justified with the help of normal-form transforms and the Cauchy-Kowalevskaya theorem \cite{DSW16,SW11}. For the quasilinear KdV equation, the modulation approximation was justified by simply applying a Miura transform \cite{S11}. For a nonlinear Klein-Gordon equation with a quasilinear quadratic term that loses one derivative, the modulation approximation was justified by using the normal-form transform to construct a modified energy to overcome the problem of loss of one derivative \cite{D17}. We note that no resonances occurs there. For systems where the quadratic terms lose more than one derivative, the modulation approximation can be still justified, but only in several particular cases, where the linear operators provide the same number of derivatives that the nonlinearities lose \cite{D21,H22}.

For the isentropic Euler-Poisson system, we note that the continuous linear dispersion relation causes trivial resonance at $k=0$ and non-trivial resonances at $k=\pm k_{0}$, and the nonlinear quadratic terms lose one derivative resulting the loss of two derivatives in the transformed system when applying the normal-form transform. Although with these difficulties, the authors of this paper rigorously justified the modulation approximation for the isentropic Euler-Poisson system in Sobolev spaces over an $\mathcal{O}(\epsilon^{-2})$ long timescale in \cite{LP19}. In this situation, the linear operators provide one derivative to compensate the loss of regularity. We also justified the modulation approximation for the pressureless isentropic Euler-Poisson system where the linear operators provides no regularities recently \cite{LBP24}.



Back to non-isentropic Euler-Poisson system, Theorem \ref{Thm1} provides the first rigorous justification of the modulation approximation, by showing uniform error estimates over $\mathcal{O}(\epsilon^{-2})$ long timescales. As in many other modulation approximation results, quadratic terms will introduce $\mathcal{O}(\epsilon)$ terms in the error equation, and makes the uniform estimates only lives over $\mathcal{O}(\epsilon^{-1})$ long timescale if we naively apply the Gr\"onwall inequality directly. One classical method is to apply the normal-form transform to eliminate such quadratic terms, as did in the seminal paper of Shatah \cite{Shatah85}. Two problems will arise, i.e., resonances and loss of derivatives. These difficulties appear in both the isentropic \cite{LP19} and non-isentropic Euler-Poisson systems in this paper. However, the situation becomes much more complex and difficult in several aspects in the non-isentropic case in the present paper.

Firstly, both $\mathcal{O}(1)$ and $\mathcal{O}(\epsilon)$ terms will appear in the error system in the non-isentropic case. We try to use normal form transform to eliminate these terms. Specifically, the treatment of the resonances at $k=0$ and $k=\pm k_0$ is closely related the whether or not the quasilinear terms is close to zero in Fourier space. The treatment of other extra resonances will be discussed in detail in the next paragraph. Therefore, as just did in \cite{LP19}, we introduce the weight function \eqref{theta} and accordingly the projections \eqref{projections} to obtain the projected error system \eqref{BR} and \eqref{BR1} for $(R^0_j, R^1_j)$ for $j\in\{0,\pm1\}$. Roughly, the $(R^0_j, R^1_j)$ represents the error of the diagonalized unknowns of the original unknowns, i.e., density, velocity and temperature. Note the $\vartheta^{-1}$ at the right hand side of the transformed error system \eqref{BR} and \eqref{BR1} and $\hat\vartheta^{-1}(k)=\mathcal{O}(\epsilon^{-1})$ for $|k|<\delta$. Then many of the seemingly $\mathcal{O}(\epsilon)$ and $\mathcal{O}(\epsilon^2)$ terms are indeed $\mathcal{O}(1)$ and $\mathcal{O}(\epsilon)$ terms, which will be eliminated by the normal-form transform. An important difference between the isentropic and non-isentropic system lies in the fact that we can express the nonlinear terms in the isentropic Euler-Poisson system in terms of an $x$-derivative of an antiderivative and $k\hat\vartheta^{-1}(k)=\mathcal{O}(1)$ for $|k|<\delta$, and hence only the $\mathcal{O}(\epsilon)$ terms introduced by quadratic terms need to be eliminated. Cubic terms will not arise such a problem. But for the non-isentropic Euler-Poisson system, nonlinear terms do not share such a structure and hence both $\mathcal{O}(1)$ and $\mathcal{O}(\epsilon)$ terms will appear in the error system, coming from the quadratic and cubic terms.

Secondly, more resonances will appear in the non-isentropic case. The isentropic Euler-Poisson system has only trivial resonance at $k=0$ and non-trivial ones at $k=\pm k_0$, while for the non-isentropic system, extra resonances at $k=\pm 2k_0$ appear due to the temperature equation. We will see where these new resonances come from. We first note the absence of linear terms for the first equation in the diagonalized system \eqref{Diago}. When we eliminated all the low frequencies terms by the normal-form transform, we obtain a new error system \eqref{RJ}. Note that the $\mathcal{O}(\epsilon)$ terms in \eqref{RJ} only consist of the high frequency terms, and high frequency terms will lose one derivative when applying estimates. It is then difficult to obtain uniform estimates over  $\mathcal{O}(\epsilon^{-2})$ long timescales if we directly apply the Gr\"onwall inequality. We then try to use the corresponding normal-form transform to construct a modified energy \eqref{modiener}, which is found to be equivalent to the $H^s$-norm of the error $(\mathcal{R}^{0},\mathcal{R}^{1})$, thanks to Lemma \ref{L8} that reveals some structural identities of this normal-form transform. Then thanks to the structure of the quasilinear quadratic terms of the non-isentropic Euler-Poisson system, we have Lemma \ref{L10}, from which we see that the normal-form transform is well-defined and these resonances are trivial. Therefore, we need to get uniform estimate for the modified energy.


Thirdly, the uniform energy estimates is much more complex in the non-isentropic case. In particular, there are more complex interacting terms need estimating. In the isentropic case, the error system only involves unknowns $(\mathcal{R}^{0}_j,\mathcal{R}^{1}_j)$ for $j\in\{\pm1\}$; which in the non-isentropic case, it involves unknowns $(\mathcal{R}^{0}_j,\mathcal{R}^{1}_j)$ for $j\in\{0,\pm1\}$ and one of the unknowns has no linear terms in the error system (see \eqref{Diago} for the diagonalized system). We need uniform $\mathcal{O}(\epsilon^{2})$ estimates the interactions like
\[\epsilon^{2}\sum_{j\in\{\pm1\}}\int a_{j}\partial_x^{l}\mathcal{R}_{j}\partial_x^{l+1}\mathcal{R}_{-j}dx\ (\text{type I})\ \ \ \  \text{and} \ \ \ \  \epsilon^{2}\sum_{j\in\{\pm1\}}\int b_{j}\partial_x^{l}\mathcal{R}_{0}\partial_x^{l+1}\mathcal{R}_{j}dx\ (\text{type II}),\]
for $a_j,b_j\in H^2$ and $\mathcal{R}_{j}\in H^{l}$. Only the first type appear in the isentropic case. Indeed, in this case, according to Lemma \ref{L9}, the first type can be bounded by terms like
\begin{equation*}
\epsilon^{2}\int (a_{-1}+a_{1})\partial_x^{l}(\mathcal{R}_{1}+\mathcal{R}_{-1}) \partial_x^{l+1}(\mathcal{R}_{1}-\mathcal{R}_{-1})dx + \mathcal{O}(\epsilon^{2})
\end{equation*}
which can be finally estimated by further modifying the energy and integration by parts, thanks to the structure
\begin{equation}\label{struct}
\partial_x(\mathcal{R}_{1}-\mathcal{R}_{-1}) =\partial_t(\mathcal{R}_{1}+\mathcal{R}_{-1}) - \partial_x(\mathcal{R}_{1}+\mathcal{R}_{-1})  + \mathcal{O}(\epsilon^{2}).
\end{equation}
However, in the non-isentropic case, some nonlinear terms in the equation for $\mathcal{R}_{1}$ and $\mathcal{R}_{-1}$ interact with $\mathcal{R}_{0}$ and hence we cannot obtain structure identities like \eqref{struct}. Fortunately, according to the asymptotic behaviors of the dispersive relation in Lemma \ref{qq2}, the terms with derivatives landing on $\mathcal{R}_{0}$ cancel and the terms with derivatives not landing on $\mathcal{R}_{0}$ can be well bounded. Therefore, the type I terms are well bounded. Now we consider the type II terms. According to \eqref{part20} in Lemma \ref{L9}, we need to estimate terms like
\begin{equation}\label{typeII}
\epsilon^{2}\int (b_{1}-b_{-1})\partial_x^{l}\mathcal{R}_{0} \partial_x^{l+1}(\mathcal{R}_{1}-\mathcal{R}_{-1})dx \ \  \text{and} \ \ \epsilon^{2}\int (b_{1}+b_{-1})\partial_x^{l}\mathcal{R}_{0} \partial_x^{l+1}(\mathcal{R}_{1}+\mathcal{R}_{-1})dx.
\end{equation}
The first one can be rewritten as
\begin{align}\label{struct-2}
\epsilon^{2}\int \partial_x^{l}\mathcal{R}_{0}[\cdots]dx =  \epsilon^{2}\int \partial_x^{l}(\mathcal{R}_{0}+\mathcal{R}_{1}+\mathcal{R}_{-1}) [\cdots]dx -\epsilon^{2}\int \partial_x^{l}(\mathcal{R}_{1}+\mathcal{R}_{-1}) [\cdots] dx.
\end{align}
Note that the second integral belongs to type I, and can be well bounded. For the first integral, we note that after careful computations, we have Lemma \ref{RR}, which reveals important structure identities between the unknowns $\mathcal{R}_{j}$ for $j\in\{0,\pm1\}$. Then the first integral can be rewritten as
\begin{align*}
& \epsilon^{2}\int (b_{1}-b_{-1}) \partial_x^{l}(\mathcal{R}_{0}+\mathcal{R}_{1}+\mathcal{R}_{-1}) \partial_x^{l+1}(\mathcal{R}_{1}-\mathcal{R}_{-1})dx \\
= & \frac{\epsilon^{2}}{2}\frac{d}{dt}\int (b_{1}-b_{-1}) \left(\partial_x^{l}(\mathcal{R}_{0}+\mathcal{R}_{1}+\mathcal{R}_{-1})\right)^2 dx \\
 & -{\epsilon^{2}}\int \partial_x(b_{1}-b_{-1}) \left(\partial_x^{l}(\mathcal{R}_{0}+\mathcal{R}_{1} +\mathcal{R}_{-1})\right)^2 dx +\mathcal{O}(\epsilon^{2}) \\
= &  \frac{\epsilon^{2}}{2}\frac{d}{dt}\int (b_{1}-b_{-1}) \left(\partial_x^{l}(\mathcal{R}_{0}+\mathcal{R}_{1}+\mathcal{R}_{-1})\right)^2 dx +\mathcal{O}(\epsilon^{2})
\end{align*}
and can be well bounded. The second one in \eqref{typeII} can also be well bounded, thanks to the evolutionary equation for $\mathcal{R}_{0}$ and \eqref{1++}. In summary, all these terms can be well bounded by suitably modifying the energy.

\textbf{Notations.} We denote the Fourier transform of a function $u\in L^{2}(\mathbb{R},\mathbb{K})$, with $\mathbb{K}=\mathbb{R}$ or $\mathbb{K}=\mathbb{C}$ by
\begin{equation*}
\begin{split}
\widehat{u}(k)=\frac{1}{2\pi}\int_{\mathbb{R}}u(x)e^{-ikx}dx.
\end{split}
\end{equation*}
Let $H^{s}(\mathbb{R},\mathbb{K})$ be the space of functions mapping from $\mathbb{R}$ into $\mathbb{K}$ for which the norm
\begin{equation*}
\begin{split}
\|u\|_{H^{s}(\mathbb{R},\mathbb{K})}= \left(\int_{\mathbb{R}}|\widehat{u}(k)|^{2}(1+|k|^{2})^{s}dk\right)^{1/2}
\end{split}
\end{equation*}
is finite. We also write $L^{2}$ and $H^{s}$ instead of $L^{2}(\mathbb{R},\mathbb{R})$ and $H^{s}(\mathbb{R},\mathbb{R})$. Moreover, we use the space $L^{p}(m)(\mathbb{R},\mathbb{K})$ defined by $u\in L^{p}(m)(\mathbb{R},\mathbb{K})$ such that $\sigma^{m}u\in L^{p}(\mathbb{R},\mathbb{K})$, where $\sigma(x)=(1+x^{2})^{1/2}$. Furthermore, we write $A\lesssim B$ if $A\leq CB$ for a constant $C>0$, and $A=\mathcal{O}(B)$ if $A\lesssim B$ and $B\lesssim A$.

\textbf{Organizations.} This paper is organized as follows: In Section 2, we derive formally the NLS equation. In Section 3, we modify the approximation solution by adding high order terms and applying a cutoff function, and then give estimates for the residual. In Section 4, we derive a modified evolutionary system for the error $R$ by rescaling the error with an additional power of $\epsilon$ to solve the problem induced by the non-trivial resonances at $k=\pm k_{0}$. In Section 5, we apply the project operators to the error system to extract the low frequency and high frequency terms of order $\mathcal{O}(1)$ and $\mathcal{O}(\epsilon)$, and then apply normal-form transforms \emph{twice} to eliminate the low frequency terms and obtain the transformed system for the modified error $\mathcal{R}$. In Section 6, we analyse the structure and properties for the normal-form transform, to eliminate the high frequency terms of $\mathcal{O}(\epsilon)$, and then construct a modified energy by utilizing these transforms to deal with the difficulties induced by the loss of derivatives.

\section{\textbf{Formal NLS approximation}}
In this section, we present a short and formal derivation of the NLS from the Euler-Poisson system \eqref{EP}. We rewrite \eqref{EP} in terms of $(\rho,v,\theta)=(n-1,u,T-1)$ and isolate the linear, quadratic and higher order terms to obtain
\begin{subequations}\label{normalEP}
\begin{numcases}{}
\partial_{t}\rho+\partial_{x}v+\partial_{x}(\rho v)=0,\\
\partial_{t}v+\partial_{x}\rho+\partial_{x}\theta+\partial_{x}\phi+v\partial_{x} v-\partial_{x}\frac{\rho^{2}}{2}+\theta\partial_{x}\rho\\
 \ \ \ \ \ \ \  =-\partial_{x}[\ln(1+\rho)-\rho+\frac{\rho^{2}}{2}](1+\theta)
+\theta\partial_{x}\frac{\rho^{2}}{2},\\
\partial_{t}\theta+(\gamma-1)\partial_{x}v+(\gamma-1)\theta\partial_{x}v+v\partial_{x}\theta=0,\\
\rho=(1-\partial_{x}^{2})\phi+\frac{\phi^{2}}{2}+[e^{\phi}-1-\phi-\frac{\phi^{2}}{2}].
\end{numcases}
\end{subequations}
For small $\rho$, the last line defines an inverse operator $\rho\mapsto\phi(\rho)$, which is further expanded into the linear, quadratic, cubic and higher order terms
\begin{equation}
\begin{split}\label{Poisson}
\phi(\rho)=&(1-\partial_{x}^{2})^{-1}\rho-\frac{1}{2}(1-\partial_{x}^{2})^{-1}
[(1-\partial_{x}^{2})^{-1}\rho]^{2}\\
&-\frac{1}{3!}(2+\partial_{x}^{2})(1-\partial_{x}^{2})^{-2}
\big[(1-\partial_{x}^{2})^{-1}\rho\big]^{3}+\mathcal{M}(\rho),
\end{split}
\end{equation}
where $\mathcal{M}$ satisfies some good properties \cite{GP11}. In this way, we can rewrite the system \eqref{normalEP} as
\begin{equation}
\begin{split}\label{diag}
\partial_{t}&\begin{pmatrix}\rho\\ v\\ \theta \end{pmatrix}+\begin{pmatrix}0&\partial_{x}&0\\ \partial_{x}(1-\partial_{x}^{2})^{-1}+\partial_{x}&0&\partial_{x}\\0&(\gamma-1)\partial_{x}&0
\end{pmatrix}\begin{pmatrix}\rho\\ v\\ \theta\end{pmatrix}\\
=&\begin{pmatrix}-\partial_{x}(\rho v) \\-\partial_{x}\frac{v^{2}}{2}+\partial_{x}\frac{\rho^{2}}{2}
-\theta\partial_{x}\rho
+\frac{1}{2}\partial_{x}(1-\partial_{x}^{2})^{-1}[(1-\partial_{x}^{2})^{-1}\rho]^{2}
\\ -v\partial_{x}\theta-(\gamma-1)\theta\partial_{x}v\end{pmatrix}+\begin{pmatrix} 0 \\ \mathcal{H}(\rho,\theta) \\ 0 \end{pmatrix},
\end{split}
\end{equation}
where $\mathcal{H}(\rho,\theta)$ is at least cubic
\begin{align*} \mathcal{H}(\rho,\theta)=&\frac{1}{3!}\partial_{x}(2+\partial_{x}^{2})(1-\partial_{x}^{2})^{-2}
\big[(1-\partial_{x}^{2})^{-1}\rho\big]^{3}
+\partial_{x}(\frac{\rho^{3}}{3})-\theta\partial_{x}\frac{\rho^{2}}{2}\\
&+\left(-\partial_{x}(\ln(1+\rho)-\rho+\frac{\rho^{2}}{2}-\frac{\rho^{3}}{3})(1+\theta)
+\theta\partial_{x}\frac{\rho^{3}}{3}\right).
\end{align*}
Let $$S=\begin{pmatrix}1&1&1\\0&-q(|\partial_{x}|)& q(|\partial_{x}|)\\(\gamma-1)-q^{2}(|\partial_{x}|)&\gamma-1&\gamma-1 \end{pmatrix}\ \ \  \text{and}\ \ \ \begin{pmatrix}\rho\\ v\\\theta\end{pmatrix}=S\begin{pmatrix}U_{0}\\ U_{1}\\U_{-1}\end{pmatrix}.$$
We can diagonalize the linear part of the equation \eqref{diag}, to obtain
\begin{equation}
\begin{split}\label{Diago}
\partial_{t}U_{0}&=Q_{0}(U,U),\\
\partial_{t}U_{j}&=j\Omega U_{j}+Q_{j}(U,U)+N_{j}(U),
\end{split}
\end{equation}
where $U=(U_{0},U_{1},U_{-1})^{T}$, $j\in\{1,-1\}$ and $\widehat{\Omega}(k)=\omega(k)=k\widehat{q}(k)$ with the linear dispersive relation $ \omega(k)$ in \eqref{equation3}. For later purpose, we also write \eqref{Diago} in the following short form
\begin{equation}
\begin{split}\label{short}
\partial_{t}U=\Lambda U +Q (U,U)+N(U).
\end{split}
\end{equation}
After careful calculation, the quadratic term $Q_{j}$ and the high order term $N_{j}$ take the form
\begin{align}\label{qua}
Q_{0}(U,U)
=&\frac{1}{q^{2}}\left(q(U_{1}-U_{-1})\partial_{x}q^{2}U_{0}\right)\nonumber\\
&-\frac{(\gamma-1)}{q^{2}}\Big(((\gamma-2-q^{2})U_{0}+(\gamma-2)U_{1}
+(\gamma-2)U_{-1}))\partial_{x}q(U_{1}-U_{-1})\Big),\nonumber\\
Q_{j}(U,U)
=&-\frac{1}{2q^{2}}\left(q(U_{1}-U_{-1})\partial_{x}q^{2}U_{0}\right)
+\frac{1}{2}\Big(q(U_{1}-U_{-1})\partial_{x}(U_{0}+U_{1}+U_{-1})\Big)\nonumber\\
&+\frac{1}{2}\Big((U_{0}+U_{1}+U_{-1})\partial_{x}q(U_{1}-U_{-1})\Big)\nonumber\\
&+\frac{\gamma-1}{2q^{2}}\bigg(\Big((\gamma-2-q^{2})U_{0}+(\gamma-2)U_{1}
+(\gamma-2)U_{-1}\Big)\partial_{x} q(U_{1}-U_{-1})\bigg)\nonumber\\
&+\frac{j\partial_{x}}{4q}\big(q(U_{1}-U_{-1})\big)^{2}\nonumber\\
&+\frac{j}{2q}\left(\Big((\gamma-2-q^{2})U_{0}+(\gamma-2)U_{1}
+(\gamma-2)U_{-1}\Big)\partial_{x}(U_{0}+U_{1}+U_{-1})\right)\\
&-\frac{j\partial_{x}}{4q(1-\partial_{x}^{2})}
\Big(\frac{U_{0}+U_{1}+U_{-1}}{1-\partial_{x}^{2}}\Big)^{2},\nonumber\\
N_{j}(U)=&-\frac{j}{12q}\frac{\partial_{x}(2+\partial^{2}_{x})}{(1-\partial^{2}_{x})^{2}}
\left(\frac{U_{0}+U_{1}+U_{-1}}{1-\partial^{2}_{x}}\right)^{3}
-\frac{j}{6q}\partial_{x}(U_{0}+U_{1}+U_{-1})^{3}\nonumber\\
&+\frac{j}{4q}\left((\gamma-1-q^{2})U_{0}+(\gamma-1)U_{1}+(\gamma-1)U_{-1}\right)
\partial_{x}(U_{0}+U_{1}+U_{-1})^{2}+\widetilde{\mathcal{H}}(U^{4})\nonumber.
\end{align}
In order to derive the NLS equation as an approximation equation for system \eqref{Diago} with \eqref{qua}, we make the ansatz
\begin{equation}
\begin{split}\label{app}
\begin{pmatrix} U_{0}\\ U_{1} \\ U_{-1} \end{pmatrix}=\epsilon\widetilde{\Psi}
=\epsilon\widetilde{\Psi}_{1}+\epsilon\widetilde{\Psi}_{-1}+\epsilon^{2}\widetilde{\Psi}_{0}
+\epsilon^{2}\widetilde{\Psi}_{2}+\epsilon^{2}\widetilde{\Psi}_{-2},
\end{split}
\end{equation}
with
\begin{equation*}
\begin{split}
&\epsilon\widetilde{\Psi}_{\pm1}=
\epsilon\widetilde{A}_{\pm1}(\epsilon(x-c_{g}t),\epsilon^{2}t)E^{\pm1}\begin{pmatrix} 0 \\ 0 \\ 1 \end{pmatrix},\\
&\epsilon^{2}\widetilde{\Psi}_{0}=\begin{pmatrix}
\epsilon^{2}\widetilde{A}_{00}(\epsilon(x-c_{g}t),\epsilon^{2}t) \\ \epsilon^{2}\widetilde{A}_{01}(\epsilon(x-c_{g}t),\epsilon^{2}t)\\ \epsilon^{2}\widetilde{A}_{0-1}(\epsilon(x-c_{g}t,\epsilon^{2}t)\end{pmatrix},\\
&\epsilon^{2}\widetilde{\Psi}_{\pm2}=\begin{pmatrix}
\epsilon^{2}\widetilde{A}_{(\pm2)0}(\epsilon(x-c_{g}t),\epsilon^{2}t)E^{\pm2} \\ \epsilon^{2}\widetilde{A}_{(\pm2)1}(\epsilon(x-c_{g}t),\epsilon^{2}t)E^{\pm2} \\ \epsilon^{2}\widetilde{A}_{(\pm2)-1}(\epsilon(x-c_{g}t),\epsilon^{2}t)E^{\pm2}\end{pmatrix},
\end{split}
\end{equation*}
where $E^{j}=e^{ij(k_{0}x-\omega_{0}t)}$, $\omega_{0}=\omega(k_{0})$, $\widetilde{A}_{-1}=\overline{\widetilde{A}}_{1}$ and $\widetilde{A}_{-j\ell}=\overline{\widetilde{A}}_{j\ell}$ with $\ell\in\{0, \ \pm1\}$. Inserting \eqref{app} into \eqref{Diago} and replacing $\omega(k)$ with their Taylor expansions around $k=j k_{0}$ in all terms of $\Omega \widetilde{A}_{j}E^{j}$ and $\Omega \widetilde{A}_{j\ell}E^{j}$, then we balance the coefficients of the $\epsilon^{m}E^{j}$ in the following. Similar expansions can be found in Lemma 25 of \cite{SW11}, for the quasilinear water wave model, for example.

Due to the definition of $\omega_{0}$ and $c_{g}=\omega'(k_{0})$, the coefficients of $\epsilon E^{1}$ and $\epsilon^{2} E^{1}$ vanish identically. In the same way, all terms of $\epsilon^{2} E^{0}$ also vanish identically in the light of the linear terms.

For $\epsilon^{2} E^{2}$, we have
\begin{align}\label{E2}
-2\omega_{0}\widetilde{A}_{20}&=\gamma_{20}(\widetilde{A}_{1})^{2},\nonumber\\
(-2\omega_{0}-\omega(2k_{0})\widetilde{A}_{21}&=\gamma_{21}(\widetilde{A}_{1})^{2},\\
(-2\omega_{0}+\omega(2k_{0})\widetilde{A}_{2-1}&=\gamma_{2-1}(\widetilde{A}_{1})^{2},\nonumber
\end{align}
where
\begin{align*}
\gamma_{20}&=\frac{-2k_{0}\widehat{q}(k_{0})}
{9\widehat{q}^{2}(2k_{0})},\\
\gamma_{2j}&=\frac{k_{0}\widehat{q}(k_{0})}{9\widehat{q}^{2}(2k_{0})}
+\frac{jk_{0}\widehat{q}^{2}(k_{0})}{2\widehat{q}(2k_{0})}
-\frac{k_{0}\widehat{q}^{3}(k_{0})}{\widehat{q}^{2}(2k_{0})}
-\frac{jk_{0}}{6\widehat{q}(k_{0})}
-\frac{jk_{0}}{2\widehat{q}(2k_{0})(1+4k_{0}^{2})(1+k_{0}^{2})^{2}},
\end{align*}
with $j\in\{\pm1\}$. Since $\omega_{0}\neq0$ and $-2\omega_{0}\pm\omega(2k_{0})\neq0$, the coefficients $\widetilde{A}_{20}$ and $\widetilde{A}_{2\pm1}$ are determined in terms of $(\widetilde{A}_{1})^{2}$.

For $\epsilon^{3} E^{0}$, we have
\begin{align}\label{E0}
-c_{g}\partial_{X}\widetilde{A}_{00}&=\gamma_{00}\partial_{X}(|\widetilde{A}_{1}|^{2}),\nonumber\\
-c_{g}\partial_{X}\widetilde{A}_{01}&=-\omega'(0)\partial_{X}\widetilde{A}_{01}
+\gamma_{01}\partial_{X}(|\widetilde{A}_{1}|^{2}),\\
-c_{g}\partial_{X}\widetilde{A}_{0j}&=\omega'(0)\partial_{X}\widetilde{A}_{0-1}
+\gamma_{0-1}\partial_{X}(|\widetilde{A}_{1}|^{2}),\nonumber
\end{align}
where $\gamma_{0\ell}\in\mathbb{R}$ with $\ell\in\{0, \ \pm1\}$. Since $c_{g}=\omega'(k_{0})$, we can express $\widetilde{A}_{0\ell}$ in terms of $|\widetilde{A}_{1}|^{2}$ by integrating the equations for $\epsilon^{3} E^{0}$ w.r.t. ${X}$.

For $\epsilon^{3} E^{1}$, we have
\begin{align*}
\partial_{T}\widetilde{A}_{1}=\frac{i\omega''(k_{0})}{2}\partial_{X}^{2}\widetilde{A}_{1}
+g_{1},
\end{align*}
where $g_{1}$ is a sum of multiples of $\widetilde{A}_{1}\big|\widetilde{A}_{1}\big|^{2}$, $\widetilde{A}_{1}\widetilde{A}_{0\ell}$ and $\widetilde{A}_{-1}\widetilde{A}_{2\ell}$. Eliminating $\widetilde{A}_{0\ell}$ and $\widetilde{A}_{2\ell}$ by the algebraic relations \eqref{E2} and \eqref{E0} obtained for $\epsilon^{2} E^{2}$ and $\epsilon^{3} E^{0}$, we obtain the final NLS equation
\begin{align}\label{NLS}
\partial_{T}\widetilde{A}_{1}=\frac{i\omega''(k_{0})}{2}\partial_{X}^{2}\widetilde{A}_{1}
+i\nu_{2}(k_{0})\widetilde{A}_{1}|\widetilde{A}_{1}|^{2},
\end{align}
with $\nu_{2}(k_{0})\in\mathbb{R}$.

\section{\textbf{Modified approximation and estimates for the residual}}
It is not easy to justify the approximation of the NLS equation \eqref{NLS} for the Euler-Poisson system \eqref{EP} directly. To make it easier to justify such approximation, we make two modifications of the approximation:
\begin{description}
  \item[Step 1] Extend $\epsilon\widetilde{\Psi}$ to $\epsilon\widetilde{\Psi}^{ext}$, by adding some higher order terms, and
  \item[Step 2] Modify $\epsilon\widetilde{\Psi}^{ext}$ to the final approximation $\epsilon\Psi$, by some cut-off function such that the support of $\epsilon\Psi$ in Fourier space is restricted to small neighborhoods of integer multiples of the basic wave number $k_{0}$.
\end{description}
These modifications make the approximation $\epsilon\Psi$ an analytic function and ensure that the residual (see \eqref{short})
\begin{equation}\label{res-pu}
Res(\varepsilon\Psi):=-\varepsilon\partial_{t}\Psi + \varepsilon\Lambda \Psi +Q (\varepsilon\Psi,\varepsilon\Psi)+N(\varepsilon\Psi)
\end{equation}
is small enough.

In the first step, we add higher order terms $\epsilon^{2}\widetilde{\Psi}^{add}$ to the formal approximation $\epsilon\widetilde{\Psi}$ in \eqref{app} and denote $\epsilon\widetilde{\Psi}^{ext}$ as follows
\begin{equation}
\begin{split}\label{ext}
\epsilon\widetilde{\Psi}^{ext}
:=\epsilon\widetilde{\Psi}
+\epsilon^{2}\widetilde{\Psi}^{add},
\end{split}
\end{equation}
where 
$\epsilon^{2}\widetilde{\Psi}^{add}$ is defined as
\begin{align*}
\epsilon^{2}\widetilde{\Psi}^{add}
=&\sum_{j=\pm1}\sum_{n=1,2,3,4}\epsilon^{1+n}\begin{pmatrix}
\widetilde{A}^{n}_{j0}(\epsilon(x-c_{g}t),\epsilon^{2}t) \\ \widetilde{A}^{n}_{j1}(\epsilon(x-c_{g}t),\epsilon^{2}t) \\ \widetilde{A}^{n}_{j-1}(\epsilon(x-c_{g}t),\epsilon^{2}t)\end{pmatrix}E^{j}\\
&+\sum_{j=\pm2}\sum_{n=1,2,3}\epsilon^{2+n}\begin{pmatrix}
\widetilde{A}^{n}_{j0}(\epsilon(x-c_{g}t),\epsilon^{2}t) \\ \widetilde{A}^{n}_{j1}(\epsilon(x-c_{g}t),\epsilon^{2}t) \\ \widetilde{A}^{n}_{j-1}(\epsilon(x-c_{g}t),\epsilon^{2}t)\end{pmatrix}E^{j}\\
&+\sum_{n=1,2,3}\epsilon^{2+n}\begin{pmatrix}
\widetilde{A}^{n}_{00}(\epsilon(x-c_{g}t),\epsilon^{2}t) \\ \widetilde{A}^{n}_{01}(\epsilon(x-c_{g}t),\epsilon^{2}t) \\ \widetilde{A}^{n}_{0-1}(\epsilon(x-c_{g}t),\epsilon^{2}t)\end{pmatrix}\\
&+\sum_{j=\pm3}\sum_{n=0,1,2}\epsilon^{3+n}\begin{pmatrix}
\widetilde{A}^{n}_{j0}(\epsilon(x-c_{g}t),\epsilon^{2}t) \\ \widetilde{A}^{n}_{j1}(\epsilon(x-c_{g}t),\epsilon^{2}t) \\ \widetilde{A}^{n}_{j-1}(\epsilon(x-c_{g}t),\epsilon^{2}t)\end{pmatrix}E^{j}\\
&+\sum_{j=\pm4}\sum_{n=0,1}\epsilon^{4+n}\begin{pmatrix}
\widetilde{A}^{n}_{j0}(\epsilon(x-c_{g}t),\epsilon^{2}t) \\ \widetilde{A}^{n}_{j1}(\epsilon(x-c_{g}t),\epsilon^{2}t) \\ \widetilde{A}^{n}_{j-1}(\epsilon(x-c_{g}t),\epsilon^{2}t)\end{pmatrix}E^{j}\\
&+\sum_{j=\pm5}\epsilon^{5}\begin{pmatrix}
\widetilde{A}^{0}_{j0}(\epsilon(x-c_{g}t),\epsilon^{2}t) \\ \widetilde{A}^{0}_{j1}(\epsilon(x-c_{g}t),\epsilon^{2}t) \\ \widetilde{A}^{0}_{j-1}(\epsilon(x-c_{g}t),\epsilon^{2}t)\end{pmatrix}E^{j},
\end{align*}
with $\widetilde{A}^{n}_{-j\ell}=\overline{\widetilde{A}^{n}_{j\ell}}$. The amplitudes $\widetilde{A}_{j\ell}$ and $\widetilde{A}_{\pm1}$ from the last section are regarded as $\widetilde{A}^{0}_{j\ell}$ and $\widetilde{A}^{0}_{\pm1-1}$ hereafter. Inserting \eqref{ext} into system \eqref{Diago}, we find that the residual is formally at least of order $\mathcal{O}(\epsilon^{6})$ if $\widetilde{A}^{n}_{j\ell}$ is chosen in a suitable way. The process of choosing $\widetilde{A}^{n}_{j\ell}$ is similar to the derivation of the formulas for $\epsilon\widetilde{\Psi}_{\pm1}$, $\epsilon\widetilde{\Psi}_{\pm2}$ and $\epsilon\widetilde{\Psi}_{0}$ in the above section. Specifically, $\widetilde{A}^{n}_{j\ell}$ is determined by letting the coefficients of $\epsilon^{m}E^{j}$ to be zero. Since we focus on real approximate solutions for the non-isentropic Euler-Poisson system \eqref{EP}, it is enough to determine all amplitudes $\widetilde{A}^{n}_{j\ell}$ with $j\geq0$.

For $\epsilon^{j+n}E^{j}$ with $j\in\{2,3,4,5\}$, we can obtain the equations for $\widetilde{A}^{n}_{j\ell}$ as follows
\begin{align}\label{Anj}
-j\omega_{0}\widetilde{A}^{n}_{j0}&=h^{n}_{j0},\nonumber\\
(-j\omega_{0}+\omega(jk_{0}))\widetilde{A}^{n}_{j1}&=h^{n}_{j1},\\
(-j\omega_{0}-\omega(jk_{0}))\widetilde{A}^{n}_{j-1}&=h^{n}_{j-1},\nonumber
\end{align}
where $h^{n}_{j\ell}$ with $\ell\in\{0, \ \pm1\}$ depends polynomially on $\widetilde{A}^{n'}_{j'\ell'}$ with $n'\leq n$. In particular, if $n'=n$ then $|j'|<j$. Since $j\omega_{0}\neq0$ and $-j\omega_{0}\pm\omega(jk_{0})\neq0$ for $j\in\{2,3,4,5\}$, each $\widetilde{A}^{n}_{j\ell}$ on the LHS can be uniquely expressed in terms of $\widetilde{A}^{n'}_{j'\ell'}$ on the RHS.


We next consider the equations for $j=0$ and obtain the equations for $\widetilde{A}^{n}_{0\ell}$ as follows
\begin{align}\label{An0}
-c_{g}\partial_{X}\widetilde{A}^{n}_{00}&=h^{n}_{00},\nonumber\\
-c_{g}\partial_{X}\widetilde{A}^{n}_{01}&=-\omega'(0)\partial_{X}\widetilde{A}^{n}_{01}+h^{n}_{01},\\
-c_{g}\partial_{X}\widetilde{A}^{n}_{0-1}&=\omega'(0)\partial_{X}\widetilde{A}^{n}_{0-1}+h^{n}_{0-1},
\nonumber
\end{align}
where $h^{n}_{0\ell}$ depend polynomially on $\widetilde{A}^{n'}_{j'\ell'}$ with $n'\leq n$. In particular, if $n'=n$ then $j'=\pm1$. Similar to \eqref{E0}, all $h^{n}_{0\ell}$ can be written as an $X$-derivative of an expression containing $\widetilde{A}^{n'}_{j'\ell'}$. 
Since $c_{g}\neq0$ and $c_{g}\neq \pm\widetilde{A}^{n}_{0\ell}$, each $\widetilde{A}^{n}_{0\ell}$ can be determined in terms of $\widetilde{A}^{n'}_{j'\ell'}$ by a straightforward integration.


Now we only need to consider the equations for $\widetilde{A}^{n}_{1\ell}$. In the last section, we deduced that $\widetilde{A}^{0}_{10}=\widetilde{A}^{0}_{11}=0$ and $\widetilde{A}^{0}_{1-1}=\widetilde{A}_{1}$ satisfy the NLS equation \eqref{NLS}. Furthermore, the amplitudes $\widetilde{A}^{n}_{1\ell}$ with $n\in\{1, \ 2, \ 3, \ 4\}$ and $\ell\in\{0, \ 1\}$ satisfy the equations
\begin{align}\label{A1n}
-\omega_{0}\widetilde{A}^{n}_{10}&=h^{n}_{10},\nonumber\\
(-\omega_{0}-\omega(k_{0}))\widetilde{A}^{n}_{11}&=h^{n}_{11},
\end{align}
where $h^{n}_{1\ell}$ depend polynomially on $\widetilde{A}^{n'}_{j'\ell'}$ 
with $n'<n$. Since $\omega_{0}\neq0$ and $-\omega_{0}-\omega(k_{0})\neq0$, $\widetilde{A}^{n}_{10}$ and $\widetilde{A}^{n}_{11}$ are determined in terms of $\widetilde{A}^{n'}_{j'\ell'}$. Moreover, we have $h^{1}_{10}=h^{1}_{11}=0$ since there are no terms of order $\mathcal{O}(\epsilon^{2})$ at the wave numbers $k=\pm k_{0}$. Therefore, $\widetilde{A}^{1}_{10}=\widetilde{A}^{1}_{11}=0$.

Due to \eqref{A1n} and the expressions for $\widetilde{A}^{n}_{j\ell}$ with $j\in\{0, 2, 3, 4, 5\}$,  the amplitudes $\widetilde{A}^{n}_{j\ell}$ for all $j, \ell$ and $n$ are uniquely determined if the amplitudes $\widetilde{A}^{n'}_{1-1}$ with $n'\leq n$ are determined. Thus, it remains to turn to the equations for $\widetilde{A}^{n}_{1-1}$ with $n\in\{1, \ 2, \ 3, \ 4\}$. Note that we have
\begin{align}\label{ANLS}
\partial_{T}\widetilde{A}^{n}_{1-1}=i\frac{\omega''(k_{0})}{2}\partial^{2}_{X}\widetilde{A}^{n}_{1-1}
+h^{n}_{1-1},
\end{align}
where $h^{n}_{1-1}$ are affine in $\widetilde{A}^{n}_{1-1}$ and depend polynomially on $\widetilde{A}^{n'}_{j'\ell'}$ 
with $n'<n$. According to the form of $g^{n}_{1-1}$, $\widetilde{A}^{n}_{1-1}$ for $n\in\{1, \ 2, \ 3, \ 4\}$ satisfy inhomogeneous but linear Schr\"odinger equations. Therefore it is not difficult to establish the existence of solutions of these equations in $T\in[0, \ T_{0}]$ with $T_{0}=\mathcal{O}(1)$ by using the variation of constants formula and Gronwall's lemma whenever $\widetilde{A}_{1}$ is the solution of the NLS equation \eqref{NLS} for $T\in[0, \ T_{0}]$. Finally, since the amplitude $\widetilde{A}^{4}_{1-1}$ does not appear in the equation for any other amplitudes $\widetilde{A}^{n}_{j\ell}$, we can set
\begin{align}\label{An4}
\widetilde{A}^{4}_{1-1}=0.
\end{align}

Till now all amplitudes $\widetilde{A}^{n}_{j\ell}$ hae been determined. For the regularity of the $\widetilde{A}^{n}_{j\ell}$, we note that derivatives of the right-hand-side of \eqref{Anj}-\eqref{ANLS} may come either from the derivatives in the nonlinear terms of \eqref{Diago} or from the dispersion relation $\omega(k)$. However, each such derivative provides one extra power of $\epsilon$, thanks to the scaling on time-space of $\widetilde{A}^{n}_{j\ell}$. As a result, the maximum number of derivatives that can occur in the equations for $\widetilde{A}^{n}_{j\ell}$ is $n$ if $(j, \ \ell)\neq(\pm1, \ -1)$ and $n+2$ if $(j, \ \ell)=(\pm1, \ -1)$ with $n\in\{1, \ 2, \ 3\}$. Therefore, we have the following Lemma.

\begin{lemma}\label{L1}
Fix $s_{A}\geq6$. Let $\widetilde{A}_{1}\in C([0, \ T_{0}], \ H^{s_{A}})$ be a solution of the NLS equation \eqref{NLS}. Then $\widetilde{A}^{n}_{j\ell}$ determined by \eqref{Anj}-\eqref{An4} exists for all $T\in [0, \ T_{0}]$ and satisfies $\widetilde{A}^{n}_{j\ell}\in C([0, T_{0}], H^{s_{A}-n})$ if $(j, \ell)\neq(\pm1, -1)$ or $\widetilde{A}^{n}_{j\ell}\in C([0, T_{0}], H^{s_{A}-n-2})$ if $(j, \ell)=(\pm1, -1)$ with $n\in\{1, 2, 3\}$.
\end{lemma}

In the second step, we use a Fourier truncation procedure to modify the extended approximation $\epsilon\widetilde{\Psi}^{ext}$ into $\epsilon\Psi$. We introduce the characteristic function
\begin{align*}
\chi_{[-\delta,\delta]}(k)=\Big\{\begin{matrix} 1, \ |k|\leq\delta, \\ 0, \ |k|>\delta\end{matrix}
\end{align*}
and define
\begin{align}\label{SE}
A^{n}_{j\ell}:=\mathcal{F}^{-1}(\chi_{[-\delta,\delta]}\mathcal{F}\widetilde{A}^{n}_{j\ell})(x).
\end{align}
Then the amplitudes $A^{n}_{j\ell}$ have the compact support
\begin{align*}
\{k\in\mathbb{R}:|k-\ell k_{0}|<\delta\}
\end{align*}
in Fourier space for some $\delta>0$ sufficiently small, but independent of $0<\epsilon\ll1$.

Hence, we get our final approximation $\epsilon\Psi$ with
\begin{equation}
\begin{split}\label{modi}
\epsilon\Psi
:=\epsilon\Psi_{1}+\epsilon\Psi_{-1}+\epsilon^{2}\Psi_{p},
\end{split}
\end{equation}
where
\begin{align}\label{Psi}
\epsilon\Psi_{\pm1}=&\epsilon\psi_{\pm1}\begin{pmatrix}0\\0\\1\end{pmatrix}
=\epsilon A_{\pm1}(\epsilon(x-c_{g}t),\epsilon^{2}t)E^{\pm1}\begin{pmatrix} 0 \\ 0 \\ 1 \end{pmatrix},\nonumber\\
\epsilon^{2}\Psi_{p}=&\epsilon^{2}\begin{pmatrix}
\psi_{p_{0}} \\ \psi_{p_{1}}\\ \psi_{p_{-1}}\end{pmatrix}=
\epsilon^{2}\Psi^{1}_{1}
+\epsilon^{2}\Psi^{1}_{-1}+\epsilon^{2}\Psi_{0}+\epsilon^{2}\Psi_{2}
+\epsilon^{2}\Psi_{-2}+\epsilon^{3}\Psi_{h},\nonumber\\
\epsilon^{2}\Psi^{1}_{\pm1}=&\epsilon^{2}\psi^{1}_{\pm1}\begin{pmatrix}0\\0\\1\end{pmatrix}
=\epsilon A^{1}_{\pm1}(\epsilon(x-c_{g}t),\epsilon^{2}t)E^{\pm1}\begin{pmatrix} 0 \\ 0 \\ 1 \end{pmatrix},\nonumber\\
\epsilon^{2}\Psi_{0}
=&\epsilon^{2}\begin{pmatrix}
\psi_{00} \\ \psi_{01}\\ \psi_{0-1}\end{pmatrix}
=\epsilon^{2}\begin{pmatrix}
A_{00}(\epsilon(x-c_{g}t),\epsilon^{2}t)\\ A_{01}(\epsilon(x-c_{g}t),\epsilon^{2}t) \\ A_{0-1}(\epsilon(x-c_{g}t),\epsilon^{2}t)\end{pmatrix},\nonumber\\
\epsilon^{2}\Psi_{\pm2}
=&\epsilon^{2}\begin{pmatrix}
\psi_{(\pm2)0}\\ \psi_{(\pm2)1}\\ \psi_{(\pm2)-1}\end{pmatrix}
=\epsilon^{2}\begin{pmatrix}
A_{(\pm2)0}(\epsilon(x-c_{g}t),\epsilon^{2}t)\\ A_{(\pm2)1}(\epsilon(x-c_{g}t),\epsilon^{2}t)\\ A_{(\pm2)-1}(\epsilon(x-c_{g}t),\epsilon^{2}t)\end{pmatrix}E^{\pm2}, \\
\epsilon^{3}\Psi_{h}
=&\sum_{j=\pm1}\sum_{n=2,3,4}\epsilon^{1+n}\begin{pmatrix}
A^{n}_{j0}(\epsilon(x-c_{g}t),\epsilon^{2}t) \\ A^{n}_{j1}(\epsilon(x-c_{g}t),\epsilon^{2}t)\\ A^{n}_{j-1}(\epsilon(x-c_{g}t),\epsilon^{2}t)\end{pmatrix}E^{j}\nonumber\\
&+\sum_{j=\pm2}\sum_{n=1,2,3}\epsilon^{2+n}\begin{pmatrix}
A^{n}_{j0}(\epsilon(x-c_{g}t),\epsilon^{2}t) \\ A^{n}_{j1}(\epsilon(x-c_{g}t),\epsilon^{2}t) \\ A^{n}_{j-1}(\epsilon(x-c_{g}t),\epsilon^{2}t)\end{pmatrix}E^{j}\nonumber\\
&+\sum_{n=1,2,3}\epsilon^{2+n}\begin{pmatrix}
A^{n}_{00}(\epsilon(x-c_{g}t),\epsilon^{2}t) \\ A^{n}_{01}(\epsilon(x-c_{g}t),\epsilon^{2}t)\\ A^{n}_{0-1}(\epsilon(x-c_{g}t),\epsilon^{2}t)\end{pmatrix}\nonumber\\
&+\sum_{j=\pm3}\sum_{n=0,1,2}\epsilon^{3+n}\begin{pmatrix}
A^{n}_{j0}(\epsilon(x-c_{g}t),\epsilon^{2}t) \\ A^{n}_{j1}(\epsilon(x-c_{g}t),\epsilon^{2}t) \\ A^{n}_{j-1}(\epsilon(x-c_{g}t),\epsilon^{2}t)\end{pmatrix}E^{j}\nonumber\\
&+\sum_{j=\pm4}\sum_{n=0,1}\epsilon^{4+n}\begin{pmatrix}
A^{n}_{j0}(\epsilon(x-c_{g}t),\epsilon^{2}t) \\ A^{n}_{j1}(\epsilon(x-c_{g}t),\epsilon^{2}t)\nonumber \\ A^{n}_{j-1}(\epsilon(x-c_{g}t),\epsilon^{2}t)\end{pmatrix}E^{j}\nonumber\\
&+\sum_{j=\pm5}\epsilon^{5}\begin{pmatrix}
A^{0}_{j0}(\epsilon(x-c_{g}t),\epsilon^{2}t) \\ A^{0}_{j1}(\epsilon(x-c_{g}t),\epsilon^{2}t) \\ A^{0}_{j-1}(\epsilon(x-c_{g}t),\epsilon^{2}t)\end{pmatrix}E^{j}.\nonumber
\end{align}
Here and throughout the remainder of the paper, we will use upper case $\Psi$ to denote vector valued functions and lower case $\psi$ to denote scalar functions.

Since the Fourier transform of the functions in the extended approximation $\epsilon\widetilde{\Psi}^{ext}$ are concentrated around the wave numbers $\ell k_{0}$ if $\widetilde{A}^{n}_{j\ell}$ is sufficiently regular, $\epsilon\widetilde{\Psi}^{ext}$ only changes slightly by the Fourier truncation procedure. This fact is a consequence of the estimate
\begin{align}\label{1/2}
\|\chi_{[-\delta,\delta]\epsilon^{-1}}\widehat{f}(\epsilon^{-1}\cdot)\|_{L^{2}(m)}
\leq C(\delta)\epsilon^{m+M-1/2}\|f\|_{H^{m+M}}
\end{align}
for all $m, M\geq0$. The Fourier truncation procedure makes the final approximation $\epsilon\Psi$ an analytic function and the estimate much simpler for the error. For more related strategies refer to \cite{SW11, DSW16}.

\begin{lemma}\label{L2}
Let $s_{N}\geq6$ and $\tilde{A}_{1}\in C([0,T_{0}], \ H^{s_{N}}(\mathbb{R},\mathbb{C}))$ be a solution of the NLS equations \eqref{NLS} with
\begin{equation*}
\begin{split}
\sup_{T\in[0,T_{0}]}\|\tilde{A}_{1}\|_{H^{s_{N}}}\leq C_{A}.
\end{split}
\end{equation*}
Then for all $s\geq0$ there exist $C_{Res}, C_{\Psi}$ and $\epsilon_{0}>0$ depending on $C_{A}$ such that for all $\epsilon\in(0,\epsilon_{0})$ the corresponding approximation $\epsilon\Psi$ satisfies
\begin{equation}
\begin{split}\label{Aesti-1}
\sup_{t\in[0,T_{0}/\epsilon^{2}]}\|Res(\epsilon\Psi)\|_{H^{s}}\leq C_{Res} \epsilon^{11/2},
\end{split}
\end{equation}
\begin{equation}
\begin{split}\label{Aesti-2}
\sup_{t\in[0,T_{0}/\epsilon^{2}]}\|\epsilon\Psi-\epsilon\widetilde{\Psi}_{1}-\epsilon\widetilde{\Psi}_{-1}
\|_{H^{s_{N}}}\leq C_{\Psi}\epsilon^{3/2},
\end{split}
\end{equation}
\begin{equation}
\begin{split}\label{Aesti-3}
\sup_{t\in[0,T_{0}/\epsilon^{2}]}(\|\widehat{\Psi}_{\pm1}\|_{L^{1}(s+1)}
+\|\widehat{\Psi}_{p}\|_{L^{1}(s+1)})\leq C_{\Psi}.
\end{split}
\end{equation}
\end{lemma}
\begin{proof}
From the process of constructing the extended approximation $\epsilon\widetilde{\Psi}^{ext}$, we have formally $Res(\epsilon\widetilde{\Psi}^{ext})=\mathcal{O}(\epsilon^{6})$ and $\epsilon\widetilde{\Psi}-\epsilon\widetilde{\Psi}_{\pm1}=\mathcal{O}(\epsilon^{2})$ on the time interval $[0,T_{0}/\epsilon^{2}]$. However, since $\|A(\epsilon \cdot)\|_{L^{2}}=\epsilon^{-1/2}\|A\|_{L^{2}}$, we lose a factor $\epsilon^{-1/2}$ in \eqref{Aesti-1} and \eqref{Aesti-2}.

Due to the Fourier truncation procedure the final approximation $\epsilon\Psi$ has the form
\begin{align*}
\epsilon\Psi=\sum_{j=-5}^{5}a_{j}, \ \text{with} \ supp\mathcal{F}(a_{j})\subset(jk_{0}-\delta, \ jk_{0}+\delta).
\end{align*}
Thus, there exists a constant $C=C(k_{0})>0$ such that $\|\Psi\|_{H^{s}}\leq C\|\Psi\|_{L^{2}}$ and $\|\widehat{\Psi}\|_{L^{1}(s)}\leq C\|\widehat{\Psi}\|_{L^{1}}$ for all $s\geq0$. Now applying estimate \eqref{1/2} to $A_{j\ell}^{n}$  with $m=0$ and $M$ determined by the maximum regularity of the respective $\widetilde{A}_{j\ell}^{n}$, see Lemma \ref{L1}, we obtain \eqref{Aesti-1}-\eqref{Aesti-2} from construction of $\epsilon\Psi$, if $s_{N}\geq6$.

In contrast to the $L^{2}$-norm, we have $\|u(\epsilon \cdot)\|_{C_{b}^{0}}=\|u\|_{C_{b}^{0}}$ and $\|\widehat{u}\|_{L^{1}}=\|\epsilon^{-1}\widehat{A}(\epsilon^{-1}\cdot)\|_{L^{1}}$. Consequently, \eqref{Aesti-3} follows from construction of $\Psi_{\pm1}$ and $\Psi_{p}$.
\end{proof}
\begin{remark}
From bound \eqref{Aesti-3}, it follows that
\begin{align*}
\|\Psi R\|_{H^{s}}\leq C\|\Psi\|_{C_{b}^{s}}\|R\|_{H^{s}}\leq C\|\widehat{\Psi}\|_{L_{1}^{s}}\|R\|_{H^{s}},
\end{align*}
without loss of powers in $\epsilon$.
\end{remark}
Moreover, from an analogous argumentation as in the proof of Lemma 3.3 in \cite{D} it follows that $\partial_{t}\Psi_{\pm1}$ can be approximated by $-\Omega\Psi_{\pm1}$ respectively. More precisely, we have the following lemma.
\begin{lemma}\label{LO}
Fix $s>0$, there exists constant $C_{A}>0$ such that
\begin{align}\label{11s}
\|\partial_{t}\widehat{\psi}_{\pm1}+i\omega\widehat{\psi}_{\pm1}\|_{L^{1}(s)}\leq C_{A}\epsilon^{2}.
\end{align}
\end{lemma}

\section{\textbf{Evolutionary equations for the error $R$}}
In order to prove Theorem \ref{Thm1}, we need to obtain the uniform estimate for the error between the real solutions for the non-isentropic Euler-Poisson system \eqref{EP} and the formal approximation solutions $\epsilon\Psi$. In particular, we define the \emph{error}
\begin{align*}
\epsilon^{\beta}R=U-\epsilon\Psi
\end{align*}
between $U$ for the diagonalized system \eqref{short} of \eqref{EP} and the approximation solutions $\epsilon\Psi$. We need to show that it is of order $\mathcal{O}(\epsilon^{\beta})$ for all $t\in[0, T_{0}/\epsilon^{2}]$ and some $\beta>1$, i.e. we have to prove that $R$ is of order $\mathcal{O}(1)$ for all $t\in[0, T_{0}/\epsilon^{2}]$. Before giving the uniform estimate of the error $R$, we deduce the evolutionary equation for the error $R$ in this section.

Recall the final approximation $\epsilon\Psi$ in \eqref{modi}
\begin{align*}
\epsilon\Psi=\epsilon\begin{pmatrix} 0 \\ 0 \\ \psi_{1}+\psi_{-1}\end{pmatrix}
+\epsilon^{2}\begin{pmatrix} \psi_{p_{0}} \\ \psi_{p_{1}} \\ \psi_{p_{2}}\end{pmatrix}.
\end{align*}
For convenience of writing, we denote
\begin{align}\label{c123}
\psi_{c}:&=\psi_{1}+\psi_{-1},\
\varphi_{1}:=\psi_{p_{1}}-\psi_{p_{2}}, \nonumber\\
\varphi_{2}:&=(\gamma-2-q^{2})\psi_{p_{0}}+(\gamma-2)\psi_{p_{1}}+(\gamma-2)\psi_{p_{2}}, \
\varphi_{3}:=\psi_{p_{0}}+\psi_{p_{1}}+\psi_{p_{2}},\\
\varphi_{4}:&=\epsilon\psi_{c}+\epsilon^{2}\varphi_{3}, \
\varphi_{5}:=\epsilon(\gamma-1)\psi_{c}+\epsilon^{2}\varphi_{1} \nonumber.
\end{align}
Setting
\begin{align*}
U_{0}&=0+\epsilon^{2}\psi_{p_{0}}+\epsilon^{\beta}R_{0},\\
U_{1}&=0+\epsilon^{2}\psi_{p_{1}}+\epsilon^{\beta}R_{1},\\
U_{-1}&=\epsilon(\psi_{1}+\psi_{-1})+\epsilon^{2}\psi_{p_{-1}}+\epsilon^{\beta}R_{-1},
\end{align*}
for some $\beta>1$ sufficiently large, we see that the equations of $R_{j_{1}}$ with $j_{1}\in\{0,\pm1\}$ contain not only the diagonal terms $0, \Omega$ and $-\Omega$, but also terms of $\mathcal{O}(\epsilon)$ of the form $\epsilon B_{j_{1}j_{2}}(\psi_{c}, R_{j_{2}})$ with $j_{2}\in\{0,\pm1\}$. However, the appearance of terms of $\mathcal{O}(\epsilon)$ can perturb the linear evolution in such a way that the solutions begin to grow on time scale $\mathcal{O}(\epsilon^{-1})$ and hence we would lose control over the size of $R$ on the desired time scale $\mathcal{O}(\epsilon^{-2})$. Hence we need to remove these $\mathcal{O}(\epsilon)$ terms by making normal-form transforms of the form $\widetilde{R}_{j_{1}}=R_{j_{1}}+\epsilon N_{j_{1}j_{2}}(\psi_{c}, R_{j_{2}})$ in the equations for $R_{j_{1}}$. Unfortunately, we find that $N_{j_{1}j_{2}}(\psi_{c}, R_{j_{2}})$ is not well-defined because of the form of $\Omega$. More precisely, after careful calculation we find that the kernel function $n_{j_{1}j_{2}}(k,k-\ell,\ell)$ of $N_{j_{1}j_{2}}(\psi_{c}, R_{j_{2}})$ in Fourier space satisfies
\begin{align*}
n_{j_{1}j_{2}}(k,k-\ell,\ell)
=\frac{b_{j_{1}j_{2}}(k,k-\ell,\ell)}{-j_{1}\omega(k)-\omega(\pm k_{0})+j_{2}\omega(k\mp k_{0})},
\end{align*}
where $b_{j_{1}j_{2}}(k,k-\ell,\ell)$ is the kernel function of $B_{j_{1}j_{2}}(\psi_{c}, R_{j_{2}})$ in Fourier space. Note that the resonances $k=\pm k_{0}$ with $j_{1}=\mp1$ appear in the dominator of the expression of $n_{j_{1}j_{2}}$. This problem is induced by the behavior of $R_{j_{1}}$ near wave number zero. To solve this problem, we rescale the error by an additional power of $\epsilon$ for wave numbers close to zero. For some $\delta>0$ sufficiently small, but independent of $0<\epsilon\ll 1$, we define a weight function $\vartheta$ via the Fourier transform
\begin{equation}
\begin{split}\label{theta}
\widehat{\vartheta}(k)=\Big\{\begin{matrix} 1\ \ \ \ \ \ \ \ \ \ \ \ \ \ \ \ \ \ \ \ \ \ \ \ \ \ \ \ \text{for} \ \ |k|>\delta, \\ \epsilon+(1-\epsilon)| k|/\delta \ \ \ \ \ \ \ \ \ \ \text{for} \ \ |k|\leq\delta,\end{matrix}
\end{split}
\end{equation}
and make the following new ansatz
\begin{align}\label{var}
U_{0}&=0+\epsilon^{2}\psi_{p_{0}}+\epsilon^{\beta}\vartheta R_{0},\nonumber\\
U_{1}&=0+\epsilon^{2}\psi_{p_{1}}+\epsilon^{\beta}\vartheta R_{1},\\
U_{-1}&=\epsilon(\psi_{1}+\psi_{-1})+\epsilon^{2}\psi_{p_{-1}}+\epsilon^{\beta}\vartheta R_{-1},\nonumber
\end{align}
where $\beta=7/2$ and $\vartheta R_{j_{1}}$ is defined by $\widehat{\vartheta R}_{j_{1}}
=\widehat{\vartheta} \widehat{R}_{j_{1}}$. By this choice $\widehat{\vartheta} \widehat{R}_{j_{1}}(k)$ is small at the wave numbers close to zero reflecting the fact that the nonlinearity vanishes at $k=0$ in Fourier space.

Inserting \eqref{var} into \eqref{Diago}, after tedious computation, we have
\begin{align}\label{Rj}
\partial_{t}R_{0}=&-\frac{\epsilon}{\vartheta q^{2}}\left(q\psi_{c}\partial_{x}q^{2}\vartheta R_{0}\right)
+\frac{\epsilon(\gamma-1)}{\vartheta q^{2}}\left(\partial_{x} q\psi_{c}\vartheta((\gamma-2-q^{2})R_{0}+(\gamma-2)R_{1}+(\gamma-2)R_{-1})\right)\nonumber\\
&-\frac{\epsilon(\gamma-1)}{\vartheta q^{2}}\left((\gamma-2)\psi_{c}\partial_{x} q\vartheta(R_{1}-R_{-1})\right)
+\frac{\epsilon^{2}}{\vartheta q^{2}}\left(\partial_{x}q^{2}\psi_{p_{0}}q\vartheta(R_{1}-R_{-1})\right)\nonumber\\
&+\frac{\epsilon^{2}}{\vartheta q^{2}}\left(q\varphi_{1}\partial_{x}q^{2}\vartheta R_{0}\right)
-\frac{\epsilon^{2}(\gamma-1)}{\vartheta q^{2}}\left(\varphi_{2}\partial_{x}q\vartheta (R_{1}-R_{-1})\right)\nonumber\\
&-\frac{\epsilon^{2}(\gamma-1)}{\vartheta q^{2}}\left(\partial_{x}q\varphi_{1}
\vartheta\left((\gamma-2-q^{2})R_{0}+(\gamma-2)R_{1}+(\gamma-2)R_{-1})\right)\right)\nonumber\\
&-\frac{\epsilon^{\beta}(\gamma-1)}{\vartheta q^{2}}
\left(\vartheta\Big((\gamma-2-q^{2})R_{0}+(\gamma-2)R_{1}+(\gamma-2)R_{-1})\Big)
\partial_{x}q\vartheta(R_{1}-R_{-1})\right)\nonumber\\
&+\frac{\epsilon^{\beta}}{\vartheta q^{2}}\left(q\vartheta(R_{1}-R_{-1})\partial_{x}q^{2}\vartheta R_{0})\right)
+\frac{\epsilon^{-\beta}}{\vartheta}Res_{0}(\epsilon\Psi),\nonumber\\
\partial_{t}R_{j}=&j\Omega R_{j}
+\frac{\epsilon}{2\vartheta q^{2}}\left(q\psi_{c}
(\partial_{x}q^{2}\vartheta R_{0})\right)
-\frac{\epsilon\partial_{x}}{2\vartheta}\left(q\psi_{c}\vartheta(R_{0}+R_{1}+R_{-1})\right)\nonumber\\
&+\frac{\epsilon\partial_{x}}{2\vartheta}\left(\psi_{c}q\vartheta(R_{1}-R_{-1})\right)
+\frac{\epsilon(\gamma-1)(\gamma-2)}{2\vartheta q^{2}}\left(\psi_{c}\partial_{x}q \vartheta(R_{1}-R_{-1})\right)\nonumber\\
&-\frac{\epsilon(\gamma-1)}{2\vartheta q^{2}}\left(\partial_{x}q\psi_{c}
\vartheta((\gamma-2-q^{2})R_{0}+(\gamma-2)R_{1}+(\gamma-2)R_{-1})\right)
\nonumber\\
&-\frac{\epsilon j\partial_{x}}{2\vartheta q}\left(q\psi_{c}q\vartheta(R_{1}-R_{-1})\right)
+\frac{\epsilon j(\gamma-2)}{2\vartheta q}\left(\psi_{c}\partial_{x}\vartheta(R_{0}+R_{1}+R_{-1})\right)\nonumber\\
&+\frac{\epsilon j}{2\vartheta q}\left(\partial_{x}\psi_{c}
\vartheta((\gamma-2-q^{2})R_{0}+(\gamma-2)R_{1}+(\gamma-2)R_{-1})\right)\nonumber\\
&-\frac{\epsilon j\partial_{x}}{2\vartheta q(1-\partial_{x}^{2})}
\left(\frac{\psi_{c}}{1-\partial_{x}^{2}}
\frac{\vartheta (R_{0}+R_{1}+R_{-1})}{1-\partial_{x}^{2}}\right)
-\frac{\epsilon^{2}}{2\vartheta q^{2}}\left(q\varphi_{1}\partial_{x}q^{2}\vartheta R_{0}\right)\nonumber\\
&-\frac{\epsilon^{2}}{2\vartheta q^{2}}\left(\partial_{x}q^{2}\psi_{p_{0}}q\vartheta(R_{1}-R_{-1})\right)
+\frac{\epsilon^{2}\partial_{x}}{2\vartheta}
\left(q\varphi_{1}\vartheta(R_{0}+R_{1}+R_{-1})\right)\nonumber\\
&+\frac{\epsilon^{2}\partial_{x}}{2\vartheta}\left(\varphi_{3}q\vartheta(R_{1}-R_{-1})\right)
+\frac{\epsilon^{2}(\gamma-1)}{2\vartheta q^{2}}\left(\varphi_{2}\partial_{x}q\vartheta(R_{1}-R_{-1})\right)\nonumber\\
&+\frac{\epsilon^{2}(\gamma-1)}{2\vartheta q^{2}}\left(\partial_{x}q\varphi_{1}
\vartheta((\gamma-2-q^{2})R_{0}+(\gamma-2)R_{1}+(\gamma-2)R_{-1})\right)\nonumber\\
&+\frac{\epsilon^{2}j\partial_{x}}{2\vartheta q}\left(q\varphi_{1}q\vartheta(R_{1}-R_{-1}\right))
+\frac{\epsilon^{2}j}{2\vartheta q}\left(\varphi_{2}\partial_{x}\vartheta(R_{0}+R_{1}+R_{-1})\right)\nonumber\\
&+\frac{\epsilon^{2}j}{2\vartheta q}\left(\partial_{x}\varphi_{3}
\vartheta((\gamma-2-q^{2})R_{0}+(\gamma-2)R_{1}+(\gamma-2)R_{-1})\right)\nonumber\\
&-\frac{\epsilon^{2}j\partial_{x}}{2\vartheta q(1-\partial_{x}^{2})}
\left(\frac{\varphi_{3}}{1-\partial_{x}^{2}}
\frac{\vartheta (R_{0}+R_{1}+R_{-1})}{1-\partial_{x}^{2}}\right)\nonumber\\
&-\frac{\epsilon^{2}j}{4\vartheta q}\frac{\partial_{x}(2+\partial_{x}^{2})}{(1-\partial_{x}^{2})^{2}}
\left((\frac{\psi_{c}}{1-\partial_{x}^{2}})^{2}\frac{\vartheta (R_{0}+R_{1}+R_{-1})}{1-\partial_{x}^{2}}\right)\nonumber\\
&-\frac{\epsilon^{2}j\partial_{x}}{2\vartheta q}\left(\psi_{c}^{2}\vartheta(R_{0}+R_{1}+R_{-1})\right)
+\frac{\epsilon^{2}j}{2\vartheta q}\left(\psi_{c}\partial_{x}(\psi_{c}\vartheta(R_{0}+R_{1}+R_{-1}))\right)\nonumber\\
&+\frac{\epsilon^{2}j}{4\vartheta q}\left(\partial_{x}(\psi_{c}^{2})
\vartheta((\gamma-1-q^{2})R_{0}+(\gamma-1)R_{1}+(\gamma-1)R_{-1})\right)\nonumber\\
&-\frac{\epsilon^{3}j}{2\vartheta q}\frac{\partial_{x}(2+\partial_{x}^{2})}{(1-\partial_{x}^{2})^{2}}
\left(\frac{\varphi_{3}}{1-\partial_{x}^{2}}\frac{\psi_{c}}{1-\partial_{x}^{2}}
\frac{\vartheta (R_{0}+R_{1}+R_{-1})}{1-\partial_{x}^{2}}\right)\nonumber\\
&-\frac{\epsilon^{3}j\partial_{x}}{\vartheta q}\left(\psi_{c}\varphi_{3}\vartheta(R_{0}+R_{1}+R_{-1})\right)
+\frac{\epsilon^{3}(\gamma-1)j}{2\vartheta q}\left(\psi_{c}\partial_{x}(\varphi_{3}\vartheta(R_{0}+R_{1}+R_{-1}))\right)\nonumber\\
&+\frac{\epsilon^{3}j}{2\vartheta q}\left(\varphi_{2}\partial_{x}(\psi_{c}
\vartheta(R_{0}+R_{1}+R_{-1}))\right)\nonumber\\
&+\frac{\epsilon^{3}j}{2\vartheta q}\left(\varphi_{2}\partial_{x}(\psi_{c}\varphi_{3})
\vartheta((\gamma-1-q^{2})R_{0}+(\gamma-1)R_{1}+(\gamma-1)R_{-1})\right)\nonumber\\
&-\frac{\epsilon^{\beta}}{2\vartheta q^{2}}\left(q\vartheta(R_{1}-R_{-1})\partial_{x}q^{2}\vartheta R_{0}\right)
+\frac{\epsilon^{\beta}\partial_{x}}{2\vartheta}\left(q\vartheta(R_{1}-R_{-1})\vartheta(R_{0}+R_{1}+R_{-1})\right)\nonumber\\
&+\frac{\epsilon^{\beta}(\gamma-1)}{2\vartheta q^{2}}\left(\vartheta(
(\gamma-2-q^{2})R_{0}+(\gamma-2)R_{1}+(\gamma-2)R_{-1})\partial_{x}q\vartheta(R_{1}-R_{-1})\right)\nonumber\\
&+\frac{\epsilon^{\beta}j\partial_{x}}{4\vartheta q}(q\vartheta(R_{1}-R_{-1}))^{2}-\frac{\epsilon^{\beta}j\partial_{x}}{4\vartheta q(1-\partial_{x}^{2})}
\left(\frac{\vartheta (R_{0}+R_{1}+R_{-1})}{1-\partial_{x}^{2}}\right)^{2}\nonumber\\
&+\frac{\epsilon^{\beta}j}{2\vartheta q}\left(
\vartheta((\gamma-2-q^{2})R_{0}+(\gamma-2)R_{1}+(\gamma-2)R_{-1})
\partial_{x}\vartheta(R_{0}+R_{-1}+R_{1})\right)\nonumber\\
&-\frac{\epsilon^{\beta}j}{4\vartheta q}\frac{\partial_{x}(2+\partial_{x}^{2})}{(1-\partial_{x}^{2})^{2}}
\left(\frac{\varphi_{4}}{1-\partial_{x}^{2}}
\left(\frac{\vartheta (R_{0}+R_{1}+R_{-1})}{1-\partial_{x}^{2}}\right)^{2}\right)\nonumber\\
&-\frac{\epsilon^{\beta}j\partial_{x}}{2\vartheta q}\left(\varphi_{4}(\vartheta(R_{0}+R_{1}+R_{-1}))^{2}\right)
+\frac{\epsilon^{\beta}j}{4\vartheta q}\left(\varphi_{5}\partial_{x}(\vartheta(R_{0}+R_{1}+R_{-1}))^{2})\right)\nonumber\\
&+\frac{\epsilon^{\beta}j}{2\vartheta q}\left(\partial_{x}(\varphi_{4}
\vartheta(R_{0}+R_{1}+R_{-1}))\vartheta((\gamma-1-q^{2})R_{0}+(\gamma-1)R_{1}+(\gamma-1)R_{-1})\right)\nonumber\\
&-\frac{\epsilon^{2\beta}j}{12\vartheta q}\frac{\partial_{x}(2+\partial_{x}^{2})}{(1-\partial_{x}^{2})^{2}}
\left(
\left(\frac{\vartheta(R_{0}+R_{1}+R_{-1})}{1-\partial_{x}^{2}}\right)^{3}\right)
-\frac{\epsilon^{2\beta}j\partial_{x}}{6\vartheta q}\left((\vartheta(R_{0}+R_{1}+R_{-1}))^{3}\right)\nonumber\\
&+\frac{\epsilon^{2\beta}j}{4\vartheta q}\left(\vartheta((\gamma-1-q^{2})R_{0}+(\gamma-1)R_{1}+(\gamma-1)R_{-1})
\partial_{x}(\vartheta(R_{0}+R_{1}+R_{-1}))^{2})\right)\nonumber\\
&+\frac{\epsilon^{3}}{\vartheta}\widetilde{\mathcal{H}}(\vartheta R)
+\frac{\epsilon^{-\beta}}{\vartheta}Res_{j}(\epsilon\Psi),
\end{align}
where $j\in\{\pm1\}$, $\frac{1}{\vartheta}=\vartheta^{-1}$ and $\widehat{\frac{1}{\vartheta}}(k)=\widehat{\vartheta}^{-1}(k) =(\widehat{\vartheta}(k))^{-1}$. These RHS of these long equations can be grouped according to powers of $\epsilon$. However, we keep all these terms here and we will show how to treat these terms in the next section.

From the definition of $\vartheta$ in \eqref{theta}, it is obvious that $\widehat{\vartheta}^{-1}(k)$ is at most of order $\mathcal{O}(\epsilon^{-1})$ for $|k|\leq\delta$ but of order $\mathcal{O}(1)$ for $|k|>\delta$. Therefore, many terms that are formally of $\mathcal{O}(\epsilon)$ and $\mathcal{O}(\epsilon^{2})$ in above equations are actually of $\mathcal{O}(1)$ and $\mathcal{O}(\epsilon)$, respectively, due to the appearance of ${\vartheta}^{-1}$. Therefore, in order to justify the NLS approximation on the desired long time intervals $\mathcal{O}(\epsilon^{-2})$, we have to make normal-form transforms to remove these $\mathcal{O}(1)$ and $\mathcal{O}(\epsilon)$ terms, to simplify the equations of the error $R$.

\section{\textbf{Normal-form transforms}}
In order to eliminate the $\mathcal{O}(1)$ and $\mathcal{O}(\epsilon)$ terms on the right hand sides of evolutionary equations for $R$ in \eqref{Rj} we make a series of normal-form transforms that removes one or more of the bad terms. In the process of making normal-form transforms for the non-isentropic Euler-Poisson system \eqref{EP}, we find that the occurrence of the trivial resonance at wave number $k=0$ always implies the existence of nontrivial resonances at wave number $k=\pm k_{0}$. In addition, in the non-isentropic case, due to the appearance of the temperature $T$, one component of the diagonalized system of equations \eqref{Diago} has a linear frequency that is identically zero. As a result we get additional resonances at wave number $k=\pm 2k_{0}$. This necessitates different rescalings of the correction to the NLS approximation at different wave numbers. This problem does not occur in the isentropic case, as studied in our previous paper \cite{LP19}.

\subsection{\textbf{Basic strategies: a toy example}}\label{Sect5.1}

In this subsection, we will explain the general strategy with a quite specific example. We consider the terms from the second equation in \eqref{Rj}
\begin{align}\label{rn}
\partial_{t}R_{1}=&\Omega R_{1}
-\frac{\epsilon\partial_{x}}{2\vartheta}\left(q\psi_{c}\vartheta R_{1}\right)+\cdot\cdot\cdot
\end{align}
and we want to eliminate the term $\vartheta^{-1}\epsilon B(\psi_{c},\vartheta R_{1}):=\frac{\epsilon\partial_{x}}{2\vartheta}\left(q\psi_{c}\vartheta R_{1}\right)$ from the above equation.

Let
\begin{align}\label{RN}
\widetilde{R}_{1}=R_{1}+\epsilon N(\psi_{c},R_{1})
\end{align}
where $N$ is a bilinear form to be determined. 
Inserting \eqref{RN} to equation \eqref{rn}, we have
\begin{align*}
\partial_{t}\widetilde{R}_{1}=&\partial_{t}R_{1}+\epsilon N(\partial_{t}\psi_{c},R_{1})+\epsilon N(\psi_{c},\partial_{t}R_{1})\\
=&\Omega R_{1}
+\vartheta^{-1}\epsilon B(\psi_{c},\vartheta R_{1})
+\epsilon N(\partial_{t}\psi_{c},R_{1})+\epsilon N(\psi_{c},\partial_{t}R_{1})+\cdot\cdot\cdot\\
=&\Omega \widetilde{R}_{1}
+\vartheta^{-1}\epsilon B(\psi_{c},\vartheta R_{1})
-\epsilon \Omega N(\psi_{c},R_{1})
+\epsilon N(\Omega\psi_{c},R_{1})+\epsilon N(\psi_{c},\Omega R_{1})+\cdot\cdot\cdot,
\end{align*}
where we have used \eqref{11s} in Lemma \ref{LO}. Therefore, in order to eliminate the term $\vartheta^{-1}\epsilon B(\psi_{c},\vartheta R_{1})$, we choose $N$ such that
\begin{align}\label{B}
\vartheta^{-1}\epsilon B(\psi_{c},\vartheta R_{1})
-\epsilon \Omega N(\psi_{c},R_{1})
+\epsilon N(\Omega\psi_{c},R_{1})+\epsilon N(\psi_{c},\Omega R_{1})=0.
\end{align}
In order to compute the kernel $n(k,k-\ell,\ell)$ of $N$ in Fourier space, we take Fourier transform of \eqref{B} to obtain
\begin{align*}
&\int \widehat{\vartheta}^{-1}(k)b(k,k-\ell,\ell)\widehat{\psi}_{c}(k-\ell)\widehat{\vartheta}(\ell) \widehat{R}_{1}(\ell)d\ell\\
=&i\omega(k)\int n(k,k-\ell,\ell)\widehat{\psi}_{c}(k-\ell)\widehat{R}_{1}(\ell)d\ell\\
&-i\int n(k,k-\ell,\ell)\omega(k-\ell)\widehat{\psi}_{c}(k-\ell)\widehat{R}_{1}(\ell)d\ell\\
&-i\int n(k,k-\ell,\ell)\widehat{\psi}_{c}(k-\ell)\omega(\ell)\widehat{R}_{1}(\ell)d\ell,
\end{align*}
where $b(k,k-\ell,\ell)$ is the kernel of $B(\psi_{c}, R_{1})$ in Fourier space. Therefore
\begin{align*}
n(k,k-\ell,\ell)=\frac{-ib(k,k-\ell,\ell)}{\omega(k)-\omega(k-\ell)-\omega(\ell)}
\frac{\widehat{\vartheta}(\ell)}{\widehat{\vartheta}(k)}.
\end{align*}
The kernel $n(k,k-\ell,\ell)$ will be well-defined only when the denominator $\omega(k)-\omega(k-\ell)-\omega(\ell)$ is bounded away from zero, or in some cases, a zero in this expression is off-set by a zero of the numerator $b(k,k-\ell,\ell)$ at the same values of $k$ and $\ell$.

\subsection{\textbf{Frequency projections}}
Note that the size of the Fourier transform of the nonlinear terms in the error equations \eqref{Rj} depends on whether $k$ is close to zero or not. In order to separate the behavior in these regions more clearly we define projection operators $P^{0}$ and $P^{1}$ by the Fourier multipliers
\begin{equation}\label{projections}
\begin{split}
\widehat{P}^{0}(k)=\chi_{\mid k\mid\leq\delta}(k)\ \ \ \text{and} \ \ \ \widehat{P}^{1}(k)=\mathbf{1}-\widehat{P}^{0}(k),
\end{split}
\end{equation}
for a $\delta>0$ sufficiently small (the same $\delta$ in the definition of $\vartheta$), but independent of $0<\epsilon\ll1$. When necessary we will write $R=R^{0}+R^{1}$ with $R^{j}=P^{j}R$, for $j=0,1$. In the following, the superscripts $0,1$ always denote the spectrum projections, and should not be confused with the subscripts which denote the component of $R$.

Applying the projection operators $P^{0}$ and $P^{1}$ to system \eqref{Rj}, 
we have
\begin{align}\label{BR}
\partial_{t}R^{0}_{j}&=j\Omega R^{0}_{j}
+\frac{P^{0}}{\vartheta}\sum_{n\in\{0,\pm1\}}\Big(\epsilon \sum_{l\in\{\pm1\}}B^{0,1}_{j,l,n}(\psi_{l},\vartheta R^{1}_{n})
+\epsilon^{2} B^{0,0}_{j,0,n}(\psi_{0},\vartheta R^{0}_{n})\nonumber\\
&+\epsilon^{2} B^{0,1}_{j,0,n}(\psi_{0},\vartheta R^{1}_{n})
+\epsilon^{2}\sum_{l\in\{\pm2\}}B^{0,1}_{j,l,n}(\psi_{l},\vartheta R^{1}_{n})
+\epsilon^{2} \sum_{l\in\{\pm1\}}B^{0,1}_{j,l,n}(\psi^{1}_{l},\vartheta R^{1}_{n})\\
&+\epsilon^{2}j\sum_{l_{1},l_{2}\in\{\pm1\}}T^{0,1}_{j,l_{1},l_{2},n}(\psi_{l_{1}},\psi_{l_{2}},\vartheta R^{1}_{n})\Big)
+\mathcal{O}(\epsilon^{2}),\nonumber
\end{align}
\begin{align}\label{BR1}
\partial_{t}R^{1}_{j}&=j\Omega R^{1}_{j}+\frac{P^{1}}{\vartheta}\sum_{n\in\{0,\pm1\}}\Big(\epsilon \sum_{l\in\{\pm1\}}B^{1,0}_{j,l,n}(\psi_{l},\vartheta R^{0}_{n})
+\epsilon^{2} B^{1,0}_{j,0,n}(\psi_{0},\vartheta R^{0}_{n})\nonumber\\
&+\epsilon^{2} \sum_{l\in\{\pm2\}}B^{1,0}_{j,l,n}(\psi_{l},\vartheta R^{0}_{n})
+\epsilon^{2} \sum_{l\in\{\pm1\}}B^{1,0}_{j,l,n}(\psi^{1}_{l},\vartheta R^{0}_{n})\nonumber\\
&+\epsilon^{2}j\sum_{l_{1},l_{2}\in\{\pm1\}}T^{1,0}_{j,l_{1},l_{2},n}(\psi_{l_{1}},\psi_{l_{2}},\vartheta R^{0}_{n})\Big)\\
&+\frac{P^{1}}{\vartheta}\sum_{n\in\{0,\pm1\}}\Big(\epsilon
\sum_{l\in\{\pm1\}}B^{1,1}_{j,l,n}(\psi_{l},\vartheta R^{1}_{n})
+\epsilon^{2} B^{1,1}_{j,0,n}(\psi_{0},\vartheta R^{1}_{n})
+\epsilon^{2} \sum_{l\in\{\pm2\}}B^{1,1}_{j,l,n}(\psi_{l},\vartheta R^{1}_{n})\nonumber\\
&+\epsilon^{2} \sum_{l\in\{\pm1\}}B^{1,1}_{j,l,n}(\psi^{1}_{l},\vartheta R^{1}_{n})
+\epsilon^{2}j\sum_{l_{1},l_{2}\in\{\pm1\}}T^{1,1}_{j,l_{1},l_{2},n}(\psi_{l_{1}},\psi_{l_{2}},\vartheta R^{1}_{n})\Big)
+\mathcal{O}(\epsilon^{2})\nonumber,
\end{align}
where $j\in\{0, \pm1\}$.

According to \eqref{Rj}, we have for $j=0$
\begin{align}\label{Bc02}
&\epsilon P^{j_{1}}\sum_{n\in\{0,\pm1\}}B^{j_{1},j_{2}}_{0,l,n}(\psi_{l},R^{j_{2}}_{n})\nonumber\\
:=&\epsilon P^{j_{1}}\Big[-\frac{1}{q^{2}}\left(q\psi_{l}\partial_{x}q^{2} R^{j_{2}}_{0}\right)-\frac{\gamma-1}{q^{2}}\left((\gamma-2)\psi_{l}\partial_{x} q (R^{j_{2}}_{1}-R^{j_{2}}_{-1})\right)\\
&+\frac{\gamma-1}{q^{2}}\left(\partial_{x} q\psi_{l}((\gamma-2-q^{2}) R^{j_{2}}_{0}+(\gamma-2)(R^{j_{2}}_{1}+R^{j_{2}}_{-1})\right)\Big],\nonumber
\end{align}
\begin{align}\label{B00n}
&\epsilon^{2}P^{j_{1}}\sum_{n\in\{0,\pm1\}}B^{j_{1},j_{2}}_{0,0,n}(\psi_{0},R^{j_{2}}_{n})\nonumber\\
:=&\epsilon^{2}P^{j_{1}}\Big[\frac{1}{q^{2}}\left(\partial_{x}q^{2}\psi_{00}q (R^{j_{2}}_{1}-R^{j_{2}}_{-1})\right)
+\frac{1}{q^{2}}\left(q(\psi_{01}-\psi_{0-1})\partial_{x}q^{2}R^{j_{2}}_{0}\right)\nonumber\\
&-\frac{\gamma-1}{ q^{2}}\left(((\gamma-2-q^{2})\psi_{00}+(\gamma-2)(\psi_{01}+\psi_{0-1}))\partial_{x}q (R^{j_{2}}_{1}-R^{j_{2}}_{-1})\right)\\
&-\frac{\gamma-1}{q^{2}}\left(\partial_{x}q(\psi_{01}-\psi_{0-1})
\left((\gamma-2-q^{2})R^{j_{2}}_{0}+(\gamma-2)(R^{j_{2}}_{1}+ R^{j_{2}}_{-1})\right)\right)\Big],\nonumber
\end{align}
\begin{align}\label{B02n}
&\epsilon^{2}P^{j_{1}}\sum_{n\in\{0,\pm1\}}B^{j_{1},j_{2}}_{0,\pm2,n}(\psi_{\pm2},R^{j_{2}}_{n})\nonumber\\
:=&\epsilon^{2}P^{j_{1}}\bigg[\frac{1}{ q^{2}}\left(\partial_{x}q^{2}\psi_{\pm20}q (R^{j_{2}}_{1}-R^{j_{2}}_{-1})\right)
+\frac{1}{ q^{2}}\left(q(\psi_{\pm21}-\psi_{\pm2-1})\partial_{x}q^{2} R^{j_{2}}_{0}\right)\nonumber\\
&-\frac{\gamma-1}{q^{2}}\left(((\gamma-2-q^{2})\psi_{\pm20}
+(\gamma-2)(\psi_{\pm21}+\psi_{\pm2-1}))\partial_{x}q (R^{j_{2}}_{1}-R^{j_{2}}_{-1})\right)\\
&-\frac{\gamma-1}{q^{2}}\left(\partial_{x}q(\psi_{\pm21}-\psi_{\pm2-1})
\left((\gamma-2-q^{2})R^{j_{2}}_{0}+(\gamma-2)(R^{j_{2}}_{1}+ R^{j_{2}}_{-1})\right)\right)\bigg],\nonumber
\end{align}
and for $j\in\{\pm1\}$ we have
\begin{align}\label{Bc03}
&\epsilon P^{j_{1}}\sum_{n\in\{0,\pm1\}}B^{j_{1},j_{2}}_{j,l,n}(\psi_{l},R^{j_{2}}_{n})\nonumber\\
:=&\epsilon P^{j_{1}}\Bigg[\frac{1}{2q^{2}}\left(q\psi_{l}
(\partial_{x}q^{2}R^{j_{2}}_{0})\right)
-\frac{\partial_{x}}{2}\left(q\psi_{l}(R^{j_{2}}_{0}+R^{j_{2}}_{1}+R^{j_{2}}_{-1})\right)\nonumber\\
&+\frac{\partial_{x}}{2}\left(\psi_{l}q(R^{j_{2}}_{1}-R^{j_{2}}_{-1})\right)
+\frac{(\gamma-1)(\gamma-2)}{2q^{2}}\left(\psi_{l}\partial_{x}q (R^{j_{2}}_{1}-R^{j_{2}}_{-1})\right)\nonumber\\
&-\frac{\gamma-1}{2q^{2}}\left(\partial_{x}q\psi_{l}
((\gamma-2-q^{2})R^{j_{2}}_{0}+(\gamma-2)(R^{j_{2}}_{1}+R^{j_{2}}_{-1})\right)
\\
&-\frac{j\partial_{x}}{2q}\left(q\psi_{l}q(R^{j_{2}}_{1}-R^{j_{2}}_{-1})\right)
+\frac{j(\gamma-2)}{2 q}\left(\psi_{l}\partial_{x}(R^{j_{2}}_{0}+R^{j_{2}}_{1}+R^{j_{2}}_{-1})\right)\nonumber\\
&+\frac{j}{2q}\left(\partial_{x}\psi_{l}
((\gamma-2-q^{2})R^{j_{2}}_{0}+(\gamma-2)(R^{j_{2}}_{1}+R^{j_{2}}_{-1})\right)\nonumber\\
&-\frac{j\partial_{x}}{2q(1-\partial_{x}^{2})}
\left(\frac{\psi_{l}}{1-\partial_{x}^{2}}
\frac{(R^{j_{2}}_{0}+R^{j_{2}}_{1} +R^{j_{2}}_{-1})}{1-\partial_{x}^{2}}\right)\Bigg],\nonumber
\end{align}
\begin{align}\label{Bj0n}
&\epsilon^{2}P^{j_{1}}\sum_{n\in\{0,\pm1\}}B^{j_{1},j_{2}}_{j,0,n}(\psi_{0},R^{j_{2}}_{n})\nonumber\\
:=&\epsilon^{2}P^{j_{1}}\Bigg[-\frac{1}{2 q^{2}}\left(q(\psi_{01}-\psi_{0-1})\partial_{x}q^{2}R^{j_{2}}_{0}\right)\nonumber\\
&-\frac{1}{2 q^{2}}\left(\partial_{x}q^{2}\psi_{00}q(R^{j_{2}}_{1}-R^{j_{2}}_{-1})\right)
+\frac{\partial_{x}}{2}
\left(q(\psi_{01}-\psi_{0-1})(R^{j_{2}}_{0}+R^{j_{2}}_{1}+R^{j_{2}}_{-1})\right)\nonumber\\
&+\frac{\partial_{x}}{2}\left((\psi_{00}+\psi_{01}+\psi_{0-1})
q\vartheta(R^{j_{2}}_{1}-R^{j_{2}}_{-1})\right)\nonumber\\
&+\frac{\gamma-1}{2 q^{2}}\left(((\gamma-2-q^{2})\psi_{00}
+(\gamma-2)(\psi_{01}+\psi_{0-1}))\partial_{x}q(R^{j_{2}}_{1}-R^{j_{2}}_{-1})\right)\\
&+\frac{\gamma-1}{2q^{2}}\left(\partial_{x}q(\psi_{01}-\psi_{0-1})
((\gamma-2-q^{2})R^{j_{2}}_{0}+(\gamma-2)(R^{j_{2}}_{1}+R^{j_{2}}_{-1}))\right)\nonumber\\
&+\frac{j\partial_{x}}{2 q}\left(q(\psi_{01}-\psi_{0-1})q(R^{j_{2}}_{1}-R^{j_{2}}_{-1}\right))\nonumber\\
&+\frac{j}{2q}\left(((\gamma-2-q^{2})\psi_{00}+(\gamma-2)(\psi_{01}+\psi_{0-1}))
\partial_{x}(R^{j_{2}}_{0}+R^{j_{2}}_{1}+R^{j_{2}}_{-1})\right)\nonumber\\
&+\frac{j}{2q}\left(\partial_{x}(\psi_{00}+\psi_{01}+\psi_{0-1})
((\gamma-2-q^{2})R^{j_{2}}_{0}+(\gamma-2)(R^{j_{2}}_{1}+R^{j_{2}}_{-1}))\right)\nonumber\\
&-\frac{j\partial_{x}}{2q(1-\partial_{x}^{2})}
\left(\frac{(\psi_{00}+\psi_{01}+\psi_{0-1})}{1-\partial_{x}^{2}}
\frac{(R^{j_{2}}_{0}+R^{j_{2}}_{1}+R^{j_{2}}_{-1})}{1-\partial_{x}^{2}}\right)\Bigg],\nonumber
\end{align}
\begin{align}\label{Bj2n}
&\epsilon^{2}P^{j_{1}}\sum_{n\in\{0,\pm1\}}B^{j_{1},j_{2}}_{j,\pm2,n}(\psi_{\pm2},R^{j_{2}}_{n})\nonumber\\
:=&\epsilon^{2}P^{j_{1}}\Bigg[-\frac{1}{2 q^{2}}\left(q(\psi_{\pm21}-\psi_{\pm2-1})\partial_{x}q^{2} R^{j_{2}}_{0}\right)
-\frac{1}{2 q^{2}}\left(\partial_{x}q^{2}(\psi_{\pm20})q(R^{j_{2}}_{1}-R^{j_{2}}_{-1})\right)\nonumber\\
&+\frac{\partial_{x}}{2}
\left(q(\psi_{\pm21}-\psi_{\pm2-1})
(R^{j_{2}}_{0}+R^{j_{2}}_{1}+R^{j_{2}}_{-1})\right)\nonumber\\
&+\frac{\partial_{x}}{2}
\left((\psi_{\pm20}+\psi_{\pm21}+\psi_{\pm2-1})
q(R^{j_{2}}_{1}-R^{j_{2}}_{-1})\right)\nonumber\\
&+\frac{\gamma-1}{2q^{2}}\left(((\gamma-2-q^{2})\psi_{\pm20}
+(\gamma-2)(\psi_{\pm21}+\psi_{\pm2-1}))
\partial_{x}q(R^{j_{2}}_{1}-R^{j_{2}}_{-1})\right)\\
&+\frac{\gamma-1}{2q^{2}}\left(\partial_{x}q(\psi_{\pm21}-\psi_{\pm2-1})
((\gamma-2-q^{2})R^{j_{2}}_{0}+(\gamma-2)(R^{j_{2}}_{1}+R^{j_{2}}_{-1})\right)\nonumber\\
&+\frac{j\partial_{x}}{2 q}\left(q(\psi_{\pm21}-\psi_{\pm2-1})q(R^{j_{2}}_{1}-R^{j_{2}}_{-1}\right))\nonumber\\
&+\frac{j}{2q}\left(((\gamma-2-q^{2})\psi_{\pm20}
+(\gamma-2)(\psi_{\pm21}+\psi_{\pm2-1}))
\partial_{x}(R^{j_{2}}_{0}+R^{j_{2}}_{1}+R^{j_{2}}_{-1})\right)\nonumber\\
&+\frac{j}{2 q}\left(\partial_{x}(\psi_{\pm20}+\psi_{\pm21}+\psi_{\pm2-1})
((\gamma-2-q^{2})R^{j_{2}}_{0}+(\gamma-2)(R^{j_{2}}_{1}+R^{j_{2}}_{-1})\right)\nonumber\\
&-\frac{j\partial_{x}}{2q(1-\partial_{x}^{2})}
\left(\frac{(\psi_{\pm20}+\psi_{\pm21}+\psi_{\pm2-1})}{1-\partial_{x}^{2}}
\frac{(R^{j_{2}}_{0}+R^{j_{2}}_{1}+R^{j_{2}}_{-1})}{1-\partial_{x}^{2}}\right)\Bigg],\nonumber
\end{align}
\begin{align}\label{Bj12n}
&\epsilon^{2}P^{j_{1}}\sum_{n\in\{0,\pm1\}}T^{j_{1},j_{2}}_{j,l_{1},l_{2},n}
(\psi_{l_{1}},\psi_{l_{2}},R^{j_{2}}_{n})\nonumber\\
:=&\epsilon^{2}P^{j_{1}}\sum_{n\in\{0,\pm1\}}\bigg[-\frac{\partial_{x}}{2 q}\left(\psi_{l_{1}}\psi_{l_{2}}(R^{j_{2}}_{0}+R^{j_{2}}_{1}+R^{j_{2}}_{-1})\right)
+\frac{1}{3 q}\left(\psi_{l_{1}}\partial_{x}(\psi_{l_{2}}(R^{j_{2}}_{0}+R^{j_{2}}_{1}+R^{j_{2}}_{-1}))\right)\nonumber\\
&+\frac{1}{4q}\left(\partial_{x}(\psi_{l_{1}}\psi_{l_{2}})
((\gamma-2-q^{2})R^{j_{2}}_{0}+(\gamma-2)(R^{j_{2}}_{1}+R^{j_{2}}_{-1}))\right)\bigg],
\end{align}
for $j_{1}, j_{2}\in\{0,1\}$ and $B^{j_{1},j_{2}}_{j,l,n}(\psi^{1}_{l},R^{j_{2}}_{n})$ has similar expression with $B^{j_{1},j_{2}}_{j,l,n}(\psi_{l},R^{j_{2}}_{n})$ for $j\in\{0,\pm1\}$, only $\psi_{\pm1}$ being replaced by $\psi^{1}_{\pm1}$. Note that 
by definition, $\widehat{\vartheta}^{-1}(k)$ is at most of order $\mathcal{O}(\epsilon^{-1})$ for $k\leq\delta$ and is of order $\mathcal{O}(1)$ for $k>\delta$. Hence, we retain some terms in \eqref{BR} of the form $B_{j,0,n}$, $B_{j,2,n}$ and $T_{j,l_{1},l_{2},n}$ standing for trilinear terms, which will be treated in details below. Here we have used the fact that
\begin{align*}
P^{0}B_{j,l,n}^{0,0}(\psi_{l},\vartheta R^{0}_{n})=P^{0}B_{j,2,n}^{0,0}(\psi_{2},\vartheta R^{0}_{n})
=P^{0}T_{j,l_{1},l_{2},n}^{0,0}(\psi_{l_{1}},\psi_{l_{2}},\vartheta R^{0}_{n})=0
\end{align*}
for $j,n\in\{0,\pm1\}$
, due to the fact that $\widehat{\Psi}_{l}(k-\ell)=0$ unless $|k-\ell-lk_{0}|<\delta$ and $\widehat{R}^{0}(\ell)=0$ for $|\ell|>\delta$.

On the right hand side of the evolution for $R^{1}_{j}$, since $\frac{P^{1}}{\vartheta}=\mathcal{O}(1)$, all terms except $j\Omega R^{1}_{j}$ and $\frac{\epsilon P^{1}}{\vartheta}B^{1,r}_{j,l,n}(\psi_{l},\vartheta R^{l}_{n})$ with $r\in\{0, 1\}$ are at least of order $\mathcal{O}(\epsilon^{2})$ and then need not to be removed. Furthermore, the terms $\frac{\epsilon^{2}P^{0}}{\vartheta}B^{0,r}_{j,0,n}(\psi_{0},\vartheta R^{r}_{n})$ from \eqref{Bj0n} and $\frac{\epsilon^{2}P^{0}}{\vartheta}T^{0,r}_{j,1,-1,n}(\psi_{1},\psi_{-1},\vartheta R^{r}_{n})$ from \eqref{Bj12n} with $r\in\{0, \ 1\}$ in the evolutionary equation \eqref{BR} for $R^{0}_{j}$ are at least of $\mathcal{O}(\epsilon^{2})$ and need not to be removed either. Let us explain them one by one in the following.

Firstly, we show that the terms $\frac{\epsilon^{2}P^{0}}{\vartheta}B^{0,r}_{j,0,n}(\psi_{0},\vartheta R^{r}_{n})$ from \eqref{Bj0n} in the evolution equation \eqref{BR} for $R^{0}_{j}$ for $j,n\in\{0,\pm1\}$ and $r\in\{0,1\}$ are at least of $\mathcal{O}(\epsilon^{2})$. We split $B^{0,r}_{j,0,n}(\psi_{0},\vartheta R^{r}_{n})$ in the components
\begin{align*}
B^{0,r}_{j,0,n}(\psi_{0},\vartheta R^{r}_{n})
=\sum_{m=0,\pm1}B^{0,r,m}_{j,0,n}(\psi_{0m},\vartheta R^{r}_{n}).
\end{align*}
Note that each summand in $B^{0,r}_{j,0,n}$ contains as least one $x$-derivative and $\left|\ell\right|\leq|k|+|k-\ell|$. We obtain
\begin{align*}
\sup_{|\ell|=\mathcal{O}(1)}\widehat{P}^{0}b^{0,r,m}_{j,0,n}(k,k-\ell,\ell)
\leq C(|k|+ |k-\ell|)
\end{align*}
for $m\in\{0, \pm1\}$, where $b^{0,r,m}_{j,0,n}(k,k-\ell,\ell)$ is the kernel of $B^{0,r,m}_{j,0,n}$. Applying this estimate and noting that $\left|\frac{k}{\widehat{\vartheta}(k)}\right|\leq1$ and $\frac{k-\ell}{\epsilon}
\widehat{\psi}_{0m}(k-\ell)=\widehat{\partial_{X}\psi}_{0m}(k-\ell)$, we obtain
\begin{align*}
&\left|\int\frac{\widehat{P}^{0}(k)}{\widehat{\vartheta}(k)}b^{0,r,m}_{j,0,n}(k,k-\ell,\ell)
\widehat{\vartheta}(\ell)\widehat{\psi}_{0m}(k-\ell)\widehat{R}^{r}_{n}(\ell)d\ell\right|\\
\leq &C\left(\int|\widehat{P}^{0}(k)|\left|\frac{k}{\widehat{\vartheta}(k)}\right|
|\widehat{\psi}_{0m}(k-\ell)||\widehat{R}^{r}_{n}(\ell)|d\ell
+\int\left|\widehat{P}^{0}(k)||\frac{k-\ell}{\epsilon}\right|
|\widehat{\psi}_{0m}(k-\ell)||\widehat{R}^{r}_{n}(\ell)|d\ell\right)\\
\leq &C\left(\int|\widehat{P}^{0}(k)|
|\widehat{\psi}_{0m}(k-\ell)||\widehat{R}^{r}_{n}(\ell)|d\ell
+\int|\widehat{P}^{0}(k)|
|\widehat{\partial_{X}\psi}_{0m}(k-\ell)| |\widehat{R}^{r}_{n}(\ell)|d\ell\right),
\end{align*}
for $j, n, m\in\{0, \pm1\}$ and $r\in\{0, 1\}$. This implies
\begin{align}\label{b01i}
\left\|\frac{\epsilon^{2}P^{0}}{\vartheta}B^{0,r}_{j,0,n}(\psi_{0},\vartheta R^{r}_{n})\right\|_{H^{s}}
\leq \epsilon^{2}C(\Psi_{0})\|R^{r}_{n}\|_{H^{s}},
\end{align}
where $C(\Psi_{0})$ is independent of $\epsilon$, since $\widehat{\Psi}_{0}(k)$ is concentrated near $k=0$. Hence, the terms of $\frac{\epsilon^{2}P^{0}}{\vartheta}B^{0,r}_{j,0,n}(\psi_{0},\vartheta R^{r}_{n})$ with $r\in\{0, \ 1\}$ are of order $\mathcal{O}(\epsilon^{2})$ and need not to be eliminated.

Then we consider the trilinear terms $\frac{\epsilon^{2}P^{0}}{\vartheta}T^{0,r}_{j,1,-1,n}(\psi_{1},\psi_{-1},\vartheta R^{r}_{n})$ from \eqref{Bj12n} in the evolutionary equation \eqref{BR} for $R^{0}_{j}$ for $j,n\in\{0,\pm1\}$ and $r\in\{0,1\}$. Using $\left|p\right|\leq|k|+|k-p|$, we have
\begin{align}\label{t01}
\sup_{|p|=\mathcal{O}(1)}(\widehat{P}^{0}t^{0,r}_{j,1,-1,n})(k,k-\ell,\ell-p,p)
\leq C(|k|+|k-p|),
\end{align}
where $t^{0,r}_{j,1,-1,n}(k,k-\ell,\ell-p,p)$ is the kernel of $T^{0,r}_{j,1,-1,n}$. This estimate \eqref{t01} is obtained based on the fact that each summand of $T^{0,r}_{j,1,-1,n}$ contains at least one $x$-derivative. Applying \eqref{t01} we obtain
\begin{align*}
&\left|\int\int\frac{\widehat{P}^{0}(k)}{\widehat{\vartheta}(k)}t^{0,r}_{j,1,-1,n}(k,k-\ell,\ell-p,p)
\widehat{\vartheta}(\ell)\widehat{\psi}_{1}(k-\ell)\widehat{\psi}_{-1}(\ell-p)
\widehat{R}^{r}_{n}(p)d\ell dp\right|\\
\leq &C\bigg(\int\int|\widehat{P}^{0}(k)|\left|\frac{k}{\widehat{\vartheta}(k)}\right|
|\widehat{\psi}_{1}(k-\ell)\widehat{\psi}_{-1}(\ell-p)||\widehat{R}^{r}_{n}(p)|d\ell dp\\
&+\int\int\left|\widehat{P}^{0}(k)||\frac{k-p}{\epsilon}\right|
|\widehat{\psi}_{1}(k-\ell)\widehat{\psi}_{-1}(\ell-p)||\widehat{R}^{r}_{n}(p)|d\ell dp\bigg)\\
\leq &C\bigg(\int\int|\widehat{P}^{0}(k)|
|\widehat{\psi}_{1}(k-\ell)\widehat{\psi}_{-1}(\ell-p)||\widehat{R}^{r}_{n}|d\ell dp
+\int|\widehat{P}^{0}(k)|
|\widehat{\partial_{X}(\psi_{1}\psi_{-1})}(k-p)||\widehat{R}^{r}_{n}| dp\bigg),
\end{align*}
for $j,n\in\{0,\pm1\}$ and $r\in\{0,1\}$. This implies
\begin{align}\label{b01i-2}
\left\|\frac{\epsilon^{2}P^{0}}{\vartheta}T^{0,r}_{j,1,-1,n}(\psi_{1},\psi_{-1},\vartheta R^{r}_{n})\right\|_{H^{s}}
\leq \epsilon^{2}C(\psi_{1},\psi_{-1})\|R^{r}_{n}\|_{H^{s}},
\end{align}
where $C(\psi_{1},\psi_{-1})$ is independent of $\epsilon$, since $(\widehat{\psi_{1}\psi_{-1}})(k)$ is concentrated near $k=0$. Hence, the terms of $\frac{\epsilon^{2}P^{0}}{\vartheta}T^{0,r}_{j,1,-1,n}(\psi_{1},\psi_{-1},\vartheta R^{r}_{n})$ are of order $\mathcal{O}(\epsilon^{2})$ and need not to be eliminated.

According to the above analysis, all remaining terms that are actually of order $\mathcal{O}(1)$ or $\mathcal{O}(\epsilon)$ from the evolutionary equations \eqref{BR}-\eqref{BR1} can be classified into the following two categories:
\begin{description}
  \item[low frequency terms]
  \begin{equation*}
  \begin{split}\displaystyle &\epsilon\frac{P^{0}}{\vartheta}B^{0,1}_{j,\pm1,n}(\psi_{\pm1}, \vartheta R^{1}_{n}),\ \ \ \ \ \epsilon^{2}\frac{P^{0}}{\vartheta}B^{0,1}_{j,\pm1,n}(\psi^{l}_{\pm1}, \vartheta R^{1}_{n}),\ \ \ \  \\
  &\epsilon^{2}\frac{P^{0}}{\vartheta}B^{0,1}_{j,\pm2,n}(\psi_{\pm2}, \vartheta R^{1}_{n}), \ \ \ \ \epsilon^{2}\frac{j P^{0}}{\vartheta}T^{0,1}_{j,\pm1,\pm1,n}(\psi_{\pm1},\psi_{\pm1},\vartheta R^{1}_{n})\ \ \ \ \text{from \eqref{BR} and}\\
  \displaystyle & \epsilon\frac{P^{1}}{\vartheta}B^{1,0}_{j,\pm1,n}(\psi_{\pm1},\vartheta R^{0}_{n})\ \ \ \ \ \ \text{ from \eqref{BR1}},
  \end{split}
  \end{equation*}
  \item[high frequency terms]
  \begin{equation*}
  \begin{split}\displaystyle \displaystyle\epsilon\frac{P^{1}}{\vartheta} \sum_{l\in\{\pm1\}}B^{1,1}_{j,\pm1,n}(\psi_{\pm1},\vartheta R^{1}_{n}) \ \ \ \  \text{ from \eqref{BR1},}
  \end{split}
  \end{equation*}
\end{description}
where $j,n\in\{0,\pm1\}$. We note that only these low frequency terms and high frequency terms need to be eliminated by normal-form transforms. However, the loss of regularity for the high frequency terms leads to the loss of more derivatives for the transformed system. Hence, we leave temporarily the high frequency terms $\epsilon\frac{P^{1}}{\vartheta}
\sum_{l\in\{\pm1\}}B^{1,1}_{j,\pm1,n}(\psi_{\pm1},\vartheta R^{1}_{n})$ in the equation of $R^{1}_{j}$, which will be eliminated in Section 6 by a method of constructing a modified energy. In the next subsections 5.3 and 5.4, we will use twice normal-form transforms to eliminate the low frequency terms.

\subsection{\textbf{The first normal-form transform for the low frequency terms}}
In order to eliminate the above $\mathcal{O}(1)$ and $\mathcal{O}(\epsilon)$ low frequency terms in the system \eqref{BR}-\eqref{BR1}, we look for normal-form transforms of the form
\begin{align}\label{N0110}
\widetilde{R}_{j}^{0}&=R_{j}^{0}+\epsilon N_{j}^{0,1}(\Psi,R),\nonumber\\
\widetilde{R}_{j}^{1}&=R_{j}^{0}+\epsilon N_{j}^{1,0}(\Psi,R),
\end{align}
where $j\in\{0, \ \pm1\}$ and
\begin{align}\label{add}
\epsilon N_{j}^{0,1}(\Psi,R)
=&\sum_{n\in\{0,\pm1\}}\Bigg(\epsilon \sum_{l\in\{\pm1\}}\Big(N^{0,1}_{j,l,n}(\psi_{l},\vartheta R^{1}_{n})+\epsilon^{2} N^{0,1}_{j,l,n}(\psi^{1}_{l},\vartheta R^{1}_{n})
+\epsilon^{2}N^{0,1}_{j,l,l,n}(\psi_{l},\psi_{l},\vartheta R^{1}_{n})\Big)\nonumber\\
& \ \ \ \ \ \ \ \ \ \ \ \  \ +\epsilon^{2} N^{0,1}_{j,2,n}(\psi_{2},\vartheta R^{1}_{n})
+\epsilon^{2} N^{0,1}_{j,-2,n}(\psi_{-2},\vartheta R^{1}_{n})\Bigg),\\
\epsilon N_{j}^{1,0}(\Psi,R)=&\sum_{n\in\{0,\pm1\}}\sum_{l\in\{\pm1\}}\epsilon N^{1,0}_{j,l,n}(\psi_{l},\vartheta R^{0}_{n})\nonumber.
\end{align}
\subsubsection{\textbf{Construction and properties of $N^{0,1}$}}\label{Sect5.3.1}
In this subsection, we discuss systematically the construction of all components of this normal-form transform.

Firstly, we focus on $\epsilon N^{0,1}_{j,l,n}(\psi_{l},\vartheta R^{1}_{n})$ and $\epsilon^{2}N^{0,1}_{j,l,n}(\psi^{1}_{l},\vartheta R^{1}_{n})$ for $j, n\in\{0,\pm1\}$ and $l\in\{\pm1\}$. Proceeding as in Subsection 5.1, the kernels $n^{0,1}_{j,l,n}$ of $N^{0,1}_{j,l,n}$ should satisfy
\begin{align}\label{n01}
n^{0,1}_{j,l,n}(k,k-\ell,\ell)=\frac{i\widehat{P}^{0}(k)b^{0,1}_{j,l,n}(k,k-\ell,\ell)}
{-j\omega(k)-\omega(k-\ell)+n\omega(\ell)}\frac{\widehat{\vartheta}(\ell)}{\widehat{\vartheta}(k)},
\end{align}
where $b^{0,1}_{j,l,n}$ is the kernel of $B^{0,1}_{j,l,n}$. Note that the kernels of $N^{0,1}_{j,l,n}(\psi^{1}_{l},\vartheta R^{1}_{n})$ has same form with $N^{0,1}_{j,l,n}(\psi_{l},\vartheta R^{1}_{n})$.

Note that the $\widehat{P}^{0}$ and $\widehat{\psi}_{\pm1}$ (or $\widehat{\psi}^{1}_{\pm1}$) have supports localized near $k=0$ and $k-\ell=\pm k_{0}$ respectively, hence we only need to analyse $n^{0,1}_{j,l,n}(k,k-\ell,\ell)$ in \eqref{n01} in the region $|(k-\ell)\pm k_{0}|<\delta$ and $|k|<\delta$. As a result, the wave number $\ell$ is restricted to be bounded away from $0$ for $\delta>0$ sufficiently small but independent of $0<\epsilon\ll1$. Hence only the resonance $k=0$ will work for $N^{0,1}_{j,l,n}$. Now we consider the denominator of $n^{0,1}_{j,l,n}(k,k-\ell,\ell)$ near $k=0$, and obtain
\begin{align*}
-j\omega(k)-\omega(k-\ell)+n\omega(\ell)=-j\omega'(0)k-(\omega(-\ell)+\omega'(-\ell)k)+n\omega(\ell)
+\mathcal{O}(k^{2}).
\end{align*}
If $n\neq-1$ this quantity is bounded from below by some $\mathcal{O}(1)$ constant for all $|k|<\delta$. If, on the other hand, $n=-1$, there exists a positive constant $C$ such that
\begin{align}\label{ome01}
|-j\omega(k)-\omega(k-\ell)+n\omega(\ell)|\geq C|k|.
\end{align}
Here, we have used the fact that $\ell\approx \pm k_{0}$ due to the support of $\widehat{\psi}_{\pm1}$ and the fact that $j\omega'(0)\neq\omega'(k_{0})=\mathcal{O}(1)$. Thus, the denominator of $n^{0,1}_{j,l,n}$ gets close to zero only for $n=-1$. Fortunately, in the case of $n=-1$, we have the inequality $|b^{0,1}_{j,l,-1}(k,k-\ell,\ell)|\leq C|k|$ in light of the definition of the form of the nonlinear terms. Indeed, we have the following lemma.
\begin{lemma}
For $j\in\{0,\pm1\}$ and $l\in\{\pm1\}$, there exists a constant $C$ such that $$|b^{0,1}_{j,l,-1}(k,k-\ell,\ell)|\leq C|k|.$$
\end{lemma}
\begin{proof}
Because of the smoothness of kernel we only need to show $b^{0,1}_{j,l,-1}(0,-\ell,\ell)=0$ for $j\in\{0,\pm1\}$ and $l\in\{\pm1\}$. Due to the form of $B^{0,1}_{j,l,-1}$ in \eqref{Bc02} and \eqref{Bc03}, we have
\begin{align*}
B^{0,1}_{0,l,-1}(\psi_{l},\vartheta R^{1}_{-1})=&\frac{(\gamma-1)(\gamma-2)}{q^{2}}\left(\psi_{l}\partial_{x} q R^{1}_{-1}\right)
+\frac{(\gamma-1)(\gamma-2)}{q^{2}}\left(\partial_{x} q\psi_{l}R^{1}_{-1}\right) \ \text{for} \ j=0,\nonumber\\
B^{0,1}_{j,l,-1}(\psi_{l},\vartheta R^{1}_{-1})=&-\frac{\partial_{x}}{2}\left(q\psi_{l}R^{1}_{-1}\right)
-\frac{\partial_{x}}{2}\left(\psi_{l}qR^{1}_{-1}\right)
-\frac{(\gamma-1)(\gamma-2)}{2q^{2}}\left(\psi_{l}\partial_{x}q R^{1}_{-1}\right)\\
&-\frac{(\gamma-1)(\gamma-2)}{2q^{2}}\left(\partial_{x}q\psi_{l}
R^{1}_{-1}\right)
+\frac{j\partial_{x}}{2q}\left(q\psi_{l}qR^{1}_{-1}\right)\\
&+\frac{j(\gamma-2)\partial_{x}}{2q}\left(\psi_{l}
R^{1}_{-1}\right)
-\frac{j\partial_{x}}{2q(1-\partial_{x}^{2})}
\left(\frac{\psi_{l}}{1-\partial_{x}^{2}}
\frac{R^{1}_{-1}}{1-\partial_{x}^{2}}\right) \ \text{for} \ j=\pm1.
\end{align*}
This implies that the kernel $b^{0,1}_{0,l,-1}$ of $B^{0,1}_{0,l,-1}$ at $k=0$ satisfies
\begin{align*}
b^{0,1}_{0,l,-1}(0,-\ell,\ell)=&\frac{(\gamma-1)(\gamma-2)}{\widehat{q}^{2}(k)}(i\ell\widehat{q}(\ell)
+i(k-\ell)\widehat{q}(k-\ell))\Big|_{k=0}\\
=&\frac{(\gamma-1)(\gamma-2)}{\widehat{q}^{2}(0)}(i\ell\widehat{q}(\ell)
+i(-\ell)\widehat{q}(-\ell))\\
=&\frac{(\gamma-1)(\gamma-2)}{\gamma}(i\ell\widehat{q}(\ell)-i\ell\widehat{q}(\ell))\\
=&0.
\end{align*}
Similarly, $b^{0,1}_{j,l,-1}(0,-\ell,\ell)=0$ for $j\in\{\pm1\}$. Here, we have used the fact that $\widehat{q}(k)=\sqrt{\gamma+\frac{1}{1+k^{2}}}$ is an even function.
\end{proof}

Therefore, in the case $n=-1$, there is a cancellation between the numerator and the denominator of $b^{0,1}_{j,l,-1}$; for other values of $n$ the denominator is bounded away from zero. Therefore, there exists a constant $C\geq 0$ such that
\begin{align}\label{n01-2}
|\widehat{\vartheta}(k)n^{0,1}_{j,l,n}(k,k-\ell,\ell)|\leq C
\end{align}
for all $|k|\leq\delta$ and $\ell$ under consideration.

Assume $R^{1}_{n}\in H^{s}$ for some $s>1$, then for any $s'$, there exists a constant $C_{s'}$ such that
\begin{align}\label{N01}
&\|\epsilon\vartheta N^{0,1}_{j,l,n}(\psi_{l},R^{1}_{n})\|_{H^{s'}}
\leq C_{s'}\epsilon\|R^{1}_{n}\|_{H^{s}},\nonumber\\
&\|\epsilon^{2}\vartheta N^{0,1}_{j,l,n}(\psi^{1}_{l},R^{1}_{n})\|_{H^{s'}}
\leq C_{s'}\epsilon^{2}\|R^{1}_{n}\|_{H^{s}}.
\end{align}
In fact, these inequalities in \eqref{N01} hold for $s'=s$ due to the factor of $\widehat{P}^{0}(k)$.  However, 
generally, we cannot assume $C_{s'}\sim\mathcal{O}(\epsilon)$ since $\widehat{\vartheta}^{-1}(k)=\mathcal{O}(\epsilon^{-1})$ for $|k|\leq\delta$.

Secondly, we consider $\epsilon^{2}N^{0,1}_{j,\pm2,n}$ with $j,n\in\{0,\pm1\}$. Proceeding as in subsection 5.1, the kernels $n^{0,1}_{j,\pm2,n}$ of $N^{0,1}_{j,\pm2,n}$ should satisfy
\begin{align}\label{n012}
n^{0,1}_{j,\pm2,n}(k,k-\ell,\ell)=\frac{i\widehat{P}^{0}(k)b^{0,1}_{j,\pm2,n}(k,k-\ell,\ell)}
{-j\omega(k)\mp2\omega(k-\ell)+n\omega(\ell)}\frac{\widehat{\vartheta}(\ell)}{\widehat{\vartheta}(k)},
\end{align}
where $b^{0,1}_{j,\pm2,n}$ are the kernels of $B^{0,1}_{j,\pm2,n}$ and we have used $\partial_{t}\psi_{\pm2j}=\mp2i\omega_{0}\psi_{\pm2j}+\mathcal{O}(\epsilon)$ to approximate the denominator, thanks to the following lemma. 
\begin{lemma}
Fix $s>0$, there exists a constant $C_{s}>0$ such that
\begin{align*}
\|\partial_{t}\psi_{\pm2j}\pm2i\omega_{0}\psi_{\pm2j}\|_{H^{s}}
\leq C_{s}\epsilon.
\end{align*}
\end{lemma}
\begin{proof}
We only prove the case of $\psi_{2j}$, while the case of $\psi_{-2j}$ works analogously by the change $2k_{0}\rightarrow-2k_{0}$. Recall that $\psi_{2j}=A_{2j}(\epsilon(x-c_{g}t,\epsilon^{2}t))e^{2i(k_{0}x-\omega_{0}t)}$ in \eqref{Psi}, then we have
\begin{align*}
\partial_{t}\psi_{2j}+2i\omega_{0}\psi_{2j}
=(\epsilon c_{g}\partial_{X}A_{2j}+\epsilon^{2}\partial_{T}A_{2j})e^{2i(k_{0}x-\omega_{0}t)}.
\end{align*}
Due to the equations \eqref{E2} and the relation \eqref{SE} between $A_{2j}$ and $\widetilde{A}_{2j}$, we complete the proof.
\end{proof}

We use the fact that $\ell\approx \pm 2k_{0}$ due to the support of $\widehat{\psi}_{\pm2j}$ to further approximate the denominator of the kernels $n^{0,1}_{j,\pm2,n}(k,k-\ell,\ell)$ in \eqref{n012} as
\begin{align*}
-j\omega(k)\pm2\omega(k_{0})+n\omega(k\mp2k_{0}).
\end{align*}
Thus, the denominator of $n^{0,1}_{j,\pm2,n}(k,k-\ell,\ell)$ is bounded away from zero for all $|k|\leq\delta$ due to $\omega(2k_{0})\neq\pm2\omega(k_{0})\neq0$. Moreover, since $k$ and $\ell$ are restricted to bounded intervals, the operators $ N^{0,1}_{j,\pm2,n}$ define bounded transforms from $H^{s}$ to itself.

Finally, we consider $\epsilon^{2}N^{0,1}_{j,\pm1,\pm1,n}$. Proceeding analogously as in the case of the bilinear terms we find that the kernels $n^{0,1}_{j,\pm1,\pm1,n}$ of $N^{0,1}_{j,\pm1,\pm1,n}$ should satisfy
\begin{align}\label{nll}
n^{0,1}_{j,\pm1,\pm1,n}(k,k-\ell,\ell-p,p)
=\frac{i\widehat{P}^{0}(k)t^{0,1}_{j,\pm1,\pm1,n}(k,k-\ell,\ell-p,p)}
{-j\omega(k)-\omega(k-\ell)-\omega(\ell-p)+n\omega(p)}
\frac{\widehat{\vartheta}(p)}{\widehat{\vartheta}(k)},
\end{align}
where $t^{0,1}_{j,\pm1,\pm1,n}$ is the kernel of $T^{0,1}_{j,\pm1,\pm1,n}$. Similarly, we use the fact that $k-\ell\approx \pm k_{0}$ and $\ell-p\approx \pm k_{0}$ due to the support of $\widehat{\psi}_{\pm1}$ to further approximate the denominator as
\begin{align*}
-j\omega(k)\mp2\omega(k_{0})+n\omega(k\mp2k_{0}),
\end{align*}
and therefore the denominator of the kernels $n^{0,1}_{j,\pm1,\pm1,n}$ in \eqref{nll} is bounded away from zero for all $|k|\leq\delta$.
Moreover, since $k$, $\ell$ and $p$ are restricted to bounded intervals, the operators $ N^{0,1}_{j,\pm1,\pm1,n}$ define bounded transforms from $H^{s}$ to itself.

In summary, $N^{0,1}_{j,\pm2,n}$ and $N^{0,1}_{j,\pm1,\pm1,n}$ are smooth in the sense that for any $s'$, there exists a constant $C_{s'}$ such that
\begin{align}\label{N012}
&\|\epsilon^{2}N^{0,1}_{j,\pm2,n}(\psi_{\pm2},R^{1}_{n})\|_{H^{s'}}
\leq C_{s'}\epsilon^{2}\|R^{1}_{n}\|_{H^{s}},\\
&\|\epsilon^{2}N^{0,1}_{j,\pm1,\pm1,n}(\psi_{\pm1},\psi_{\pm1},R^{1}_{n})\|_{H^{s'}}
\leq C_{s'}\epsilon^{2}\|R^{1}_{n}\|_{H^{s}},\nonumber
\end{align}
for any $R^{1}_{n}\in H^{s}$ with some $s>1$.
\subsubsection{\textbf{Construction and properties of $N^{1,0}$}}
Recall that $\widehat{\vartheta}_{0}(k)=\widehat{\vartheta}(k)-\epsilon$. Before constructing the normal-form transform $N^{1,0}$, we replace the term $\frac{\epsilon P^{1}}{\vartheta}B^{1,0}_{j,l,n}(\psi_{l},\vartheta R^{0}_{n})$ with $\frac{\epsilon P^{1}}{\vartheta}B^{1,0}_{j,l,n}(\psi_{l},\vartheta_{0}R^{0}_{n})+\frac{\epsilon^{2} P^{1}}{\vartheta}B^{1,0}_{j,l,n}(\psi_{l},R^{0}_{n})$ for $j,n\in\{0,\pm1\}$ and $l\in\{\pm1\}$ in the evolutionary equation \eqref{BR1} for $R^{1}_{j}$. Because $\widehat{\vartheta}^{-1}(k)=\mathcal{O}(1)$ on the support of $\widehat{P}^{1}$, the terms $\frac{\epsilon^{2} P^{1}}{\vartheta}B^{1,0}_{j,l,n}(\psi_{l},R^{0}_{n})$ are of order $\mathcal{O}(\epsilon^{2})$ and need not to be eliminated. By this modification, we can solve the problem caused by resonances at wave numbers $k=\pm k_{0}$ due to the fact $\widehat{\vartheta}_{0}(0)=0$.

Proceeding as in Subsection 5.1, the components $N_{j,l,n}^{1,0}$ should satisfy
\begin{align}\label{N10}
-j&\Omega N_{j,l,n}^{1,0}(\psi_{c},R^{0}_{n})-N_{j,l,n}^{1,0}(\Omega\psi_{l},R^{0}_{n})
+N_{j,l,n}^{1,0}(\psi_{l},n\Omega R^{0}_{n})
=-\frac{P^{1}}{\vartheta}B_{j,l,n}^{1,0}(\psi_{l},\vartheta_{0}R^{0}_{n}).
\end{align}

Before proceeding, we give the subsequent lemma that will simplify the discussion below and help us extract the really dangerous terms from $\frac{P^{1}}{\vartheta}B_{j,l,n}^{1,0}(\psi_{l},\vartheta_{0}R^{0}_{n})$. This lemma mainly uses the strong localization of $\widehat{\psi}_{\pm1}$ near wave numbers $k=\pm k_{0}$.
\begin{lemma}\label{L6}
Fix $p\in\mathbb{R}$ and assume that $\kappa=\kappa(k,k-m,m)\in C(\mathbb{R},\mathbb{C})$. Assume further that $\psi$ has a finitely supported Fourier transform and $R\in H^{s}$. Then, \\
(i)\  if $\kappa$ is Lipschitz with respect to its second argument in some neighborhood of $p\in\mathbb{R}$,  there exists $C_{\psi,\kappa,p}>0$ such that
\begin{equation}
\begin{split}\label{equation56}
\bigg\|\int&\kappa(\cdot,\cdot-m,m)\widehat{\psi}(\frac{\cdot-m-p}{\epsilon})\widehat{R}(m)dm
-\int\kappa(\cdot,p,m)\widehat{\psi}(\frac{\cdot-m-p}{\epsilon})\widehat{R}(m)dm\bigg\|_{H^{s}}\\
&\leq C_{\psi,\kappa,p}\epsilon\|R\|_{H^{s}},
\end{split}
\end{equation}
(ii) if $\kappa$ is globally Lipschitz with respect to its third argument, there exists $D_{\psi,\kappa}>0$ such that
\begin{align}\label{equation57}
\bigg\|\int&\kappa(\cdot,\cdot-m,m)\widehat{\psi} (\frac{\cdot-m-p}{\epsilon})\widehat{R}(m)dm
-\int\kappa(\cdot,\cdot-m,\cdot-p)\widehat{\psi} (\frac{\cdot-m-p}{\epsilon})\widehat{R}(m)dm\bigg\|_{H^{s}}\nonumber\\
&\leq D_{\psi,\kappa}\epsilon\|R\|_{H^{s}}.
\end{align}
\end{lemma}
\begin{proof}
For (i), we have
 \begin{equation*}
\begin{split}
\Big\|\int&\kappa(\cdot,\cdot-m,m)\widehat{\psi}\big(\frac{\cdot-m-p}{\epsilon}\big)\widehat{R}(m)dm
-\int\kappa(\cdot,\cdot-m,\cdot-p)\widehat{\psi}\big(\frac{\cdot-m-p}{\epsilon}\big)\widehat{R}(m)dm\Big\|^{2}_{H^{s}}\\
&=\int\Big(\int(\kappa(k,k-m,m)-\kappa(k,p,m))\widehat{\psi}(\frac{k-m-p}{\epsilon})\widehat{R}(m)dm\Big)^{2}(1+k^{2})^{s}dk\\
&\leq\int\Big(C_{\kappa,p}\int|k-m-p|\widehat{\psi}(\frac{k-m-p}{\epsilon})\widehat{R}(m)dm\Big)^{2}(1+k^{2})^{s}dk\\
&\leq C_{\kappa,p}^{2}\Big(\int(1+\ell)^{s/2}|\ell||\widehat{\psi}(\frac{\ell}{\epsilon})|d\ell\Big)^{2}\|R\|^{2}_{H^{s}}\\
&\leq C_{\psi,\kappa,p}\epsilon^{2}\|R\|^{2}_{H^{s}},
\end{split}
\end{equation*}
thanks to the Young's inequality and the fact that $\widehat{\psi}$ has compact support. The second one (ii) is similar.
\end{proof}

To explain more clearly the nontrivial resonances at the wave numbers $k=\pm k_{0}$, we do not directly compute the kernels $n_{j,l,n}^{1,0}$ of $N_{j,l,n}^{1,0}$ from \eqref{N10}. Instead, we use Lemma \ref{L6} to obtain a modified equation for $N_{j,l,n}^{1,0}$ that will make the normal-form transformation easier to be bounded, only at the expense of adding some additional $\mathcal{O}(\epsilon^{2})$ terms. Specifically, we use Lemma \ref{L6} to make the following changes:\\
$\bullet$  use $\omega(\pm k_{0})$ to replace $\omega(k)$ for $N_{j,l,n}^{1,0}(\Omega\psi_{l},R^{0}_{n})$ in Fourier space;\\
$\bullet$  use $\omega(k\mp k_{0})$ to replace $\omega(\ell)$ for $N_{j,l,n}^{1,0}(\psi_{l},n\Omega R^{0}_{n})$ in Fourier space; and\\
$\bullet$ use $\widehat{\vartheta}_{0}(k\pm k_{0})$ to replace $\widehat{\vartheta}_{0}(\ell)$ for $B_{j,l,n}^{1,0}(\Omega\psi_{l},\vartheta_{0}R^{0}_{n})$ in Fourier space.\\
Inserting these changes into \eqref{N10}, the kernels $n_{j,l,n}^{1,0}$ of $N_{j,l,n}^{1,0}$ should satisfy
\begin{align}\label{n100}
n_{j,\pm1,n}^{1,0}(k)=\frac{i\widehat{P}^{1}(k)b_{j,\pm1,n}^{1,0}(k,\pm k_{0},k\mp k_{0})}
{-j\omega(k)-\omega(\pm k_{0})+n\omega(k\mp k_{0})}\frac{\widehat{\vartheta}_{0}(k\mp k_{0})}
{\widehat{\vartheta}(k)}.
\end{align}
\begin{remark}
We remark that we did not make these changes on $N^{0,1}$ in Subsect. \S\ref{Sect5.3.1}, although it would be much simpler to analyse the kernels of $N^{0,1}$ by applying these changes. The reason is that these changes would make it much more complex in the analysis of the subsequent second normal-form transform, which is necessary due to the fact that $N^{0,1}=\mathcal{O}(\epsilon^{-1})$ for certain wave numbers, in Subsect. \S5.4 below.
\end{remark}

Since the support of $\psi_{\pm1}$ is concentrated around $k=\pm k_{0}$, and due to the projection operator $\widehat{P}^{1}$, the kernels $n_{j,\pm1,n}^{1,0}(k)$ in \eqref{n100} only need to be analyzed for $|k-\ell\mp k_{0}|<\delta$, $|k|>\delta$ and $|\ell|<\delta$. We now consider the possible resonances in the denominator of \eqref{n100}:
\begin{itemize}
  \item $k=0$: Since $\widehat{P}^{1}(k)=0$ for $|k|<\delta$, this resonance does not play a role in the analysis of $N^{1,0}$;
  \item $k=\pm k_{0}$: The kernels $n_{j,\pm1,n}^{1,0}(k)$ have resonances at wave numbers $k=\pm k_{0}$ if $j=-1$ and at $k=\mp k_{0}$ if $j=1$, for $n=0$. Moreover, since $j\omega'(\pm k_{0})=\mathcal{O}(1)$ for $j\in\{\pm1\}$, we can bound the denominator as
      \begin{align}\label{n10}
      |-j\omega(k)-\omega(\pm k_{0})+n\omega(k\mp k_{0})|\geq C|k\mp k_{0}|.
      \end{align}
      The resonances $k=\pm k_{0}$ can be offset by the fact that $|\widehat{\vartheta}_{0}(k\mp k_{0})|\leq C|k\mp k_{0}|$ and hence $n_{j,\pm1,n}^{1,0}(k)$ is also well-defined at $k=\pm k_{0}$ with an $\mathcal{O}(1)$ bound on its size.
  \item There are no other resonances for the normal-form transform $N^{1,0}$.
\end{itemize}

Now we can summarize the results of the first normal-form transform as follows.
\begin{proposition}\label{P3}
Let $\widetilde{R}^{0}$ and $\widetilde{R}^{1}$ be defined by \eqref{N0110}. Then
this transform maps $(R_{j}^{0},R_{j}^{1})\in H^{s}\times H^{s}$ into $(\widetilde{R}_{j}^{0},\widetilde{R}_{j}^{1})\in H^{s}\times H^{s}$ for all $s>0$ and is invertible on its range. Furthermore, if we write the inverse transformations as
\begin{equation*}
\begin{split}
R_{j}^{0}=\widetilde{R}_{j}^{0}+\epsilon\mathcal{N}^{-1}_{0,1}
(\widetilde{R}_{j}^{0},\widetilde{R}_{j}^{1}), \
R_{j}^{1}=\widetilde{R}_{j}^{1}+\epsilon\mathcal{N}^{-1}_{1,0}
(\widetilde{R}_{j}^{0},\widetilde{R}_{j}^{1}),
\end{split}
\end{equation*}
then there exist constants $C_{0},C_{1}$ such that the inverse transformations satisfy the estimates
\begin{equation*}
\begin{split}
\|\epsilon\mathcal{N}^{-1}_{0,1}(\widetilde{R}_{j}^{0},\widetilde{R}_{j}^{1})\|_{H^{s}}\leq C_{0}(\|\widetilde{R}_{j}^{0}\|_{H^{s}}+\|\widetilde{R}_{j}^{1}\|_{H^{s}}),\\
\|\epsilon\mathcal{N}^{-1}_{1,0}(\widetilde{R}_{j}^{0},\widetilde{R}_{j}^{1})\|_{H^{s}}\leq C_{1}\epsilon(\|\widetilde{R}_{j}^{0}\|_{H^{s}}+\|\widetilde{R}_{j}^{1}\|_{H^{s}}).
\end{split}
\end{equation*}
Finally, if $(R_{j}^{0},R_{j}^{1})$ satisfy \eqref{N0110}, then $(\widetilde{R}_{j}^{0},\widetilde{R}_{j}^{1})$ satisfy
\begin{equation}
\begin{split}\label{Rj00}
\partial_{t}\widetilde{R}_{j}^{0}=&j\Omega\widetilde{R}_{j}^{0}+
\epsilon^{2} \sum_{m,n\in\{0,\pm1\}}\sum_{l_{1},l_{2}\in\{\pm1\}}N^{0,1}_{j,l_{1},m}\Big(\psi_{l_{1}},\frac{P^{1}}{\vartheta}
B^{1,1}_{m,l_{2},n}(\psi_{l_{2}},\vartheta \widetilde{R}^{1}_{n}\Big)+\epsilon^{2}\widetilde{F}_{j}^{0},\\
\partial_{t}\widetilde{R}^{1}_{j}=&j\Omega \widetilde{R}^{1}_{j}
+\frac{\epsilon P^{1}}{\vartheta}\sum_{n\in\{0,\pm1\}}\sum_{l\in\{\pm1\}} B^{1,1}_{j,l,n}(\psi_{l},\vartheta \widetilde{R}^{1}_{n})+\epsilon^{2}\widetilde{F}_{j}^{1},
\end{split}
\end{equation}
for $j\in\{0,\pm1\}$, where $\epsilon^{2}\widetilde{F}^{r}=\epsilon^{2}(\widetilde{F}_{0}^{r}, \widetilde{F}_{1}^{r}, \widetilde{F}_{-1}^{r})$ for $r=0,1$. Here, $\epsilon^{2}\widetilde{F}^{0}$ does not lose any derivatives due to the support of $P^{0}$, and $\epsilon^{2}\widetilde{F}^{1}$ denotes a collection of terms whose $H^{s-1}$-norms are bounded by $C \epsilon^{2}$ for $(\widetilde{R}_{j}^{0},\widetilde{R}_{j}^{1})\in H^{s}\times H^{s}$.
\end{proposition}
\begin{proof}
The invertibility of the transform \eqref{N0110} in this case results from a simple application of the Neumann series since there is no loss of smoothness. The structure of the equations \eqref{Rj00} follows immediately using $R_{j}^{1}=\widetilde{R}^{1}_{j}+\mathcal{O}(\epsilon)$ for $j\in\{0, \ \pm1\}$ and Lemma \ref{L1}, which implies that the residual can be put into $\epsilon^{2}\widetilde{F}_{j}^{r}$ with $r=0,1$, just noting that $\Psi$ has compact support in Fourier space.
\end{proof}

\subsection{\textbf{The second normal-form transform for the low frequency terms}}
Due to the fact that $\widehat{\vartheta}^{-1}(k)=\mathcal{O}(\epsilon^{-1})$ for $|k|\leq\delta$ and the estimate for $N^{0,1}$, we have $N^{0,1}=\mathcal{O}(\epsilon^{-1})$ for $|k|\leq\delta$. Therefore, there are still $\mathcal{O}(\epsilon)$ terms in the first equation of \eqref{Rj00}. We now construct a second normal-form transform to remove these terms. Before doing so we first extract these really dangerous $\mathcal{O}(\epsilon)$ terms by analyzing the offending terms in more detail. The terms in Fourier space can be written as
\begin{align}\label{NB01}
&\epsilon^{2}\widehat{N}^{0,1}_{j,l_{1},m}\Big(\psi_{l_{1}},\frac{P^{1}}{\vartheta}
B^{1,1}_{m,l_{2},n}(\psi_{l_{2}},\vartheta \widetilde{R}^{1}_{n}\Big)(k)\\
&=\epsilon^{2}\int n^{0,1}_{j,l_{1},m}(k,k-\ell,\ell)\widehat{\psi}_{l_{1}}(k-\ell)
\frac{\widehat{P}^{1}(\ell)}{\widehat{\vartheta}(\ell)}\left(
\int b^{1,1}_{m,l_{2},n}(\ell,\ell-p,p)\widehat{\psi}_{l_{2}}(\ell-p)
\widehat{\vartheta}(p)
\widehat{\widetilde{R}^{1}_{n}}(p)dp\right)d\ell, \nonumber
\end{align}
where we recall that
\begin{equation*}
\begin{split}
n^{0,1}_{j,l_{1},m}(k,k-\ell,\ell)=\frac{i\widehat{P}^{0}(k)b^{0,1}_{j,l_{1},m}(k,k-\ell,\ell)}
{-j\omega(k)-\omega(k-\ell)+m\omega(\ell)}\frac{\widehat{\vartheta}(\ell)}{\widehat{\vartheta}(k)}.
\end{split}
\end{equation*}
As was done for $N^{1,0}$, we can use Lemma \ref{L6} to simplify \eqref{NB01} to obtain
\begin{equation*}
\begin{split}
&\epsilon^{2}\widehat{N}^{0,1}_{j,l_{1},m}\Big(\psi_{l_{1}},\frac{P^{1}}{\vartheta}(\cdot-l_{1}k_{0})
B^{1,1}_{m,l_{2},n}(\psi_{l_{2}},\vartheta(\cdot-l_{2}k_{0}) \widetilde{R}^{1}_{n}\Big)(k)\\
=&\epsilon^{2}\int \tilde{n}^{0,1}_{j,l_{1},m}(k)\widehat{\psi}_{l_{1}}(k-\ell)
\frac{\widehat{P}^{1}(k-l_{1}k_{0})}{\widehat{\vartheta}(k-l_{1}k_{0})}\\
&\times\left(
\int b^{1,1}_{m,l_{2},n}(k-l_{1}k_{0},k-l_{1}k_{0}-p,p)\widehat{\psi}_{l_{2}}(\ell-p)
\widehat{\vartheta}(k-(l_{1}+l_{2})k_{0})
\widehat{\widetilde{R}^{1}_{n}}(p)dp\right)d\ell
+\epsilon^{2}\mathcal{F}^{0,1}_{j},
\end{split}
\end{equation*}
where $\epsilon^{2}\mathcal{F}^{0,1}=\epsilon^{2}(\mathcal{F}^{0,1}_{0},\mathcal{F}^{0,1}_{1},
\mathcal{F}^{0,1}_{-1})$ satisfies
\begin{equation*}
\begin{split}
\|\epsilon^{2}\mathcal{F}^{0,1}\|_{H^{s}}\leq C \epsilon^{2}\|\widetilde{R}\|_{H^{s}}
\end{split}
\end{equation*}
for some constant $C>0$. 
Furthermore, we apply Lemma \ref{L6} to write the abbreviative form for the kernels $n^{0,1}_{j,l_{1},m}$ as
\begin{equation*}
\begin{split}
\widetilde{n}^{0,1}_{j,l_{1},m}(k)=\frac{i\widehat{P}^{0}(k)b^{0,1}_{j,l_{1},m}(k,l_{1}k_{0},k-l_{1}k_{0})}
{-j\omega(k)-\omega(l_{1}k_{0})+m\omega(k-l_{1}k_{0})}\frac{1}{\widehat{\vartheta}(k)}.
\end{split}
\end{equation*}
Under these modifications, we will show that all terms of the form \eqref{NB01} with $l_{1}=-l_{2}$ are actually of order $\mathcal{O}(\epsilon^{2})$ and hence can be put into $\epsilon^{2}\widetilde{F}_{j}^{0}$ in \eqref{Rj00}.

\begin{lemma}\label{L7}
There exists $C>0$ such that
\begin{equation*}
\begin{split}
&\left\|\epsilon^{2}N^{0,1}_{j,1,m}\Big(\psi_{1},\frac{P^{1}}{\vartheta}
B^{1,1}_{m,-1,n}(\psi_{-1},\vartheta \widetilde{R}^{1}_{n}\Big)\right\|_{H^{s}}
\leq C\epsilon^{2}\|\widetilde{R}^{1}_{n}\|_{H^{s}},\\
&\left\|\epsilon^{2}N^{0,1}_{j,-1,m}\Big(\psi_{-1},\frac{P^{1}}{\vartheta}
B^{1,1}_{m,1,n}(\psi_{1},\vartheta \widetilde{R}^{1}_{n}\Big)\right\|_{H^{s}}
\leq C\epsilon^{2}\|\widetilde{R}^{1}_{n}\|_{H^{s}}.
\end{split}
\end{equation*}
\end{lemma}
\begin{proof}
Since the factor $\widehat{P}^{0}(k)$ appears in $\widehat{N}^{0,1}_{j,\pm1,m}$, the integral over $k$ that occurs in the $H^{s}$-norm runs only over $\{|k|<\delta\}$. Thus, we only need to bound the maximum of the kernel to bound the $H^{s}$-norm of these terms. The modified kernel of the first term in Lemma \ref{L7} satisfies
\begin{equation}
\begin{split}\label{nb01}
\epsilon^{2}\widetilde{n}^{0,1}_{j,1,m}(k)\widehat{P}^{1}(k-k_{0})
b^{1,1}_{m,-1,n}(k-k_{0},k-k_{0}-p,p)
\widehat{\vartheta}(k).
\end{split}
\end{equation}
Recalling that  $\widehat{\vartheta}(k)\widetilde{n}^{0,1}_{j,1,m}(k)$ is $\mathcal{O}(1)$ bounded in light of \eqref{N01} and all other terms in \eqref{nb01} are $\mathcal{O}(1)$ bounded for $|k|<\delta$, we have an $\mathcal{O}(\epsilon^{2})$ bound for the kernel \eqref{nb01}. We can obtain the estimate for the second term in Lemma \ref{L7} similarly.
\end{proof}

Lemma \ref{L7} implies that we do not need to eliminate the terms from \eqref{NB01} with $l_{1}=-l_{2}$ by the norm-form transform. Now we focus on the terms from \eqref{NB01} with $l_{1}=l_{2}$. Modify the kernels of these terms by using Lemma \ref{L6} as follows
\begin{equation*}
\begin{split}
\epsilon^{2}\widetilde{n}^{0,1}_{j,\pm1,m}(k)\widehat{P}^{1}(k\mp k_{0})
b^{1,1}_{m,-1,n}(k\mp k_{0},k\mp k_{0}-p,p)
\widehat{\vartheta}(k\mp2k_{0}).
\end{split}
\end{equation*}
Different from the terms considered in Lemma \ref{L7}, $\widehat{\vartheta}(k)$ does not occur in this expression. Thus, we need a second normal-form transform to eliminate these terms from \eqref{NB01}. 

We look for a transform of the form
\begin{equation}
\begin{split}\label{mathR}
\mathcal{R}_{j}^{0}&=\widetilde{R}_{j}^{0}+\epsilon \sum_{m,n\in\{0,\pm1\}}\left(D_{j,1,m}^{0,1,+}(\psi_{1},\psi_{1},\widetilde{R}_{n}^{1})+ D_{j,-1,m}^{0,1,-}(\psi_{-1},\psi_{-1},\widetilde{R}_{n}^{1})\right),\\
\mathcal{R}_{j}^{1}&=\widetilde{R}_{j}^{1}.
\end{split}
\end{equation}
Inserting this transform \eqref{mathR} to \eqref{Rj00} for $\mathcal{R}_{j}^{0}$, we find that, just as in Subsect. \S\ref{Sect5.1}, some $\mathcal{O}(\epsilon)$ terms in \eqref{Rj00} will be eliminated if $D_{j}^{0,1,+}$ satisfies
\begin{equation*}
\begin{split}
-j\Omega D_{j,1,m}^{0,1,+}(\psi_{1},\psi_{1},\widetilde{R}_{n}^{1})
-D_{j,1,m}^{0,1,+}(\Omega\psi_{1},\psi_{1},\widetilde{R}_{n}^{1})
-D_{j,1,m}^{0,1}(\psi_{1},\Omega\psi_{1},\widetilde{R}_{n}^{1})\\
+D_{j,1,m}^{0,1,+}(\Omega\psi_{1},\psi_{1},n\Omega\widetilde{R}_{n}^{1})
+N^{0,1}_{j,1,m}\Big(\psi_{1},\frac{P^{1}}{\vartheta}
B^{1,1}_{m,1,n}(\psi_{1},\vartheta \widetilde{R}^{1}_{n}\Big)=0.
\end{split}
\end{equation*}
This means we have to choose
\begin{equation}
\begin{split}\label{D+}
&\epsilon\widehat{D}_{j,1,m}^{0,1,+}(\psi_{1},\psi_{1},\widetilde{R}_{n}^{1})(k)\\
&=\epsilon^{2}\int\widehat{\widetilde{n}}_{j,1,m}^{0,1}(k)\widehat{\psi}_{1}(k-\ell)
\frac{\widehat{P}^{1}(k-k_{0})}{\widehat{\vartheta}(k-k_{0})}\\
&\times\left(
\int \frac{b^{1,1}_{m,1,n}(k-k_{0},k-k_{0}-p,p)\widehat{\psi}_{l_{2}}(\ell-p)
\widehat{\vartheta}(k-2k_{0})
\widehat{\widetilde{R}^{1}_{n}}(p)}{-j\omega(k)-2\omega(k_{0}) +n\omega(k-2k_{0})}dp\right)d\ell.
\end{split}
\end{equation}
Here we have used the fact that $k\approx0$ and $k-\ell\approx\ell-p\approx k_{0}$ due to the support of $\widehat{\psi}_{1}$ and $\widehat{P}^{0}$ in $\widehat{\tilde{n}}_{j,1,m}^{0,1}(k)$, so we have $p\approx-2k_{0}$. Therefore
\begin{equation*}
\begin{split}
-j\omega(k)-2\omega(k_{0})+n\omega(k-2k_{0})\approx-2\omega(k_{0})-n\omega(2k_{0})\neq0.
\end{split}
\end{equation*}
Regardless of the value of $j$ and $n$, the denominator of this expression is bounded strictly away from zero. The numerator in this expression is $\mathcal{O}(\epsilon)$ and hence the mapping $\epsilon D_{j,1,m}^{0,1,+}$ is $\mathcal{O}(\epsilon)$-bounded. In a similar fashion, $D_{j,-1,m}^{0,1,-}$ can be constructed and estimated. Therefore, the normal-form transform \eqref{mathR} is well-defined and invertible. We have the following lemma.
\begin{lemma}\label{L11}
If
\begin{equation*}
\begin{split}
\mathcal{R}_{j}^{0}&=\widetilde{R}_{j}^{0}+\epsilon \sum_{m,n\in\{0,\pm1\}}\left(D_{j,1,m}^{0,1,+}(\psi_{1},\psi_{1},\widetilde{R}_{n}^{1})+ D_{j,-1,m}^{0,1,-}(\psi_{-1},\psi_{-1},\widetilde{R}_{n}^{1})\right),
\end{split}
\end{equation*}
with $\epsilon D_{j,\pm1,m}^{0,1,\pm}$ defined as in \eqref{D+}, then for any $s>1$ there exists $C>0$ such that
\begin{equation*}
\begin{split}
\|D_{j,\pm1,m}^{0,1,\pm}(\psi_{\pm1},\psi_{\pm1},\widetilde{R}_{n}^{1})\|_{H^{s}}
\leq C \epsilon\|\widetilde{R}_{n}^{1}\|_{H^{s}}.
\end{split}
\end{equation*}
\end{lemma}
Note that there is no loss of smoothness in this transform due to the factor $\widehat{P}^{0}$ in \eqref{D+} via $\widehat{\tilde{n}}_{j,1,m}^{0,1}(k)$. Now, just as in Proposition \ref{P3} we have
\begin{proposition}\label{P4}
Fix $s>1$. Suppose $(\widetilde{R}_{j}^{0},\widetilde{R}_{j}^{1})$ satisfies \eqref{Rj00} and define $(\mathcal{R}_{j}^{0},\mathcal{R}_{j}^{1})$ via the transform \eqref{mathR}. Then there exists an $\epsilon_{0}>0$ such that for all $\epsilon<\epsilon_{0}$, the transform \eqref{mathR} is invertible in $H^{s}$. Furthermore, $(\mathcal{R}_{j}^{0},\mathcal{R}_{j}^{1})$ satisfies the equations
\begin{equation}
\begin{split}\label{RJ}
\partial_{t}\mathcal{R}_{j}^{0}=&j\Omega\mathcal{R}_{j}^{0}+\epsilon^{2}\mathcal{F}_{j}^{0},\\
\partial_{t}\mathcal{R}_{j}^{1}=&j\Omega\mathcal{R}_{j}^{1}
+\sum_{n\in\{0,\pm1\}}\left(\epsilon \sum_{l\in\{\pm1\}} B^{1,1}_{j,l,n}(\psi_{l}, \mathcal{R}^{1}_{n})+\epsilon^{2}\mathcal{F}_{j,n}^{1}\right),
\end{split}
\end{equation}
for $j\in\{0, \ \pm1\}$ (see \eqref{Bc03} for $B^{1,1}_{j,l,n}$). Here, $\epsilon^{2}\mathcal{F}^{0}$ does not lose any derivatives due to the factor $P^{0}$, and $\epsilon^{2}\mathcal{F}^{1}$ denotes a collection of terms whose $H^{s-1}$-norms are bounded by $C \epsilon^{2}$ for $(\mathcal{R}_{j}^{0},\mathcal{R}_{j}^{1})\in H^{s}\times H^{s}$. Indeed, we have
\begin{equation}
\begin{split}\label{J68}
\epsilon^{2}\|\mathcal{F}_{j}^{0}\|_{H^{s'}}\leq C \epsilon^{2}(1+\|\mathcal{R}\|_{H^{s}}+\epsilon^{3/2}\|\mathcal{R}\|^{2}_{H^{s}}
+\epsilon^{3}\|\mathcal{R}\|^{3}_{H^{s}}),\\
\epsilon^{2}\|\mathcal{F}_{j,n}^{1}\|_{H^{s-1}}\leq C \epsilon^{2}(1+\|\mathcal{R}\|_{H^{s}}+\epsilon^{3/2}\|\mathcal{R}\|^{2}_{H^{s}}
+\epsilon^{3}\|\mathcal{R}\|^{3}_{H^{s}}),\\
\end{split}
\end{equation}
for any $s,s'>1$. Here, we have ignored $\vartheta$ and $P^{1}$ since $\vartheta(k)=1$ in Fourier space.
\end{proposition}
\begin{proof}
The invertibility of the transform in this case results from a simple application of the Neumann series since there is no loss of smoothness. The equation for $\mathcal{R}_{j}^{0}$ and $\mathcal{R}_{j}^{1}$ follow in the same way the equations for $\widetilde{R}_{j}^{0}$ and $\widetilde{R}_{j}^{1}$ were derived in the proof of Proposition \ref{P3}.
\end{proof}

Combining \eqref{Rj} with \eqref{Bc02}-\eqref{Bc03}, we obtain that $B^{1,1}_{j,l,n}$ in the second equation of \eqref{RJ} has the form of
\begin{align}\label{B0n}
\sum_{n\in\{0,\pm1\}}B^{1,1}_{0,l,n}(\psi_{l},\mathcal{R}^{1}_{n})
:=&-\frac{1}{q^{2}}\left(q\psi_{l}\partial_{x}q^{2} \mathcal{R}^{1}_{0}\right)-\frac{(\gamma-1)(\gamma-2)}{q^{2}} \left(\psi_{l}\partial_{x} q (\mathcal{R}^{1}_{1}-\mathcal{R}^{1}_{-1})\right)\nonumber\\
&+\frac{\gamma-1}{q^{2}}\left(\partial_{x} q\psi_{l}((\gamma-2-q^{2}) \mathcal{R}^{1}_{0}+(\gamma-2)(\mathcal{R}^{1}_{1}+\mathcal{R}^{1}_{-1}) \right),
\end{align}
\begin{align}\label{Bjn}
\sum_{n\in\{0,\pm1\}}B^{1,1}_{j,l,n}(\psi_{l},\mathcal{R}^{1}_{n})
:=&\frac{1}{2q^{2}}\left(q\psi_{l}
\partial_{x}q^{2}\mathcal{R}^{1}_{0}\right)
-\frac{\partial_{x}}{2}\left(q\psi_{l}(\mathcal{R}^{1}_{0} +\mathcal{R}^{1}_{1}
+\mathcal{R}^{1}_{-1})\right)\nonumber\\
&+\frac{\partial_{x}}{2}\left(\psi_{l}q(\mathcal{R}^{1}_{1} -\mathcal{R}^{1}_{-1})\right)
+\frac{(\gamma-1)(\gamma-2)}{2q^{2}}\left(\psi_{l}\partial_{x}q (\mathcal{R}^{1}_{1}-\mathcal{R}^{1}_{-1})\right)\nonumber\\
&-\frac{\gamma-1}{2q^{2}}\left(\partial_{x}q\psi_{l}
(\gamma-2-q^{2})\mathcal{R}^{1}_{0}+(\gamma-2) (\mathcal{R}^{1}_{1}+\mathcal{R}^{1}_{-1})\right)
\nonumber\\
&-\frac{j\partial_{x}}{2q}\left(q\psi_{l}q(\mathcal{R}^{1}_{1} -\mathcal{R}^{1}_{-1})\right)
+\frac{j(\gamma-2)}{2 q}\left(\psi_{l}\partial_{x}(\mathcal{R}^{1}_{0} +\mathcal{R}^{1}_{1}+\mathcal{R}^{1}_{-1})\right)\nonumber\\
&+\frac{j}{2q}\left(\partial_{x}\psi_{l}
(\gamma-2-q^{2})\mathcal{R}^{1}_{0}+(\gamma-2) (\mathcal{R}^{1}_{1}+\mathcal{R}^{1}_{-1})\right)\\
&-\frac{j\partial_{x}}{2q(1-\partial_{x}^{2})}
\left(\frac{\psi_{l}}{1-\partial_{x}^{2}}
\frac{\mathcal{R}^{1}_{0}+\mathcal{R}^{1}_{1} +\mathcal{R}^{1}_{-1}}{1-\partial_{x}^{2}}\right),
 \ \ \text{for} \ \ j\in\{\pm1\}.\nonumber
\end{align}
We extract those terms that lose derivatives on $\mathcal{R}^{1}$ from $\mathcal{F}_{j}^{1}$ in the second equation of \eqref{RJ} in light of \eqref{Rj}. For $j=0$ we have
\begin{align}\label{F00}
\epsilon^{2}\sum_{n\in\{0,\pm1\}}\mathcal{F}_{0,n}^{1}&=
\frac{\epsilon^{2}}{q^{2}}\left(q\varphi_{1}\partial_{x}q^{2}\mathcal{R}^{1}_{0}\right)
+\frac{\epsilon^{2}}{q^{2}}\left(\partial_{x}q^{2}\psi_{p_{0}}
q(\mathcal{R}^{1}_{1}-\mathcal{R}^{1}_{-1})\right)\nonumber\\
&-\frac{\epsilon^{2}(\gamma-1)}{q^{2}}\left(\varphi_{2}\partial_{x}q
(\mathcal{R}^{1}_{1}-\mathcal{R}^{1}_{-1})\right)\nonumber\\
&-\frac{\epsilon^{2}(\gamma-1)}{q^{2}}\left(\partial_{x}q\varphi_{1}
\left((\gamma-2-q^{2})\mathcal{R}^{1}_{0}+(\gamma-2)\mathcal{R}^{1}_{1}
+(\gamma-2)\mathcal{R}^{1}_{-1})\right)\right)\nonumber\\
&-\frac{\epsilon^{\beta}(\gamma-1)}{q^{2}}
\left(\Big((\gamma-2-q^{2})\mathcal{R}^{1}_{0}+(\gamma-2)\mathcal{R}^{1}_{1}
+(\gamma-2)\mathcal{R}^{1}_{-1})\Big)
\partial_{x}q(\mathcal{R}^{1}_{1}-\mathcal{R}^{1}_{-1})\right)\nonumber\\
&+\frac{\epsilon^{\beta}}{q^{2}}\left(q(\mathcal{R}^{1}_{1}
-\mathcal{R}^{1}_{-1})\partial_{x}q^{2}\mathcal{R}^{1}_{0}\right)
+\epsilon^{-\beta}Res_{0}^{1}(\epsilon\Psi)
+\epsilon^{2}\mathcal{K}^{1}_{0},
\end{align}
and for $j\in\{\pm1\}$ we have
\begin{align}\label{F01}
\epsilon^{2}\sum_{n\in\{0,\pm1\}}\mathcal{F}_{j,n}^{1}=
&-\frac{\epsilon^{2}}{2q^{2}}\left(q\varphi_{1}\partial_{x}q^{2}\mathcal{R}^{1}_{0}\right)
-\frac{\epsilon^{2}}{2q^{2}}\left(\partial_{x}q^{2}\psi_{p_{0}}q(\mathcal{R}^{1}_{1}
-\mathcal{R}^{1}_{-1})\right)\nonumber\\
&+\frac{\epsilon^{2}\partial_{x}}{2}
\left(q\varphi_{1}(\mathcal{R}^{1}_{0}+\mathcal{R}^{1}_{1}+\mathcal{R}^{1}_{-1})\right)\nonumber\\
&+\frac{\epsilon^{2}(\gamma-1)}{2q^{2}}\left(\partial_{x}q\varphi_{1}
((\gamma-2-q^{2})\mathcal{R}^{1}_{0}+(\gamma-2)\mathcal{R}^{1}_{1}+(\gamma-2)\mathcal{R}^{1}_{-1})\right)\nonumber\\
&+\frac{\epsilon^{2}\partial_{x}}{2}\left(\varphi_{3}q(\mathcal{R}^{1}_{1}-\mathcal{R}^{1}_{-1})\right)
+\frac{\epsilon^{2}(\gamma-1)}{2 q^{2}}\left(\varphi_{2}\partial_{x}q(\mathcal{R}^{1}_{1}-\mathcal{R}^{1}_{-1})\right)\nonumber\\
&+\frac{\epsilon^{2}j}{2q}\left(\partial_{x}\varphi_{3}
((\gamma-2-q^{2})\mathcal{R}^{1}_{0}+(\gamma-2)\mathcal{R}^{1}_{1}+(\gamma-2)\mathcal{R}^{1}_{-1})\right)\nonumber\\
&+\frac{\epsilon^{2}j\partial_{x}}{2q}\left(q\varphi_{1}q(\mathcal{R}^{1}_{1}-\mathcal{R}^{1}_{-1}\right))
+\frac{\epsilon^{2}j}{2 q}\left(\varphi_{2}\partial_{x}
(\mathcal{R}^{1}_{0}+\mathcal{R}^{1}_{1}+\mathcal{R}^{1}_{-1})\right)\nonumber\\
&-\frac{\epsilon^{\beta}}{2q^{2}}\left(q(\mathcal{R}^{1}_{1}-\mathcal{R}^{1}_{-1})
\partial_{x}q^{2}\mathcal{R}^{1}_{0}\right)
+\frac{\epsilon^{\beta}\partial_{x}}{2}\left(q(\mathcal{R}^{1}_{1}-\mathcal{R}^{1}_{-1})
(\mathcal{R}^{1}_{0}+\mathcal{R}^{1}_{1}+\mathcal{R}^{1}_{-1})\right)\nonumber\\
&+\frac{\epsilon^{\beta}(\gamma-1)}{2q^{2}}\left((
(\gamma-2-q^{2})\mathcal{R}^{1}_{0}+(\gamma-2)\mathcal{R}^{1}_{1}+(\gamma-2)\mathcal{R}^{1}_{-1})
\partial_{x}q(\mathcal{R}^{1}_{1}-\mathcal{R}^{1}_{-1})\right)\nonumber\\
&+\frac{\epsilon^{\beta}j\partial_{x}}{4 q}(q(\mathcal{R}^{1}_{1}-\mathcal{R}^{1}_{-1}))^{2}\nonumber\\
&+\frac{\epsilon^{\beta}j}{2q}\left(
((\gamma-2-q^{2})\mathcal{R}^{1}_{0}+(\gamma-2)\mathcal{R}^{1}_{1}+(\gamma-2)\mathcal{R}^{1}_{-1})
\partial_{x}(\mathcal{R}^{1}_{0}+\mathcal{R}^{1}_{-1}+\mathcal{R}^{1}_{1})\right)\nonumber\\
&+\epsilon^{2}\mathcal{H}(\mathcal{R}^{1})
+\epsilon^{-\beta}Res_{j}^{1}(\epsilon\Psi)
+\epsilon^{2}\mathcal{K}^{1}_{j},
\end{align}
where $\mathcal{K}^{1}_{j}$ with $j\in\{0,\pm1\}$ satisfies
\begin{equation*}
\begin{split}
\epsilon^{2}\|\mathcal{K}_{j}^{1}\|_{H^{s+1}}\leq C \epsilon^{2}(1+\|\mathcal{R}\|_{H^{s}}+\epsilon^{3/2}\|\mathcal{R}\|^{2}_{H^{s}}
+\epsilon^{3}\|\mathcal{R}\|^{3}_{H^{s}}),
\end{split}
\end{equation*}
for any $s\geq1$.

For the final system \eqref{RJ} of the error $\mathcal{R}$ with \eqref{B0n}-\eqref{F01} obtained after making above twice normal-form transforms, only the high frequency terms $\epsilon B^{1,1}_{j,l,n}(\psi_{l}, \mathcal{R}^{1}_{n})$ with $j,n\in\{0,\pm1\}$ and $l\in\{\pm1\}$ of order $\mathcal{O}(\epsilon)$ left in the evolutionary equation for $\mathcal{R}^{1}_{j}$ with $j\in\{0, \ \pm1\}$. However, due to the problem caused by loss of derivatives, we do not eliminate these terms directly but instead construct a modified energy by using the corresponding normal-form transforms in the next section. After doing so, we can not only transform the high frequency $\mathcal{O}(\epsilon)$ terms to $\mathcal{O}(\epsilon^{2})$ terms, but also deal with the loss of regularity and obtain a uniform estimate for $\mathcal{R}$ over the desired long time scale $\mathcal{O}(\epsilon^{-2})$ in Sobolev spaces.

By taking advantage of the structure of the system \eqref{RJ} with \eqref{B0n}-\eqref{F01}, we summarize the above analysis in the following lemma.


\begin{lemma}\label{RR}
We have
\begin{align}\label{+++}
\partial_{t}(\mathcal{R}^{1}_{0}
+\mathcal{R}^{1}_{1}+\mathcal{R}^{1}_{-1})
=&\Omega(\mathcal{R}^{1}_{1}-\mathcal{R}^{1}_{-1})
+\epsilon\phi_{1}\partial_{x}
(\mathcal{R}^{1}_{0}+\mathcal{R}^{1}_{1}+\mathcal{R}^{1}_{-1})
+\epsilon\phi_{3}\partial_{x}
q(\mathcal{R}^{1}_{1}-\mathcal{R}^{1}_{-1})\nonumber\\
&+\epsilon^{-\beta}(Res^{1}_{0}(\epsilon\Psi)+Res^{1}_{1}(\epsilon\Psi)+Res^{1}_{-1}(\epsilon\Psi))
+\epsilon\mathcal{L}^{1},
\end{align}
\begin{equation}
\begin{split}\label{++}
\partial_{t}(\mathcal{R}^{1}_{1}+\mathcal{R}^{1}_{-1})
=&\Omega(\mathcal{R}^{1}_{1}-\mathcal{R}^{1}_{-1})
+\epsilon\phi_{1}\partial_{x}
(\mathcal{R}^{1}_{1}+\mathcal{R}^{1}_{-1})\big)
+\epsilon\phi_{4}\partial_{x}
q(\mathcal{R}^{1}_{1}-\mathcal{R}^{1}_{-1}))\\
&+\epsilon^{-\beta}(Res^{1}_{1}(\epsilon\Psi)+Res^{1}_{-1}(\epsilon\Psi))
+\epsilon\mathcal{L}^{2},
\end{split}
\end{equation}
and
\begin{equation}
\begin{split}\label{+-}
\partial_{t}(\mathcal{R}^{1}_{1}-\mathcal{R}^{1}_{-1})
=&\Omega(\mathcal{R}^{1}_{1}+\mathcal{R}^{1}_{-1})
+\epsilon\phi_{1}\partial_{x}
(\mathcal{R}^{1}_{1}-\mathcal{R}^{1}_{-1})\big)
+\epsilon\phi_{5}\partial_{x}
(\mathcal{R}^{0}_{1}+\mathcal{R}^{1}_{1}+\mathcal{R}^{1}_{-1}))\\
&+\epsilon^{-\beta}(Res^{1}_{1}(\epsilon\Psi)-Res^{1}_{-1}(\epsilon\Psi))
+\epsilon\mathcal{L}^{3},
\end{split}
\end{equation}
where
\begin{align}\label{phi}
\phi_{1}:&=q(-\psi_{l}+\epsilon\varphi_{1}
+\epsilon^{\beta-1}(\mathcal{R}^{1}_{1}-\mathcal{R}^{1}_{-1})),\nonumber\\
\phi_{2}:&=(\gamma-2)\psi_{l}
+\epsilon\varphi_{2}+\epsilon^{\beta-1}((\gamma-2-q^{2})\mathcal{R}^{1}_{0}
+(\gamma-2)(\mathcal{R}^{1}_{1}+\mathcal{R}^{1}_{-1})),\nonumber\\
\phi_{3}:&=\psi_{l}+\epsilon\varphi_{3}+\epsilon^{\beta-1}(
\mathcal{R}^{1}_{0}+\mathcal{R}^{1}_{1}+\mathcal{R}^{1}_{-1}),\\
\phi_{4}:&=\phi_{3}+\frac{\gamma-1}{\gamma}\phi_{2},\nonumber\\
\phi_{5}:&=\frac{1}{\sqrt{\gamma}}\phi_{2}.\nonumber
\end{align}
Here $\varphi_{l}$ with $l\in\{1,2,3\}$ is defined in equation \eqref{c123}, and $\mathcal{L}^{r}$ with $r\in\{1,2,3\}$ satisfies the estimate
\begin{equation}
\begin{split}\label{Lr}
\epsilon\|\mathcal{L}^{r}\|_{H^{s}}\leq C \epsilon(1+\|\mathcal{R}\|_{H^{s}}+\epsilon^{3/2}\|\mathcal{R}\|^{2}_{H^{s}}
+\epsilon^{3}\|\mathcal{R}\|^{3}_{H^{s}}),
\end{split}
\end{equation}
for any $s\geq1$.
\end{lemma}
This lemma shows that we have the following equations:
\begin{equation}
\begin{split}\label{1+-}
\Omega(\mathcal{R}^{1}_{1}-\mathcal{R}^{1}_{-1})=&
\frac{1}{1+\epsilon\phi_{3}}\Big(\partial_{t}(\mathcal{R}^{1}_{0}
+\mathcal{R}^{1}_{1}+\mathcal{R}^{1}_{-1})
-\epsilon\phi_{1}\partial_{x}
(\mathcal{R}^{1}_{0}+\mathcal{R}^{1}_{1}+\mathcal{R}^{1}_{-1})\\
&-\epsilon^{-\beta}\big(Res^{1}_{0}(\epsilon\Psi)+Res^{1}_{1}(\epsilon\Psi)+Res^{1}_{-1}(\epsilon\Psi)\big)
-\epsilon\mathcal{L}^{1}\Big),
\end{split}
\end{equation}
\begin{equation}
\begin{split}\label{1--}
\Omega(\mathcal{R}^{1}_{1}-\mathcal{R}^{1}_{-1})
=&\frac{1}{1+\epsilon\phi_{4}}\Big(\partial_{t}(\mathcal{R}^{1}_{1}+\mathcal{R}^{1}_{-1})
-\epsilon\phi_{1}\partial_{x}
(\mathcal{R}^{1}_{1}+\mathcal{R}^{1}_{-1})\\
&-\epsilon^{-\beta}\big(Res^{1}_{1}(\epsilon\Psi)+Res^{1}_{-1}(\epsilon\Psi)\big)
-\epsilon\mathcal{L}^{2}\Big),
\end{split}
\end{equation}
and
\begin{equation}
\begin{split}\label{1++}
\partial_{x}(\mathcal{R}^{1}_{1}+\mathcal{R}^{1}_{-1})
=&\frac{1}{q+\epsilon\phi_{5}}\Big(\partial_{t}(\mathcal{R}^{1}_{1}-\mathcal{R}^{1}_{-1})
-\epsilon\phi_{1}\partial_{x}
(\mathcal{R}^{1}_{1}-\mathcal{R}^{1}_{-1})
-\epsilon\phi_{5}\partial_{x}
\mathcal{R}^{1}_{0}\\
&-\epsilon^{-\beta}\big(Res^{1}_{1}(\epsilon\Psi)-Res^{1}_{-1}(\epsilon\Psi)\big)
-\epsilon\mathcal{L}^{3}\Big).
\end{split}
\end{equation}
We remark that the equations \eqref{1+-}-\eqref{1++} are very important to obtain the uniform estimate for the error $\mathcal{R}_{j}^{1}$ with $j\in\{0,\pm1\}$ in Sect. \S\ref{Sect6}. In fact, we have to bound the interaction terms $\epsilon^{2}\sum_{j\in\{\pm1\}}\int a_{j}\mathcal{R}^{1}_{j}\partial_{x}\mathcal{R}^{1}_{-j}$ and $\epsilon^{2}\sum_{j\in\{\pm1\}}\int b_{j}\mathcal{R}^{1}_{0}\partial_{x}\mathcal{R}^{1}_{j}$ with  $\mathcal{O}(\epsilon^{2})$ bound, whenever $\mathcal{R}^{1}\in L^{2}$ and $a_{j},b_{j}\in H^{1}$. According to Lemma \ref{L9}, we only need to estimate $\epsilon^{2}\int (a_{-1}-a_{1})(\mathcal{R}^{1}_{1}+\mathcal{R}^{1}_{-1})
\partial_{x}(\mathcal{R}^{1}_{1}-\mathcal{R}^{1}_{-1})$, $\epsilon^{2}\int (b_{1}+b_{-1})\mathcal{R}^{1}_{0}
\partial_{x}(\mathcal{R}^{1}_{1}-\mathcal{R}^{1}_{-1})$ and $\epsilon^{2}\int (b_{1}-b_{-1})\mathcal{R}^{1}_{0}
\partial_{x}(\mathcal{R}^{1}_{1}+\mathcal{R}^{1}_{-1})$. The estimates of all these interaction terms heavily depend on the equations \eqref{1+-}-\eqref{1++}.

\section{\textbf{Estimate of the error $(\mathcal{R}^{0},\mathcal{R}^{1})$}}\label{Sect6}
In the system \eqref{RJ} with \eqref{B0n}-\eqref{F01}, only the high frequency $\mathcal{O}(\epsilon)$ terms $\epsilon B^{1,1}_{j,l,n}(\psi_{l}, \mathcal{R}^{1}_{n})$ with $j,n\in\{0,\pm1\}$ and $l\in\{\pm1\}$ have not been eliminated. The reason is that the loss of one derivative of these terms finally causes the transformed system loses two derivatives. As a consequence, it does not work to bound the transformed error by the variation of constants formula and Gronwall's inequality, not because of losing powers of $\epsilon$ but of regularity problems. Therefore, instead of making the corresponding normal-form transform directly for $\mathcal{R}^{1}_{j}$, we use $\epsilon N^{1,1}$ to construct a modified energy $\mathcal{E}_{s}$ as follows
\begin{equation}\label{modiener}
\begin{split}
\mathcal{E}_{s}=\sum_{\ell=0}^{s}E_{\ell},
\end{split}
\end{equation}
with
\begin{equation}
\begin{split}\label{Eell}
E_{\ell}=\sum_{j\in\{0,\pm1\}}\bigg(&\frac{1}{2}\int(\partial_{x}^{\ell}\mathcal{R}_{j}^{0})^{2}dx+
\frac{1}{2}\int(\partial_{x}^{\ell}\mathcal{R}_{j}^{1})^{2}dx\\ &+\epsilon\sum_{\substack{n\in\{0,\pm1\}\\l\in\{\pm1\}}}
\int\partial_{x}^{\ell}\mathcal{R}_{j}^{1}\partial_{x}^{\ell}N_{j,l,n}^{1,1}(\psi_{l},\mathcal{R}_{n}^{1})dx
\bigg),
\end{split}
\end{equation}
where
\begin{equation}
\begin{split}\label{N11}
\widehat{N}_{j,l,n}^{1,1}(\phi_{l},\mathcal{R}_{n}^{1})
=\int n_{j,l,n}^{1,1}(k,k-\ell,\ell)\widehat{\psi}_{l}(k-\ell)
\widehat{\mathcal{R}}_{n}^{1}(\ell)d\ell.
\end{split}
\end{equation}
As was done for $n_{j,l,n}^{1,0}$, the kernels $n_{j,l,n}^{1,1}$ of $\widehat{N}_{j,l,n}^{1,1}$ have the form
\begin{align}\label{nj11}
n_{j,l,n}^{1,1}(k,k-\ell,\ell)=\frac{i\widehat{P}^{1}(k)b_{j,l,n}^{1,1}(k,k-\ell,\ell)}
{-j\omega(k)-\omega(k-\ell)+n\omega(\ell)},
\end{align}
where $b_{j,l,n}^{1,1}$ are the kernels of $B_{j,l,n}^{1,1}$ in the equation \eqref{B0n}-\eqref{Bjn}. More explicitly, we have
\begin{align}\label{bjln}
b_{0,l,0}^{1,1}(k,k-\ell,\ell)=&-\frac{i\ell\widehat{q}(k-\ell)\widehat{q}^{2}(\ell)}{\widehat{q}^{2}(k)}
+\frac{(\gamma-1)i(k-\ell)\widehat{q}(k-\ell)(\gamma-2-\widehat{q}^{2}(\ell))}{\widehat{q}^{2}(k)}\nonumber\\
=&:b_{0,l,0}^{1,1,1}(k,k-\ell,\ell)+b_{0,l,0}^{1,1,2}(k,k-\ell,\ell),\nonumber\\
b_{0,l,n}^{1,1}(k,k-\ell,\ell)=&-\frac{n(\gamma-1)(\gamma-2)i\ell\widehat{q}(\ell)}{\widehat{q}^{2}(k)}
+\frac{(\gamma-1)(\gamma-2)i(k-\ell)\widehat{q}(k-\ell)}{\widehat{q}^{2}(k)}\ \ \text{for} \  n\in\{\pm1\},\nonumber\\
b_{j,l,0}^{1,1}(k,k-\ell,\ell)=&\frac{i\ell\widehat{q}(k-\ell)\widehat{q}^{2}(\ell)}{2\widehat{q}^{2}(k)}
-\frac{ik\widehat{q}(k-\ell)}{2}
-\frac{(\gamma-1)i(k-\ell)\widehat{q}(k-\ell)(\gamma-2-\widehat{q}^{2}(\ell))}{2\widehat{q}^{2}(k)}\nonumber\\
&+\frac{j(\gamma-2)i\ell}{2\widehat{q}(k)}
+\frac{ji(k-\ell)(\gamma-2-\widehat{q}^{2}(\ell))}{2\widehat{q}(k)}\\
&-\frac{jik}{2\widehat{q}(k)(1+k^{2})}\frac{1}{1+(k-\ell)^{2}}\frac{1}{1+\ell^{2}},
 \ \ \text{for} \ j\in\{\pm1\},\nonumber\\
b_{j,l,n}^{1,1}(k,k-\ell,\ell)=&-\frac{ik\widehat{q}(k-\ell)}{2}
+\frac{nik\widehat{q}(\ell)}{2}
+\frac{n(\gamma-1)(\gamma-2)i\ell\widehat{q}(\ell)}{2\widehat{q}^{2}(k)}\nonumber\\
&-\frac{jnik\widehat{q}(k-\ell)\widehat{q}(\ell)}{2\widehat{q}(k)}
+\frac{j(\gamma-2)ik}{2\widehat{q}(k)}
-\frac{(\gamma-1)(\gamma-2)i(k-\ell)\widehat{q}(k-\ell)}{2\widehat{q}^{2}(k)}\nonumber\\
&-\frac{jik}{2\widehat{q}(k)(1+k^{2})}\frac{1}{1+(k-\ell)^{2}}\frac{1}{1+\ell^{2}}\nonumber\\
=&:\sum_{r=1}^{7}b_{j,l,n}^{1,1,r}(k,k-\ell,\ell), \ \ \text{for} \ j, \ n\in\{\pm1\}\nonumber.
\end{align}
Thus, we can rewrite the terms $N_{j,l,n}^{1,1}$ as follows:
\begin{align}\label{N11r}
N_{0,l,0}^{1,1}&=:N_{0,l,0}^{1,1,1}+N_{0,l,0}^{1,1,2}, \
N_{j,l,n}^{1,1}=:\sum_{r=1}^{7}N_{j,l,n}^{1,1,r}, \ \ \ \ \ \text{for} \ j, n\in\{\pm1\}.
\end{align}

\subsection{\textbf{Construction and properties of $N^{1,1}$}}
First we analyse the possible resonance for $N_{j,l,n}^{1,1}$. The kernel $n_{j,\pm1,n}^{1,1}$ can be modified by using Lemma \ref{L6} as was done for $n_{j,\pm1,n}^{1,0}$ as
\begin{align}\label{n11}
n_{j,\pm1,n}^{1,1}(k)=\frac{i\widehat{P}^{1}(k)b_{j,\pm1,n}^{1,1}(k,\pm k_{0},k\mp k_{0})}
{-j\omega(k)-\omega(\pm k_{0})+n\omega(k\mp k_{0})}.
\end{align}
Due to the support of $\psi_{\pm1}$ and the projection operator $\widehat{P}^{1}$, the kernels $n_{j,\pm1,n}^{1,1}$ in \eqref{n11} only need to be restricted to $|k-\ell\mp k_{0}|<\delta$, $|k|>\delta$ and $|\ell|>\delta$. We now consider the possible resonances in the denominator of \eqref{n11}:
\begin{itemize}
  \item $k=0$: Since $\widehat{P}^{1}(k)=0$ for $|k|<\delta$, this resonance does not play a role in the analysis of $N^{1,1}$;
  \item $k=\pm k_{0}$: Since $|k-\ell\mp k_{0}|<\delta$ and $|\ell|>\delta$, these resonances also does not play a role in the analysis of $N^{1,1}$;
  \item $k=\pm 2k_{0}$: The kernels $n_{j,\pm1,n}^{1,1}(k)$ have resonance at wave numbers $k=\pm 2k_{0}$ whenever $j=0$ and $n=1$.
\end{itemize}
However, we have the following lemma.
\begin{lemma}\label{L10}
The kernel $b_{0,\pm1,1}^{1,1}(k,k-\ell,\ell)$ of $B_{0,\pm1,1}^{1,1}$ equals to zero at $k=2\ell$.
\end{lemma}
\begin{proof}
According to the form of $B_{0,\pm1,1}^{1,1}$ in the equation \eqref{Bc02}, we have
\begin{align*}
B_{0,\pm1,1}^{1,1}=-\frac{(\gamma-1)(\gamma-2)}{q^{2}}\left(\psi_{\pm1}\partial_{x} q R^{1}_{1})\right)
+\frac{(\gamma-1)(\gamma-2)}{q^{2}}\left(\partial_{x} q\psi_{\pm1}R^{1}_{1})\right).
\end{align*}
Hence, we have
\begin{align*}
b_{0,\pm1,1}^{1,1}(k,k-\ell,\ell)
=-\frac{(\gamma-1)(\gamma-2)}{\widehat{q}^{2}(k)}(i\ell\widehat{q}(\ell)-
i(k-\ell)\widehat{q}(k-\ell)),
\end{align*}
and
\begin{align*}
b_{0,\pm1,1}^{1,1}(2\ell,\ell,\ell)
=-\frac{(\gamma-1)(\gamma-2)}{\widehat{q}^{2}(2\ell)}(i\ell\widehat{q}(\ell)-
i\ell\widehat{q}(\ell))=0.
\end{align*}
The proof is complete.
\end{proof}
Lemma \ref{L10} shows that the numerator of $n_{0,\pm1,1}^{1,1}(k)$ also vanishes at $k=\pm 2k_{0}$ so that the kernels $n_{j,\pm1,n}^{1,1}$ are still well-defined. There are no other resonances for all above mentioned normal-form transforms.

Having discussed the zeroes of the denominator of \eqref{nj11} we are now interested in the asymptotics for $|k|\rightarrow\infty$, in order to see a gain or loss of regularity by the normal-form transform. Before analysing the asymptotic behaviors of the $n_{j,l,n}^{1,1}(k,k-\ell,\ell)$ for $|k|\rightarrow\infty$, we state the following lemma.
\begin{lemma}\label{qq2}
As $|k|\rightarrow\infty$, the dispersive relation $\omega(k)=k\widehat{q}(k)=k\sqrt{\gamma+\frac{1}{1+k^{2}}}$ of \eqref{equation3} satisfies
\begin{align}\label{qomega}
\widehat{q}(k)=&\sqrt{\gamma}+\mathcal{O}(|k|^{-2}), \nonumber\\
\widehat{q}^{2}(k)=&\gamma+\mathcal{O}(|k|^{-2}),\nonumber\\
\omega(k)=&k\widehat{q}(k)=\sqrt{\gamma}k+\mathcal{O}(|k|^{-1}), \nonumber\\
\omega'(k)=&\sqrt{\gamma}+\mathcal{O}(|k|^{-2}).
\end{align}
\end{lemma}
Now we analyse the asymptotic behaviors of $n_{j,l,n}^{1,1}(k,k-\ell,\ell)$ for $|k|\rightarrow\infty$.
\begin{lemma}\label{L8}
The operators $N_{j,l,n}^{1,1}$ for $j,n\in\{0,\pm1\}$ and $l\in\{\pm1\}$ have the following properties.\\
(a) Fix $h\in L^{2}(\mathbb{R},\mathbb{R})$. Then the mappings $f\mapsto N_{0,l,0}^{1,1,1}(h,f)$ and $f\mapsto N_{j,l,j}^{1,1,r}(h,f)$ with $j\in\{\pm1\}$ and $r\in\{1,2,3,4,5\}$ define continuous linear maps from $H^{1}(\mathbb{R},\mathbb{R})$ into $L^{2}(\mathbb{R},\mathbb{R})$ and $f\mapsto N_{0,l,0}^{1,1,2}(h,f)$, $f\mapsto N_{0,l,n}^{1,1}(h,f)$, $f\mapsto B_{j,l,0}^{1,1}(h,f)$, $f\mapsto B_{j,l,-j}^{1,1}(h,f)$ with $j\in\{\pm1\}$ and $r\in\{1,2,3,4,5\}$ and $f\mapsto N_{j,l,n}^{1,1,r}(h,f)$ with $r\in\{6,7\}$ define continuous linear maps from $L^{2}(\mathbb{R},\mathbb{R})$ into $L^{2}(\mathbb{R},\mathbb{R})$. In particular, for all $f\in H^{1}(\mathbb{R},\mathbb{R})$ we have
\begin{equation}
\begin{split}\label{hfjj}
&N_{0,l,0}^{1,1,1}(h,f)=-\partial_{x}^{-1}h \ \partial_{x}f+Q_{0,0}^{1}(h,f),\\
&N_{0,l,0}^{1,1,2}(h,f)=\frac{-2(\gamma-1)}{\gamma}hf+Q_{0,0}^{2}(h,f),\\
&N_{0,l,n}^{1,1}(h,f)=\frac{(\gamma-1)(\gamma-2)}{\gamma}(hf-
\Omega h\frac{f}{n\Omega})+Q_{0,n}(h,f),\\
&N_{j,l,0}^{1,1}(h,f)=\frac{\gamma-2}{2\gamma}hf-j\frac{\gamma-2}{2\gamma}
\Omega h\frac{f}{\Omega}-\frac{\gamma+2}{\gamma}\partial_{x}h\partial_{x}^{-1}f+Q_{j,0}(h,f),\\
&N_{j,l,j}^{1,1,1}(h,f)=\partial_{x}(G_{j,j}qh \ f)+Q_{j,j}^{1}(h,f),\\
&N_{j,l,j}^{1,1,2}(h,f)=-j\partial_{x}(G_{j,j}h \ qf)+Q_{j,j}^{2}(h,f),\\
&N_{j,l,j}^{1,1,3}(h,f)=\frac{-j(\gamma-1)(\gamma-2)}{q^{2}}(G_{j,j}h \ \Omega f)+Q_{j,j}^{3}(h,f),\\
&N_{j,l,j}^{1,1,4}(h,f)=\frac{\partial_{x}}{q}(G_{j,j}qh \ qf)+Q_{j,j}^{4}(h,f),\\
&N_{j,l,j}^{1,1,5}(h,f)=\frac{-j(\gamma-2)}{q}\partial_{x}(G_{j,j}h \ f)+Q_{j,j}^{5}(h,f),\\
&N_{j,l,j}^{1,1,6}(h,f)=\frac{(\gamma-1)(\gamma-2)}{\gamma}G_{j,j}\Omega h \ f+Q_{j,j}^{6}(h,f),\\
&N_{j,l,j}^{1,1,7}(h,f)=j\frac{\partial_{x}}{q(1-\partial^{2}_{x})}(\frac{G_{j,j}h}{1-\partial^{2}_{x}}
\frac{f}{1-\partial^{2}_{x}})+Q_{j,j}^{7}(h,f),
\end{split}
\end{equation}
and
\begin{align}\label{hfj-j}
&N_{j,l,-j}^{1,1,1}(h,f)=-jG_{j,-j}qh \ f+Q_{j,-j}^{1}(h,f),\nonumber\\
&N_{j,l,-j}^{1,1,2}(h,f)=-G_{j,-j}h \ qf+Q_{j,-j}^{2}(h,f),\nonumber\\
&N_{j,l,-j}^{1,1,3}(h,f)=-\frac{(\gamma-1)(\gamma-2)}{\widehat{q}^{2}(k)}(G_{j-j}h \ qf)+Q_{j-j}^{3}(h,f),\nonumber\\
&N_{j,l,-j}^{1,1,4}(h,f)=\frac{j}{\widehat{q}(k)}(G_{j,-j}qh \ qf)+Q_{j,-j}^{4}(h,f),\nonumber\\
&N_{j,l,-j}^{1,1,5}(h,f)=\frac{\gamma-2}{\widehat{q}(k)}(G_{j,-j}h \ f)+Q_{j,-j}^{5}(h,f),\\
&N_{j,l,-j}^{1,1,6}(h,f)=\frac{(\gamma-1)(\gamma-2)}{\gamma}G_{j,-j}\Omega h \ f+Q_{j,-j}^{6}(h,f),\nonumber\\
&N_{j,l,-j}^{1,1,7}(h,f)=j\frac{\partial_{x}}{q(1-\partial^{2}_{x})}(\frac{G_{j,-j}h}{1-\partial^{2}_{x}}
\frac{f}{1-\partial^{2}_{x}})+Q_{j,-j}^{7}(h,f),\nonumber
\end{align}
where
\begin{align*}
\widehat{G_{j,j}h}(k)=&\frac{\chi(k)}{-2i(\sqrt{\gamma}jk+\omega(k))}\widehat{h}(k),\\
\widehat{G_{j,-j}h}(k)=&\frac{1}{4}\chi(k)\widehat{h}(k),\\
\|Q_{0,0}^{1}(h,f)\|_{H^{1}}=&\mathcal{O}(\|h\|_{L^{2}},\|f\|_{L^{2}}),\\
\|Q_{j,\pm j}^{r}(h,f)\|_{H^{1}}=&\mathcal{O}(\|h\|_{L^{2}},\|f\|_{L^{2}}), \ r=1, 2, 3, 4, 5,
\end{align*}
and $\widehat{G_{j,j}}$ with $j\in\{\pm1\}$ satisfies
\begin{equation}\label{part4}
(\widehat{G}_{-1-1}+\widehat{G}_{11})(k)=(\frac{1}{-ik}+ik)q(k),
\end{equation}
and
\begin{equation}\label{part5}
(\widehat{G}_{-1-1}-\widehat{G}_{11})(k)=\sqrt{\gamma}(\frac{1}{-ik}+ik).
\end{equation}
\\
(b) For all $f\in H^{1}(\mathbb{R}, \mathbb{R})$ we have
\begin{equation}\label{16}
-j\Omega N_{j,l,n}^{1,1}(\psi_{l},f)-N_{j,l,n}^{1,1}(\Omega\psi_{l},f)
+nN_{j,l,n}^{1,1,n}(\psi_{l},\Omega f)
=-P^{1}B_{j,l,n}^{1,1}(\psi_{l},f).
\end{equation}
\\
(c) For all $f, g, h\in H^{1}(\mathbb{R}, \ \mathbb{R})$ we have
\begin{equation}
\begin{split}\label{17}
\int fN_{0,l,0}^{1,1,1}(h,g)=&-\int N_{0,l,0}^{1,1,1}(h,f)gdx
+\int S_{0,0}^{1}(h,f)gdx,\\
\int fN_{j,l,n}^{1,1,\{1,4\}}(h,g)=&-\int N_{j,l,n}^{1,1,\{1,4\}}(h,f)gdx
+\int S_{j,n}^{\{1,4\}}(\partial_{x}h,f)gdx,\\
\int fN_{j,l,n}^{1,1,\{2,3,5\}}(h,g)=&-\frac{j}{n}\int N_{j,l,n}^{1,1,\{2,3,5\}}(h,f)gdx
+\int S_{j,n}^{\{2,3,5\}}(\partial_{x}h,f)gdx,
\end{split}
\end{equation}
where
\begin{equation*}
\begin{split}
\widehat{S}_{j,n}^{r}(\partial_{x}h,f)g(k)=\int s_{j,n}^{r}(k,k-\ell,\ell)\widehat{\partial_{x}h}(k-\ell)\widehat{f}(\ell)d\ell,
\end{split}
\end{equation*}
with
\begin{align}\label{s00}
&s_{0,0}^{1}(k,k-\ell,\ell)=\frac{1}{k-\ell}\left(\frac{k\widehat{q}^{2}(k)}{\widehat{q}^{2}(\ell)}
-\frac{\ell\widehat{q}^{2}(\ell)}{\widehat{q}^{2}(k)}\right)
=1+\mathcal{O}(|k|^{-2}), \ \text{for} \ |k|\rightarrow\infty,\nonumber\\
&s_{j,n}^{1}(k,k-\ell,\ell)=\frac{1}{2i}\frac{\widehat{q}(k-\ell)}
{-n\omega(k)-\omega(k-\ell)+j\omega(\ell)},\nonumber\\
&s_{j,n}^{2}(k,k-\ell,\ell)=\frac{-n(k\widehat{q}(\ell)-\ell\widehat{q}(k))}
{2i(k-\ell)(-n\omega(k)-\omega(k-\ell)+j\omega(\ell))},\nonumber\\
&s_{j,n}^{3}(k,k-\ell,\ell)=\frac{n(\gamma-1)(\gamma-2)}
{2i(k-\ell)(-n\omega(k)-\omega(k-\ell)+j\omega(\ell))}
\left(\frac{k\widehat{q}(k)}{\widehat{q}^{2}(\ell)}
-\frac{\ell\widehat{q}(\ell)}{\widehat{q}^{2}(k)}\right),\nonumber\\
&s_{j,n}^{4}(k,k-\ell,\ell)=\frac{\widehat{q}(k-\ell)}
{2i(k-\ell)(-n\omega(k)-\omega(k-\ell)+j\omega(\ell))}\left(\frac{k\widehat{q}(\ell)}{\widehat{q}(k)}
-\frac{\ell\widehat{q}(k)}{\widehat{q}(\ell)}\right),\nonumber\\
&s_{j,n}^{5}(k,k-\ell,\ell)=\frac{-j(\gamma-2)}
{2i(k-\ell)(-n\omega(k)-\omega(k-\ell)+j\omega(\ell))}\left(\frac{k}{\widehat{q}(k)}
-\frac{\ell}{\widehat{q}(\ell)}\right).
\end{align}
In particular, we have
\begin{equation}
\begin{split}\label{Sjj}
\widehat{S}_{j,j}^{\{1,4\}}(\partial_{x}h,f)=&G_{j,j}\partial_{x}qh \ f+\widetilde{Q}_{j,j}^{\{1,4\}}(\partial_{x}h,f),\\
\widehat{S}_{j,j}^{2}(\partial_{x}h,f)=&-j\sqrt{\gamma}G_{j,j}\partial_{x}h \ f+\widetilde{Q}_{j,j}^{2}(\partial_{x}h,f),\\
\widehat{S}_{j,j}^{\{3,5\}}(\partial_{x}h,f)=&-\frac{j}{\sqrt{\gamma}}G_{j,j}\partial_{x}h \ f+\widetilde{Q}_{j,j}^{\{3,5\}}(\partial_{x}h,f),\\
\end{split}
\end{equation}
and for $r\in\{1, 2, 3, 4, 5\}$, we have
\begin{equation*}
\begin{split}
\|\widetilde{Q}_{j,j}^{r}(\partial_{x}h,f)\|_{H^{2}}=\mathcal{O}(\|h\|_{L^{2}},\|f\|_{L^{2}}).
\end{split}
\end{equation*}
\end{lemma}
\begin{proof}
For $n_{0,l,0}^{1,1}$, we have
\begin{equation*}
\begin{split}
n_{0,l,0}^{1,1,1}(k,k-\ell,\ell)=&\frac{i\widehat{P}^{1}(k)b_{0,l,0}^{1,1,1}(k,k-\ell,\ell)}
{-\omega(k-\ell)}
=\frac{i\widehat{P}^{1}(k)}
{-\omega(k-\ell)}\frac{i\ell\widehat{q}(k-\ell)\widehat{q}^{2}(\ell)}{-\widehat{q}^{2}(k)} \\
=& -\frac{\ell\widehat{q}^{2}(\ell)}{(k-\ell)\widehat{q}^{2}(k)} = -\frac{\ell}{k-\ell}+\mathcal{O}(|k|^{-2}),\\
n_{0,l,0}^{1,1,2}(k,k-\ell,\ell)=&\frac{i\widehat{P}^{1}(k)b_{0,l,0}^{1,1,2}(k,k-\ell,\ell)}{-\omega(k-\ell)} \\
= & \frac{i\widehat{P}^{1}(k)}{-\omega(k-\ell)}\frac{(\gamma-1)i(k-\ell)\widehat{q}(k-\ell) (\gamma-2-\widehat{q}^{2}(\ell))}{\widehat{q}^{2}(k)}\\
= & \frac{(\gamma-1)(\gamma-2-\widehat{q}^{2}(\ell))}{\widehat{q}^{2}(k)} =\frac{-2(\gamma-1)}{\gamma}+\mathcal{O}(|k|^{-2}),
\end{split}
\end{equation*}
where we have used the dispersive relation $\omega(k)=k\widehat{q}(k)$ of \eqref{equation3}. We see that $N_{0,l,0}^{1,1,1}$ will lose one derivative, but $N_{0,l,0}^{1,1,2}$ dose not lose derivatives due to $\widehat{q}^{2}(k)=\gamma+\mathcal{O}(|k|^{-2})$ in Lemma \ref{qq2}.

For $n_{0,l,n}^{1,1}$ with $n\in\{\pm1\}$, using the asymptotic expansion \eqref{qomega} we have
\begin{align*}
n_{0,l,n}^{1,1}(k,k-\ell,\ell)=&\frac{i\widehat{P}^{1}(k)b_{0,l,n}^{1,1}(k,k-\ell,\ell)}
{-\omega(k-\ell)+n\omega(\ell)}\\
=&\frac{(\gamma-1)(\gamma-2)(n\ell\widehat{q}(\ell)-(k-\ell)\widehat{q}(k-\ell))}
{(\gamma+\mathcal{O}(|k|^{-2}))(n\ell\widehat{q}(\ell)+\mathcal{O}(1))}\chi(k-\ell)\\
=&\frac{(\gamma-1)(\gamma-2)}
{\gamma}\left(1-\frac{(k-\ell)\widehat{q}(k-\ell))}{n\ell\widehat{q}(\ell)}\right)+\mathcal{O}(|k|^{-2}))
\chi(k-\ell),
\end{align*}
which implies that $N_{0,l,n}^{1,1,1}$ dose not lose any derivatives for $n\in\{\pm1\}$.

For $n_{j,l,0}^{1,1}$ with $j\in\{\pm1\}$, using the asymptotic expansion \eqref{qomega} we have
\begin{equation*}
\begin{split}
n_{j,l,0}^{1,1}(k,k-\ell,\ell)=&\frac{i\widehat{P}^{1}(k)b_{j,l,0}^{1,1}(k,k-\ell,\ell)}
{-\omega(k-\ell)}\\
=&\frac{\frac{j(\gamma-2)k}{2\widehat{q}(k)}
+(-\frac{1}{2}+\frac{1}{\gamma})(k-\ell)\widehat{q}(k-\ell)
+\frac{j(\gamma+2)(k-\ell)}{2\widehat{q}(k)}
+\mathcal{O}(|k|^{-2})}{-j\omega(k)+\mathcal{O}(1)}\chi(k-\ell)\\
=&(\frac{\gamma-2}{2\gamma}-j\frac{\gamma-2}{2\gamma}\frac{\omega(k-\ell)}{\omega(k)}
-\frac{\gamma+2}{\gamma}\frac{k-\ell}{k}+\mathcal{O}(|k|^{-2}))\chi(k-\ell).
\end{split}
\end{equation*}
Similar to $n_{0,l,n}^{1,1}$, we see that $N_{j,l,0}^{1,1,1}$ dose not lose any derivatives for $j\in\{\pm1\}$.

For $n_{j,l,j}^{1,1,r}$ with $j\in\{\pm1\}$ and $r\in\{1, 2, 3, 4, 5\}$, by the mean value theorem we get
\begin{equation*}
\begin{split}
n_{j,l,j}^{1,1,1}(k,k-\ell,\ell)=&-\frac{kq(k-\ell)\chi(k-\ell)}{2(j(\omega(k)-\omega(\ell))+\omega(k-\ell))}\\
=&-\frac{kq(k-\ell)\chi(k-\ell)}{2(j(k-\ell)\omega'(k-\theta(k-\ell))+\omega(k-\ell))}, \
 \ j\in\{\pm1\}
 \end{split}
\end{equation*}
for some $\theta\in[0,1]$. Using again the fact that supp$\chi$ is compact, we conclude with the help of the expressions \eqref{qomega} that
\begin{equation*}
\begin{split}
n_{j,l,j}^{1,1,1}(k,k-\ell,\ell)
=&-\frac{kq(k-\ell)\chi(k-\ell)}{2(j(k-\ell)(\sqrt{\gamma}+\mathcal{O}(|k|^{-2}))+\omega(k-\ell))}\\
=&\left(-\frac{kq(k-\ell)}{2(\sqrt{\gamma}j(k-\ell)+\omega(k-\ell))}+\mathcal{O}(|k|^{-1})\right)\chi(k-\ell),  \ \ \text{for} \  |k|\rightarrow\infty.
\end{split}
\end{equation*}
Similarly, for $|k|\rightarrow\infty$ and $|\ell|\rightarrow\infty$ we have
\begin{align*}
\widehat{n}_{j,l,j}^{1,1,2}(k,k-\ell,\ell)
=&\left(\frac{jkq(\ell)}{2(\sqrt{\gamma}j(k-\ell)+\omega(k-\ell))}+\mathcal{O}(|k|^{-1})\right)\chi(k-\ell),\\
\widehat{n}_{j,l,j}^{1,1,3}(k,k-\ell,\ell)
=&\left(\frac{j(\gamma-1)(\gamma-2)\ell q(\ell)}{2\widehat{q}^{2}(k)(\sqrt{\gamma}j(k-\ell)+\omega(k-\ell))}
+\mathcal{O}(|\ell|^{-1})\right)\chi(k-\ell),\\
\widehat{n}_{j,l,j}^{1,1,4}(k,k-\ell,\ell)
=&\left(-\frac{k q(k-\ell)\widehat{q}(\ell)}
{2\widehat{q}(k)(\sqrt{\gamma}j(k-\ell)+\omega(k-\ell))}+\mathcal{O}(|k|^{-1})\right)\chi(k-\ell),\\
\widehat{n}_{j,l,j}^{1,1,5}(k,k-\ell,\ell)
=&\left(\frac{j(\gamma-2)k}
{2\widehat{q}(k)(\sqrt{\gamma}j(k-\ell)+\omega(k-\ell))}+\mathcal{O}(|k|^{-1})\right)\chi(k-\ell).
\end{align*}
We see that $N_{j,l,j}^{1,1,r}$ with $j\in\{\pm1\}$ will lose one derivative for $r\in\{1, 2, 3, 4, 5\}$.

For $n_{j,l,j}^{1,1,r}$ with $j\in\{\pm1\}$ and $r\in\{6, 7\}$, we get
\begin{align*}
\widehat{n}_{j,l,j}^{1,1,6}(k,k-\ell,\ell)
=&\frac{(\gamma-1)(\gamma-2)}{\gamma}\left(\frac{(k-\ell)\widehat{q}(k-\ell)}{2(\sqrt{\gamma}j(k-\ell)
+\omega(k-\ell))}+\mathcal{O}(|k|^{-1})\right)\chi(k-\ell),\\
\widehat{n}_{j,l,j}^{1,1,7}(k,k-\ell,\ell)
=&\frac{jk}{\widehat{q}(k)(1+k^{2})(1+(k-\ell)^{2})(1+\ell^{2})}\\
&\times\left(\frac{(k-\ell)\widehat{q}(k-\ell)}{2(\sqrt{\gamma}j(k-\ell)
+\omega(k-\ell))}+\mathcal{O}(|k|^{-1})\right)\chi(k-\ell).
\end{align*}
We see that $N_{j,l,j}^{1,1,r}$ with $j\in\{\pm1\}$ does not lose any derivatives for $r\in\{6, 7\}$.

For $N_{j,l,-j}^{1,1,r}$ with $j\in\{\pm1\}$ and $r\in\{1, 2, 3, 4, 5\}$, exploiting once more the compactness of supp$\chi$ as well as \eqref{qomega} yields
\begin{align*}
n_{j,l,-j}^{1,1,1}(k,k-\ell,\ell)=&\frac{-kq(k-\ell)\chi(k-\ell)}{2(j(\omega(k)+\omega(k-(k-\ell)))+\omega(k-\ell))}\\
=&\frac{-kq(k-\ell)\chi(k-\ell)}{4j\omega(k)(1+\mathcal{O}(|k|^{-1}))}\\
=&\frac{-q(k-\ell)}{4}+\mathcal{O}(|k|^{-1}),\ \ \ \ \text{for} \  |k|\rightarrow\infty
\end{align*}
Similarly, we have
\begin{align*}
n_{j,l,-j}^{1,1,2}(k,k-\ell,\ell)
=&\frac{-jq(\ell)}{4}+\mathcal{O}(|k|^{-1}),\\
n_{j,l,-j}^{1,1,3}(k,k-\ell,\ell)
=&\frac{-j(\gamma-1)(\gamma-2)q(\ell)}{4\widehat{q}^{2}(k)}+\mathcal{O}(|\ell|^{-1}),\\
n_{j,l,-j}^{1,1,4}(k,k-\ell,\ell)
=&\frac{q(k-\ell)q(\ell)}{4\widehat{q}(k)}+\mathcal{O}(|k|^{-1}),\\
n_{j,l,-j}^{1,1,5}(k,k-\ell,\ell)
=&\frac{j(\gamma-2)}{4\widehat{q}(k)}+\mathcal{O}(|k|^{-1}).
\end{align*}
We see that $N_{j,l,-j}^{1,1,r}$ with $j\in\{\pm1\}$ does not lose any derivatives for $r\in\{1, 2, 3, 4, 5\}$.

For $N_{j,l,-j}^{1,1,r}$ with $j\in\{\pm1\}$ and $r\in\{6, 7\}$, we get
\begin{equation*}
\begin{split}
n_{j,l,-j}^{1,1,6}(k,k-\ell,\ell)
=&\frac{j(\gamma-1)(\gamma-2)\Omega(k-\ell)}{4\gamma\Omega(k)}+\mathcal{O}(|k|^{-1}),\\
n_{j,l,-j}^{1,1,7}(k,k-\ell,\ell)
=&\frac{k}{4\Omega(k)\widehat{q}(k)(1+k^{2})(1+k-\ell)^{2})(1+\ell^{2})}+\mathcal{O}(|k|^{-1}),
\end{split}
\end{equation*}
We see that $N_{j,l,-j}^{1,1,r}$ with $j\in\{\pm1\}$ does not lose any derivatives and even gain derivatives for $r\in\{6, 7\}$.

Finally, since
\begin{equation*}
\begin{split}
n_{j,l,n}^{1,1}(-k,-(k-\ell),-\ell)=\widehat{b}_{j,l,n}^{1,1}(k,k-\ell,\ell)\in\mathbb{R}
\end{split}
\end{equation*}
and $\psi_{\pm1}$ is real-valued, we obtain the validity of all assertions of (a).

(b) is a direct consequence of the construction of the operators $N_{j,l,n}^{1,1}$.

In order to prove (c), we compute for all $f, g, h\in H^{1}(\mathbb{R},\mathbb{R})$ that
\begin{align*}
(f, N_{0,l,0}^{1,1,1}(h,g))
=&\int\overline{\widehat{f}(k)}\widehat{N}_{0,l,0}^{1,1,1}(h,g)(k)dk\\
=&-\int\int\overline{\widehat{f}(k)}\frac{\ell\widehat{q}^{2}(\ell)}{(k-\ell)\widehat{q}^{2}(k)}
\widehat{h}(k-\ell)\widehat{g}(\ell)d\ell dk\\
=&-\int\int\overline{\widehat{g}(-\ell)}\frac{\ell\widehat{q}^{2}(\ell)}{(k-\ell)\widehat{q}^{2}(k)}
\widehat{h}(k-\ell)\widehat{f}(-k)dkd\ell\\
=&\int\int\overline{\widehat{g}(k)}\frac{k\widehat{q}^{2}(-k)}{(k-\ell)\widehat{q}^{2}(-\ell)}
\widehat{h}(k-\ell)\widehat{f}(\ell)dmdk\\
=&\int\int\overline{\widehat{g}(k)}\frac{\ell\widehat{q}^{2}(\ell)}{(k-\ell)\widehat{q}^{2}(k)}
\widehat{h}(k-\ell)\widehat{f}(\ell)d\ell dk\\
&+\int\int\overline{\widehat{g}(k)}\frac{1}{(k-\ell)}
\left(\frac{k\widehat{q}^{2}(k)}{\widehat{q}^{2}(\ell)}
-\frac{\ell\widehat{q}^{2}(\ell)}{\widehat{q}^{2}(k)}\right)
\widehat{h}(k-\ell)\widehat{f}(\ell)dmdk\\
=&-\int\overline{\widehat{g}(k)}\widehat{N}_{0,l,0}^{1,1,1}(h,f)(k)dk
+\int\overline{\widehat{g}(k)}\widehat{S}_{0,0}^{1}(h,f)(k)dk\\
=&-(g,N_{0,l,0}^{1,1,1}(h,f))+(g,S_{0,0}^{1}(h,f)).
\end{align*}
For $N_{j,l,n}^{1,1,r}$ with $j, n\in\{\pm1\}$ and $r\in\{1, 2, 3, 4, 5\}$, we have
\begin{align*}
(f, N_{j,l,n}^{1,1,1}(h,g)) =&\int\overline{\widehat{f}(k)}\widehat{N}_{j,l,n}^{1,1,1}(h,g)(k)dk\\
=&\frac{1}{2}\int\int\overline{\widehat{f}(k)}\frac{k\widehat{q}(k-\ell)}
{-j\omega(k)-\omega(k-\ell)+n\omega(\ell)}\widehat{h}(k-\ell)\widehat{g}(\ell)d\ell dk\\
=&\frac{1}{2}\int\int\overline{\widehat{g}(-\ell)}\frac{k\widehat{q}(k-\ell)}
{-j\omega(k)-\omega(k-\ell)+n\omega(\ell)}
\widehat{h}(k-\ell)\widehat{f}(-k)dkd\ell\\
=&\frac{1}{2}\int\int\overline{\widehat{g}(k)}\frac{-\ell\widehat{q}(k-\ell)}
{-n\omega(k)-\omega(k-\ell)+j\omega(\ell)}\widehat{h}(k-\ell)\widehat{f}(\ell)d\ell dk\\
=&\frac{1}{2}\int\int\overline{\widehat{g}(k)}\frac{-k\widehat{q}(k-\ell)}
{-n\omega(k)-\omega(k-\ell)+j\omega(\ell)}
\widehat{h}(k-\ell)\widehat{f}(\ell)d\ell dk\\
&+\frac{1}{2}\int\int\overline{\widehat{g}(k)}\frac{(k-\ell)\widehat{q}(k-\ell)}
{-n\omega(k)-\omega(k-\ell)+j\omega(\ell)}
\widehat{h}(k-\ell)\widehat{f}(\ell)d\ell dk\\
=&-\frac{1}{2}\int_{\mathbb{R}}\overline{\widehat{g}(k)}\widehat{N}_{n,l,j}^{1,1,1}(h,f)(k)dk
+\int_{\mathbb{R}}\overline{\widehat{g}(k)}\widehat{S}_{n,j}^{1}(\partial_{x}h,f)(k)dk\\
=&-(g, N_{n,l,j}^{1,1,1}(h,f))+(g, S_{n,j}^{1}(h,f)).
\end{align*}
Similarly, we have
\begin{align*}
(f,N_{j,l,n}^{1,1,\{2,3,5\}}(h,g))=&-\frac{j}{n}(g,N_{n,l,j}^{1,1,\{2,3,5\}}(h,f))
+(g,S_{n,j}^{\{2,3,5\}}(h,f)),\\
(f,N_{j,l,n}^{1,1,4}(h,g))=&-(g,N_{n,l,j}^{1,1,4}(h,f))+(g,S_{n,j}^{4}(h,f)).
\end{align*}
Then we obtain \eqref{17}, and due to \eqref{qomega} and \eqref{hfjj} we obtain \eqref{Sjj}.
\end{proof}

\subsection{\textbf{Control for the time evolution of $\mathcal{E}_{s}$}}
We have the following corollary.
\begin{corollary}\label{C1}
$\sqrt{\mathcal{E}_{s}}$ is equivalent to $\|R^{0}\|_{H^{s}}+\|R^{1}\|_{H^{s}}$ for sufficiently small $\epsilon>0$.
\end{corollary}
Indeed, it follows from assertions of Lemma \ref{L8} (a), the integration by parts and the Cauchy-Schwarz inequality.

Since the right-hand side of the error equations in \eqref{RJ} for $\mathcal{R}^{1}_{j}$ with $j\in\{0,\pm1\}$ loses one derivative, we need the following identities to control the time evolution of $\mathcal{E}_{s}$.
\begin{lemma}\label{L9}
Let $j\in\{0, \ \pm1\}$, $a_{j}\in H^{2}(\mathbb{R},\mathbb{R})$ and $f_{j}\in L^{2}(\mathbb{R},\mathbb{R})$. Then we have
\begin{equation}
\begin{split}\label{part1}
&\int a_{j}f_{j}\partial_{x}f_{j}dx=-\frac{1}{2}\int\partial_{x}a_{j}f_{j}^{2}dx,
\end{split}
\end{equation}
\begin{equation}
\begin{split}\label{part2}
\sum_{j\in\{\pm1\}}\int a_{j}f_{j}\partial_{x}f_{-j}dx = \frac{1}{2}\int(a_{-1}-a_{1})(f_{1}+f_{-1})\partial_{x}(f_{1}-f_{-1})dx\\
+\mathcal{O}(\|a_{1}\|_{H^{2}(\mathbb{R},\mathbb{R})}+\|a_{-1}\|_{H^{2}(\mathbb{R},\mathbb{R})}) (\|f_{1}\|_{L^{2}(\mathbb{R},\mathbb{R})}+\|f_{-1}\|_{L^{2}(\mathbb{R},\mathbb{R})}),
\end{split}
\end{equation}
and
\begin{equation}
\begin{split}\label{part20}
 \sum_{j\in\{\pm1\}}\int a_{j}f_{0}\partial_{x}f_{j}dx = & \frac{1}{2}\int(a_{1}+a_{-1})f_{0}\partial_{x}(f_{1}+f_{-1})dx\\
&+\frac{1}{2}\int(a_{1}-a_{-1})f_{0}\partial_{x}(f_{1}-f_{-1})dx.
\end{split}
\end{equation}
\end{lemma}
\begin{proof}
One may refer to \cite{LP19} for a proof.
\end{proof}

Now, we are prepared to analyze $\partial_{t}E_{\ell}$. We compute
\begin{align*}
\partial_{t}E_{\ell}=&\sum_{j\in\{0,\pm1\}}\Bigg(\int\partial_{x}^{\ell}\mathcal{R}_{j}^{0}
\partial_{t}\partial_{x}^{\ell}\mathcal{R}_{j}^{0}dx+
\int\partial_{x}^{\ell}\mathcal{R}_{j}^{1}\partial_{t}\partial_{x}^{\ell}\mathcal{R}_{j}^{1}dx\\
& \ \ \ \ +\epsilon\sum_{\substack{n\in\{0,\pm1\}\\ l\in\{\pm1\}}}
\bigg(\int\partial_{t}\partial_{x}^{\ell}\mathcal{R}_{j}^{1}
\partial_{x}^{\ell}N_{j,l,n}^{1,1}(\psi_{l},\mathcal{R}_{n}^{1})dx
+\int\partial_{x}^{\ell}\mathcal{R}_{j}^{1}
\partial_{x}^{\ell}N_{j,l,n}^{1,1}(\partial_{t}\psi_{l},\mathcal{R}_{n}^{1})dx\\
& \ \ \ \ \ \ \ \ \ \ \ \ \ \ \ \ \ \ \ \ \ \ \ \ +\int\partial_{x}^{\ell}\mathcal{R}_{j}^{1}
\partial_{x}^{\ell}N_{j,l,n}^{1,1}(\psi_{l},\partial_{t}\mathcal{R}_{n}^{1})dx\bigg)
\Bigg).
\end{align*}
Using the error equations \eqref{RJ}, we get
\begin{align*}
\partial_{t}E_{\ell}=&\sum_{j\in\{0,\pm1\}}\Bigg(\int\partial_{x}^{\ell}\mathcal{R}_{j}^{0}
j\Omega\partial_{x}^{\ell}\mathcal{R}_{j}^{0}dx
+\epsilon^{2}\int\partial_{x}^{\ell}\mathcal{R}_{j}^{0}
\partial_{x}^{\ell}\mathcal{F}_{j}^{0}dx
+\int\partial_{x}^{\ell}\mathcal{R}_{j}^{1}j\Omega\partial_{x}^{\ell}\mathcal{R}_{j}^{1}dx\\
&+\epsilon^{2}\sum_{n\in\{0,\pm1\}}\int\partial_{x}^{\ell}\mathcal{R}_{j}^{1}\partial_{x}^{\ell}\mathcal{F}_{j,n}^{1}dx
+\epsilon\sum_{\substack{n\in\{0,\pm1\}\\ l\in\{\pm1\}}}
\bigg(\int\partial_{x}^{\ell}\mathcal{R}_{j}^{1}\partial_{x}^{\ell}B^{1,1}_{j,l,n}(\psi_{l}, \mathcal{R}^{1}_{n})dx\\
&+\int j\Omega\partial_{x}^{\ell}\mathcal{R}_{j}^{1}
\partial_{x}^{\ell}N_{j,l,n}^{1,1}(\psi_{l},\mathcal{R}_{n}^{1})dx
-\int\partial_{x}^{\ell}\mathcal{R}_{j}^{1}
\partial_{x}^{\ell}N_{j,l,n}^{1,1}(\Omega\psi_{l},\mathcal{R}_{n}^{1})dx\\
&+\int\partial_{x}^{\ell}\mathcal{R}_{j}^{1}
\partial_{x}^{\ell}N_{j,l,n}^{1,1}(\psi_{l},n\Omega\mathcal{R}_{n}^{1})dx
+\int\partial_{x}^{\ell}\mathcal{R}_{j}^{1}
\partial_{x}^{\ell}N_{j,l,n}^{1,1}(\partial_{t}\psi_{l}+\Omega\psi_{l},\mathcal{R}_{n}^{1})dx\bigg)\\
&+\epsilon^{2}\sum_{\substack{m,n\in\{0,\pm1\}\\ l\in\{\pm1\}}}
\bigg(\int \partial_{x}^{\ell}\left(B^{1,1}_{j,l,m}(\psi_{l}, \mathcal{R}^{1}_{m})
+\epsilon\mathcal{F}_{j,m}^{1}\right)
\partial_{x}^{\ell}N_{j,l,n}^{1,1}(\psi_{l},\mathcal{R}_{n}^{1})dx\\
&+\int \partial_{x}^{\ell}\mathcal{R}_{j}^{1}
\partial_{x}^{\ell}N_{j,l,n}^{1,1}\left(\psi_{l},B^{1,1}_{n,l,m}(\psi_{l}, \mathcal{R}^{1}_{m})
+\epsilon\mathcal{F}_{n,m}^{1}\right)dx\bigg)\Bigg)\\
=& :\sum_{i=1}^{11}I_{i}.
\end{align*}
Due to the skew symmetry of $\Omega$, we have
\begin{align*}
I_{1}=\sum_{j\in\{0,\pm1\}}\int\partial_{x}^{\ell}\mathcal{R}_{j}^{0}
j\Omega\partial_{x}^{\ell}\mathcal{R}_{j}^{0}dx=0,\\
I_{3}=\sum_{j\in\{0,\pm1\}}\int\partial_{x}^{\ell}\mathcal{R}_{j}^{1}
j\Omega\partial_{x}^{\ell}\mathcal{R}_{j}^{1}dx=0.
\end{align*}
Because $\mathcal{F}_{j}^{0}$ does not lose any derivatives and satisfies the estimate \eqref{J68}, we have
\begin{align*}
I_{2}=\epsilon^{2}\sum_{j\in\{0,\pm1\}}\int\partial_{x}^{\ell}\mathcal{R}_{j}^{0}
\partial_{x}^{\ell}\mathcal{F}_{j}^{0}dx
\leq C\epsilon^{2}(\mathcal{E}_{s}+\epsilon^{3/2}\mathcal{E}_{s}^{3/2}+\epsilon^{3}\mathcal{E}_{s}^{2}+1).
\end{align*}
Since the operator $N_{j,l,n}^{1,1}$ satisfies \eqref{16}, we have $I_{5}+I_{6}+I_{7}+I_{8}=0$.
Moreover, due to the estimate \eqref{11s} in Lemma \ref{LO}, we have
\begin{align*}
I_{9}=\epsilon\sum_{\substack{j,n\in\{0,\pm1\}\\l\in\{\pm1\}}}\int\partial_{x}^{\ell}\mathcal{R}_{j}^{1}
\partial_{x}^{\ell}N_{j,l,n}^{1,1}(\partial_{t}\psi_{l}+\Omega\psi_{l},\mathcal{R}_{n}^{1})dx
\leq C\epsilon^{2}(\mathcal{E}_{s}+1).
\end{align*}
In summary, we have
\begin{align}
\partial_{t}E_{\ell}=&I_{10}+I_{11}+I_{4}
+\epsilon^{2}\mathcal{O}(1+\mathcal{E}_{s}+\epsilon^{3/2}\mathcal{E}_{s}^{3/2}
+\epsilon^{3}\mathcal{E}_{s}^{2})\nonumber\\
=&\epsilon^{2}\sum_{\substack{j,m,n\in\{0,\pm1\}\\ l\in\{\pm1\}}}
\bigg(\int\partial_{x}^{\ell}\left(B^{1,1}_{j,l,m}(\psi_{l}, \mathcal{R}^{1}_{m})
+\epsilon\mathcal{F}_{j,m}^{1}\right)
\partial_{x}^{\ell}N_{j,l,n}^{1,1}(\psi_{l},\mathcal{R}_{n}^{1})dx\nonumber\\
& \ \ \ \ \ \ \ \ \ \ \ \ \ \ \ +\int \partial_{x}^{\ell}\mathcal{R}_{j}^{1}
\partial_{x}^{\ell}N_{j,l,n}^{1,1}\left(\psi_{l},\left(B^{1,1}_{n,l,m}(\psi_{l}, \mathcal{R}^{1}_{m})
+\epsilon\mathcal{F}_{n,m}^{1}\right)\right)dx\bigg)\nonumber\\
&+\epsilon^{2}\sum_{j,n\in\{0,\pm1\}}
\int\partial_{x}^{\ell}\mathcal{R}_{j}^{1}\partial_{x}^{\ell}\mathcal{F}_{j,n}^{1}dx\nonumber +\epsilon^{2}\mathcal{O}(1+\mathcal{E}_{s}+\epsilon^{3/2}\mathcal{E}_{s}^{3/2}
+\epsilon^{3}\mathcal{E}_{s}^{2}).\nonumber
\end{align}
It can be further written as
\begin{align}\label{ptE}
\partial_{t}E_{\ell}=&\epsilon^{2}\sum_{\substack{m\in\{0,\pm1\}\\ l\in\{\pm1\}}}
\bigg(\int\partial_{x}^{\ell}\left(B^{1,1}_{0,l,m}(\psi_{l}, \mathcal{R}^{1}_{m})
+\epsilon\mathcal{F}_{0,m}^{1}\right)
\partial_{x}^{\ell}N_{0,l,0}^{1,1}(\psi_{l},\mathcal{R}_{0}^{1})dx\nonumber\\
& \ \ \ \ \ \ \ \ \  \ \ \ +\int \partial_{x}^{\ell}\mathcal{R}_{0}^{1}
\partial_{x}^{\ell}N_{0,l,0}^{1,1}\left(\psi_{l},\left(B^{1,1}_{0,l,m}(\psi_{l}, \mathcal{R}^{1}_{m})
+\epsilon\mathcal{F}_{0,m}^{1}\right)\right)dx\bigg)\nonumber\\
&+\epsilon^{2}\sum_{\substack{m\in\{0,\pm1\}\\ j,l\in\{\pm1\}}}
\bigg(\int\partial_{x}^{\ell}\left(B^{1,1}_{j,l,m}(\psi_{l}, \mathcal{R}^{1}_{m})
+\epsilon\mathcal{F}_{j,m}^{1}\right)
\partial_{x}^{\ell}N_{j,l,0}^{1,1}(\psi_{l},\mathcal{R}_{0}^{1})dx\nonumber\\
& \ \ \ \ \ \ \ \ \ \ \ \ \ \ \ +\int \partial_{x}^{\ell}\mathcal{R}_{j}^{1}
\partial_{x}^{\ell}N_{j,l,0}^{1,1}\left(\psi_{l},\left(B^{1,1}_{0,l,m}(\psi_{l}, \mathcal{R}^{1}_{m})
+\epsilon\mathcal{F}_{0,m}^{1}\right)\right)dx\bigg)\nonumber\\
&+\epsilon^{2}\sum_{\substack{m\in\{0,\pm1\}\\ j,l\in\{\pm1\}}}
\bigg(\int\partial_{x}^{\ell}\left(B^{1,1}_{0,l,m}(\psi_{l}, \mathcal{R}^{1}_{m})
+\epsilon\mathcal{F}_{0,m}^{1}\right)
\partial_{x}^{\ell}N_{0,l,j}^{1,1}(\psi_{l},\mathcal{R}_{j}^{1})dx\nonumber\\
& \ \ \ \ \ \ \ \ \ \ \ \ \ \ \ +\int \partial_{x}^{\ell}\mathcal{R}_{0}^{1}
\partial_{x}^{\ell}N_{0,l,j}^{1,1}\left(\psi_{l},\left(B^{1,1}_{j,l,m}(\psi_{l}, \mathcal{R}^{1}_{m})
+\epsilon\mathcal{F}_{j,m}^{1}\right)\right)dx\bigg)\nonumber\\
&+\epsilon^{2}\sum_{\substack{m\in\{0,\pm1\}\\ j,l\in\{\pm1\}}}
\bigg(\int\partial_{x}^{\ell}\left(B^{1,1}_{j,l,m}(\psi_{l}, \mathcal{R}^{1}_{m})
+\epsilon\mathcal{F}_{j,m}^{1}\right)
\partial_{x}^{\ell}N_{j,l,j}^{1,1}(\psi_{l},\mathcal{R}_{j}^{1})dx\nonumber\\
& \ \ \ \ \ \ \ \ \ \ \ \ \ \ \ +\int \partial_{x}^{\ell}\mathcal{R}_{j}^{1}
\partial_{x}^{\ell}N_{j,l,j}^{1,1}\left(\psi_{l},\left(B^{1,1}_{j,l,m}(\psi_{l}, \mathcal{R}^{1}_{m})
+\epsilon\mathcal{F}_{j,m}^{1}\right)\right)dx\bigg)\nonumber\\
&+\epsilon^{2}\sum_{\substack{m\in\{0,\pm1\}\\ j,l\in\{\pm1\}}}
\bigg(\int\partial_{x}^{\ell}\left(B^{1,1}_{j,l,m}(\psi_{l}, \mathcal{R}^{1}_{m})
+\epsilon\mathcal{F}_{j,m}^{1}\right)
\partial_{x}^{\ell}N_{j,l,-j}^{1,1}(\psi_{l},\mathcal{R}_{-j}^{1})dx\nonumber\\
& \ \ \ \ \ \ \ \ \ \ \ \ \ \ \ +\int \partial_{x}^{\ell}\mathcal{R}_{j}^{1}
\partial_{x}^{\ell}N_{j,l,-j}^{1,1}\left(\psi_{l},\left(B^{1,1}_{-j,l,m}(\psi_{l}, \mathcal{R}^{1}_{m})
+\epsilon\mathcal{F}_{-j,m}^{1}\right)\right)dx\bigg)\nonumber\\
&+\epsilon^{2}\sum_{j,n\in\{0,\pm1\}}
\int\partial_{x}^{\ell}\mathcal{R}_{j}^{1}\partial_{x}^{\ell}\mathcal{F}_{j,n}^{1}dx +\epsilon^{2}\mathcal{O}(1+\mathcal{E}_{s}+\epsilon^{3/2}\mathcal{E}_{s}^{3/2}
+\epsilon^{3}\mathcal{E}_{s}^{2})\nonumber\\
=&:\sum_{i=1}^{6}J_{i}+\epsilon^{2}\mathcal{O}(1+\mathcal{E}_{s}+\epsilon^{3/2}\mathcal{E}_{s}^{3/2}
+\epsilon^{3}\mathcal{E}_{s}^{2}).
\end{align}

For $J_{1}$, we have
\begin{align*}
J_{1}=&\epsilon^{2}\sum_{\substack{m\in\{0,\pm1\}\\ l\in\{\pm1\}}}
\bigg(\int\partial_{x}^{\ell}\left(B^{1,1}_{0,l,m}(\psi_{l}, \mathcal{R}^{1}_{m})
+\epsilon\mathcal{F}_{0,m}^{1}\right)
\partial_{x}^{\ell}N_{0,l,0}^{1,1}(\psi_{l},\mathcal{R}_{0}^{1})dx\\
& \ \ \ \ \ \ \ \ \ \ \ \ \ \ \ +\int \partial_{x}^{\ell}\mathcal{R}_{0}^{1}
\partial_{x}^{\ell}N_{0,l,0}^{1,1}\left(\psi_{l},B^{1,1}_{0,l,m}(\psi_{l}, \mathcal{R}^{1}_{m})
+\epsilon\mathcal{F}_{0,m}^{1}\right)dx\bigg)\\
=&\epsilon^{2}\sum_{\substack{l\in\{\pm1\}\\r\in\{1,2\}}}
\bigg(\int\partial_{x}^{\ell}\left(B^{1,1}_{0,l,0}(\psi_{l}, \mathcal{R}^{1}_{0})
+\epsilon\mathcal{F}_{0,0}^{1}\right)
\partial_{x}^{\ell}N_{0,l,0}^{1,1,r}(\psi_{l},\mathcal{R}_{0}^{1})dx\\
& \ \ \ \ \ \ \ \ \ \ \ \ \ \ \ +\int \partial_{x}^{\ell}\mathcal{R}_{0}^{1}
\partial_{x}^{\ell}N_{0,l,0}^{1,1,r}\left(\psi_{l},B^{1,1}_{0,l,0}(\psi_{l}, \mathcal{R}^{1}_{0})
+\epsilon\mathcal{F}_{0,0}^{1}\right)dx\bigg)\\
&+\epsilon^{2}\sum_{\substack{m,l\in\{\pm1\}\\r\in\{1,2\}}}
\bigg(\int\partial_{x}^{\ell}\left(B^{1,1}_{0,l,m}(\psi_{l}, \mathcal{R}^{1}_{m})
+\epsilon\mathcal{F}_{0,m}^{1}\right)
\partial_{x}^{\ell}N_{0,l,0}^{1,1,r}(\psi_{l},\mathcal{R}_{0}^{1})dx\\
& \ \ \ \ \ \ \ \ \ \ \ \ \ \ \ +\int \partial_{x}^{\ell}\mathcal{R}_{0}^{1}
\partial_{x}^{\ell}N_{0,l,0}^{1,1,r}\left(\psi_{l},B^{1,1}_{0,l,m}(\psi_{l}, \mathcal{R}^{1}_{m})
+\epsilon\mathcal{F}_{0,m}^{1}\right)dx\bigg)\\
=&:J^{1}_{11}+J^{2}_{11}+J^{1}_{12}+J^{2}_{12},
\end{align*}
where $N_{0,l,0}^{1,1,r}$ with $r\in\{1, 2\}$ is defined in \eqref{N11r}.

According to \eqref{B0n} and \eqref{F00}, we take
\begin{equation}
\begin{split}\label{B00}
B^{1,1}_{0,l,0}(\psi_{l}, \mathcal{R}^{1}_{0})
+\epsilon\mathcal{F}_{0,0}^{1}=&\frac{1}{q^{2}}\left(\phi_{1}\partial_{x}q^{2}\mathcal{R}^{1}_{0}\right)
+\frac{\gamma-1}{q^{2}}(\partial_{x}q\psi_{l}(\gamma-2-q^{2})\mathcal{R}^{1}_{0}),
\end{split}
\end{equation}
where $\phi_{1}$ is defined in \eqref{phi}. Then we have
\begin{align*}
J^{1}_{11}=&\epsilon^{2}\sum_{l\in\{\pm1\}}
\bigg(\int\partial_{x}^{\ell}\left(B^{1,1}_{0,l,0}(\psi_{l}, \mathcal{R}^{1}_{0})
+\epsilon\mathcal{F}_{0,0}^{1}\right)
\partial_{x}^{\ell}N_{0,l,0}^{1,1,1}(\psi_{l},\mathcal{R}_{0}^{1})dx\\
&+\int \partial_{x}^{\ell}\mathcal{R}_{0}^{1}
\partial_{x}^{\ell}N_{0,l,0}^{1,1,1}\left(\psi_{l},\left(B^{1,1}_{0,l,0}(\psi_{l}, \mathcal{R}^{1}_{0})
+\epsilon\mathcal{F}_{0,0}^{1}\right)\right)dx\bigg)\\
=&\epsilon^{2}\sum_{l\in\{\pm1\}}
\bigg(\int\partial_{x}^{\ell}\frac{1}{q^{2}}(\phi_{1}\partial_{x}q^{2}\mathcal{R}^{1}_{0})
\partial_{x}^{\ell}N_{0,l,0}^{1,1,1}(\psi_{l},\mathcal{R}_{0}^{1})dx\\
&+\int \partial_{x}^{\ell}\mathcal{R}_{0}^{1}
\partial_{x}^{\ell}N_{0,l,0}^{1,1,1}\left(\psi_{l},
\frac{1}{q^{2}}(\phi_{1}\partial_{x}q^{2}\mathcal{R}^{1}_{0})\right)dx\bigg)\\
&+(\gamma-1)\epsilon^{2}\sum_{l\in\{\pm1\}}
\bigg(\int\partial_{x}^{\ell}\frac{1}{q^{2}}(\partial_{x}q\psi_{l}(\gamma-2-q^{2})\mathcal{R}^{1}_{0})
\partial_{x}^{\ell}N_{0,l,0}^{1,1,1}(\psi_{l},\mathcal{R}_{0}^{1})dx\\
&+\int \partial_{x}^{\ell}\mathcal{R}_{0}^{1}
\partial_{x}^{\ell}N_{0,l,0}^{1,1,1}\left(\psi_{l},
\frac{1}{q^{2}}(\partial_{x}q\psi_{l}(\gamma-2-q^{2})\mathcal{R}^{1}_{0})\right)dx\bigg).
\end{align*}
To extract all terms with more than one spatial derivatives falling on $\mathcal{R}^{1}_{0}$, we apply Leibniz's rule and get
\begin{align*}
J^{1}_{11}=&\epsilon^{2}\sum_{l\in\{\pm1\}}
\bigg(\int\partial_{x}^{\ell}\frac{1}{q^{2}}(\phi_{1}\partial_{x}q^{2}\mathcal{R}^{1}_{0})
N_{0,l,0}^{1,1,1}(\psi_{l},\partial_{x}^{\ell}\mathcal{R}_{0}^{1})dx\\
&+\ell\int\partial_{x}^{\ell}\frac{1}{q^{2}}(\phi_{1}\partial_{x}q^{2}\mathcal{R}^{1}_{0})
N_{0,l,0}^{1,1,1}(\partial_{x}\psi_{l},\partial_{x}^{\ell-1}\mathcal{R}_{0}^{1})dx\\
&+\int \partial_{x}^{\ell}\mathcal{R}_{0}^{1}
N_{0,l,0}^{1,1,1}\left(\psi_{l},
\partial_{x}^{\ell}\frac{1}{q^{2}}(\phi_{1}\partial_{x}q^{2}\mathcal{R}^{1}_{0})\right)dx\\
&+\ell\int \partial_{x}^{\ell}\mathcal{R}_{0}^{1}
N_{0,l,0}^{1,1,1}\left(\partial_{x}\psi_{l},
\partial_{x}^{\ell-1}\frac{1}{q^{2}}(\phi_{1}\partial_{x}q^{2}\mathcal{R}^{1}_{0})\right)dx\bigg)\\
&+(\gamma-1)\epsilon^{2}\sum_{l\in\{\pm1\}}
\bigg(\int\partial_{x}^{\ell}\frac{1}{q^{2}}(\partial_{x}q\psi_{l}(\gamma-2-q^{2})\mathcal{R}^{1}_{0})
N_{0,l,0}^{1,1,1}(\psi_{l},\partial_{x}^{\ell}\mathcal{R}_{0}^{1})dx\\
&+\int \partial_{x}^{\ell}\mathcal{R}_{0}^{1}
N_{0,l,0}^{1,1,1}\left(\psi_{l},
\partial_{x}^{\ell}\frac{1}{q^{2}}(\partial_{x}q\psi_{l}(\gamma-2-q^{2})\mathcal{R}^{1}_{0})\right)dx\bigg)\\
&+\epsilon^{2}\mathcal{O}(\mathcal{E}_{s}+\epsilon^{3/2}\mathcal{E}_{s}^{3/2}).
\end{align*}
Applying the first equation in \eqref{17}, we have
\begin{align*}
J^{1}_{11}=&\epsilon^{2}\sum_{l\in\{\pm1\}}
\bigg(\int \partial_{x}^{\ell}\mathcal{R}_{0}^{1}
S_{0,0}^{1}\left(\psi_{l},
\partial_{x}^{\ell}\frac{1}{q^{2}}(\phi_{1}\partial_{x}q^{2}\mathcal{R}^{1}_{0})\right)dx\\
&+(\gamma-1)\epsilon^{2}\sum_{l\in\{\pm1\}}
\bigg(\int \partial_{x}^{\ell}\mathcal{R}_{0}^{1}
S_{0,0}^{1}\left(\psi_{l},
\partial_{x}^{\ell}\frac{1}{q^{2}}(\partial_{x}q\psi_{l}(\gamma-2-q^{2})\mathcal{R}^{1}_{0})\right)dx\bigg)\\
&+2\ell\int\partial_{x}^{\ell}\frac{1}{q^{2}}(\phi_{1}\partial_{x}q^{2}\mathcal{R}^{1}_{0})
N_{0,l,0}^{1,1,1}(\partial_{x}\psi_{l},\partial_{x}^{\ell-1}\mathcal{R}_{0}^{1})dx\\
&+\epsilon^{2}\mathcal{O}(\mathcal{E}_{s}+\epsilon^{3/2}\mathcal{E}_{s}^{3/2}).
\end{align*}
Using the asymptotic expansions \eqref{qomega}, \eqref{hfjj} and \eqref{s00} yields
\begin{equation}
\begin{split}\label{J111}
J^{1}_{11}=&(2\ell+1)\epsilon^{2}\sum_{l\in\{\pm1\}} \int \partial_{x}^{\ell}\mathcal{R}_{0}^{1} \psi_{l}\phi_{1}\partial_{x}^{\ell+1}\mathcal{R}^{1}_{0}dx\\
&-\frac{2(\gamma-1)}{\gamma}\epsilon^{2}\sum_{l\in\{\pm1\}} \bigg(\int \partial_{x}^{\ell}\mathcal{R}_{0}^{1} \psi_{l}\partial_{x}q\psi_{l}\partial_{x}^{\ell}\mathcal{R}^{1}_{0}dx\bigg) +\epsilon^{2}\mathcal{O}(\mathcal{E}_{s}+\epsilon^{3/2}\mathcal{E}_{s}^{3/2})\\
=&\epsilon^{2}\mathcal{O}(\mathcal{E}_{s}+\epsilon^{3/2}\mathcal{E}_{s}^{3/2}),
\end{split}
\end{equation}
where we have used integration by parts.  For $J^{2}_{11}$, we have
\begin{align*}
J^{2}_{11}=&\epsilon^{2}\sum_{l\in\{\pm1\}}
\bigg(\int\partial_{x}^{\ell}\left(B^{1,1}_{0,l,0}(\psi_{l}, \mathcal{R}^{1}_{0})
+\epsilon\mathcal{F}_{0,0}^{1}\right)
\partial_{x}^{\ell}N_{0,l,0}^{1,1,2}(\psi_{l},\mathcal{R}_{0}^{1})dx\\
& \ \ \ \ \ \ \ \ \ \ \ +\int \partial_{x}^{\ell}\mathcal{R}_{0}^{1}
\partial_{x}^{\ell}N_{0,l,0}^{1,1,2}\left(\psi_{l},\left(B^{1,1}_{0,l,0}(\psi_{l}, \mathcal{R}^{1}_{0})
+\epsilon\mathcal{F}_{0,0}^{1}\right)\right)dx\bigg)\\
=&\epsilon^{2}\sum_{l\in\{\pm1\}}
\bigg(\int\partial_{x}^{\ell}\frac{1}{q^{2}}(\phi_{1}\partial_{x}q^{2}\mathcal{R}^{1}_{0})
\partial_{x}^{\ell}N_{0,l,0}^{1,1,2}(\psi_{l},\mathcal{R}_{0}^{1})dx\\
& \ \ \ \ \ \ \ \ \ \ \  \ +\int \partial_{x}^{\ell}\mathcal{R}_{0}^{1}
\partial_{x}^{\ell}N_{0,l,0}^{1,1,2}\left(\psi_{l},
\frac{1}{q^{2}}(\phi_{1}\partial_{x}q^{2}\mathcal{R}^{1}_{0})\right)dx\bigg)\\
&+(\gamma-1)\epsilon^{2}\sum_{l\in\{\pm1\}}
\bigg(\int\partial_{x}^{\ell}\frac{1}{q^{2}}(\partial_{x}q\psi_{l}(\gamma-2-q^{2})\mathcal{R}^{1}_{0})
\partial_{x}^{\ell}N_{0,l,0}^{1,1,2}(\psi_{l},\mathcal{R}_{0}^{1})dx\\
& \ \ \ \ \ \ \ \ \ \ \  \ \ \  \ \ \ +\int \partial_{x}^{\ell}\mathcal{R}_{0}^{1}
\partial_{x}^{\ell}N_{0,l,0}^{1,1,2}\left(\psi_{l},
\frac{1}{q^{2}}(\partial_{x}q\psi_{l}(\gamma-2-q^{2})\mathcal{R}^{1}_{0})\right)dx\bigg).
\end{align*}
Using the asymptotic expansions \eqref{qomega}, \eqref{hfjj} and integration by parts, we have
\begin{equation}
\begin{split}\label{J1121}
J^{2}_{11}
=&\frac{-4(\gamma-1)}{\gamma}\epsilon^{2}\sum_{l\in\{\pm1\}}
\int\phi_{1}\psi_{l}\partial_{x}^{\ell+1}\mathcal{R}^{1}_{0}
\partial_{x}^{\ell}\mathcal{R}_{0}^{1}dx\\
&+\frac{8(\gamma-1)^{2}}{\gamma}\epsilon^{2}\sum_{l\in\{\pm1\}}
\int\partial_{x}q\psi_{1}\psi_{l}(\partial_{x}^{\ell}\mathcal{R}^{1}_{0})^{2}dx
+\epsilon^{2}\mathcal{O}(\mathcal{E}_{s}+\epsilon^{3/2}\mathcal{E}_{s}^{3/2})\\
=&\epsilon^{2}\mathcal{O}(\mathcal{E}_{s}+\epsilon^{3/2}\mathcal{E}_{s}^{3/2}).
\end{split}
\end{equation}
According to \eqref{B0n} and \eqref{F00}, we have
\begin{equation}
\begin{split}\label{B0m}
&\sum_{m\in\{\pm1\}}(B^{1,1}_{0,l,m}(\psi_{l}, \mathcal{R}^{1}_{m})
+\epsilon\mathcal{F}_{0,m}^{1})\\
=&-\frac{\gamma-1}{q^{2}}\left(\phi_{2}\partial_{x}q(\mathcal{R}^{1}_{1}-\mathcal{R}^{1}_{-1})\right)
+\frac{(\gamma-1)(\gamma-2)}{q^{2}}\left(\partial_{x}q\psi_{l}(
\mathcal{R}^{1}_{1}+\mathcal{R}^{1}_{-1})\right),
\end{split}
\end{equation}
where $\phi_{2}$ is defined in \eqref{phi}. For $J^{1}_{12}$, we have
\begin{align*}
J^{1}_{12}=&\epsilon^{2}\sum_{m,l\in\{\pm1\}}
\Bigg(\int\partial_{x}^{\ell}\left(B^{1,1}_{0,l,m}(\psi_{l}, \mathcal{R}^{1}_{m})
+\epsilon\mathcal{F}_{0,m}^{1}\right)
\partial_{x}^{\ell}N_{0,l,0}^{1,1,1}(\psi_{l},\mathcal{R}_{0}^{1})dx\\
& \ \ \ \ \ \ \ \ \ \ \  \ \ \  \ \ \ +\int \partial_{x}^{\ell}\mathcal{R}_{0}^{1}
\partial_{x}^{\ell}N_{0,l,0}^{1,1,1}\left(\psi_{l},\left(B^{1,1}_{0,l,m}(\psi_{l}, \mathcal{R}^{1}_{m})
+\epsilon\mathcal{F}_{0,m}^{1}\right)\right)dx\Bigg)\\
=&-\epsilon^{2}(\gamma-1)\sum_{l\in\{\pm1\}}
\Bigg(\int\partial_{x}^{\ell}\left(\frac{1}{q^{2}}
\left(\phi_{2}\partial_{x}q(\mathcal{R}^{1}_{1}-\mathcal{R}^{1}_{-1})\right)\right)
\partial_{x}^{\ell}N_{0,l,0}^{1,1,1}(\psi_{l},\mathcal{R}_{0}^{1})dx\\
& \ \ \ \ \ \ \ \ \ \ \  \  \ \ \  \  \ \ \ \ \ \ \  \  \ \ \ +\int\partial_{x}^{\ell}\mathcal{R}_{0}^{1}
\partial_{x}^{\ell}N_{0,l,0}^{1,1,1}
\left(\psi_{l},\frac{1}{q^{2}}\left(\phi_{2}\partial_{x}q
(\mathcal{R}^{1}_{1}-\mathcal{R}^{1}_{-1})\right)\right)dx\Bigg)\\
&+(\gamma-1)(\gamma-2)\epsilon^{2}\sum_{l\in\{\pm1\}}
\Bigg(\int\partial_{x}^{\ell}\left(\frac{1}{q^{2}}\left(\partial_{x}q\psi_{l}(
\mathcal{R}^{1}_{1}+\mathcal{R}^{1}_{-1})\right)\right)
\partial_{x}^{\ell}N_{0,l,0}^{1,1,1}(\psi_{l},\mathcal{R}_{0}^{1})dx\\
& \ \ \ \ \ \ \ \ \ \ \  \ \ \  \ \ \  \  \ \ \ \ \ \ \ \ \ \  \  \ \ \  \  \ \ \ +\int\partial_{x}^{\ell}\mathcal{R}_{0}^{1}
\partial_{x}^{\ell}N_{0,l,0}^{1,1,1}
\left(\psi_{l},\frac{1}{q^{2}}\left(\partial_{x}q\psi_{l}(
\mathcal{R}^{1}_{1}+\mathcal{R}^{1}_{-1})\right)\right)dx\Bigg).
\end{align*}
We apply Leibniz's rule and \eqref{17} again to extract all terms with more than one spatial derivatives falling on $\mathcal{R}^{1}_{0}$,
\begin{align*}
J^{1}_{12}
=&-\epsilon^{2}(\gamma-1)\sum_{l\in\{\pm1\}}
\Bigg(\int\partial_{x}^{\ell}\mathcal{R}_{0}^{1}
S_{0,0}^{1}
\left(\psi_{l},\partial_{x}^{\ell}\frac{1}{q^{2}}\left(\phi_{2}\partial_{x}q
(\mathcal{R}^{1}_{1}-\mathcal{R}^{1}_{-1})\right)\right)dx\bigg)\\
& \  \ \ \  \  \ \ \  \  \ \ \  \  \ \ \  \  \ \ \ +2\ell\int\partial_{x}^{\ell}\mathcal{R}_{0}^{1}
N_{0,l,0}^{1,1,1}
\left(\partial_{x}\psi_{l},\partial_{x}^{\ell-1}\frac{1}{q^{2}}\left(\phi_{2}\partial_{x}q
(\mathcal{R}^{1}_{1}-\mathcal{R}^{1}_{-1})\right)\right)dx\Bigg)\\
&+(\gamma-1)(\gamma-2)\epsilon^{2}\sum_{l\in\{\pm1\}}
\int\partial_{x}^{\ell}\mathcal{R}_{0}^{1}
S_{0,0}^{1}
\left(\psi_{l},\partial_{x}^{\ell}\frac{1}{q^{2}}\left(\partial_{x}q\psi_{l}(
\mathcal{R}^{1}_{1}+\mathcal{R}^{1}_{-1})\right)\right)dx\\
&+\epsilon^{2}\mathcal{O}(\mathcal{E}_{s}+\epsilon^{3/2}\mathcal{E}_{s}^{3/2}).
\end{align*}
Using the asymptotic expansions \eqref{qomega}, \eqref{hfjj} and \eqref{s00} again yields
\begin{align}\label{J112}
J^{1}_{12}
=&-\frac{\epsilon^{2}(\gamma-1)(2\ell+1)}{\gamma}\sum_{l\in\{\pm1\}}
\int\psi_{l}\phi_{2}\partial_{x}^{\ell}\mathcal{R}_{0}^{1}
\partial_{x}^{\ell+1}q(\mathcal{R}^{1}_{1}-\mathcal{R}^{1}_{-1})dx \\
&+\frac{\epsilon^{2}(\gamma-1)(\gamma-2)}{\gamma}\sum_{l\in\{\pm1\}}
\int\psi_{l}\partial_{x}\psi_{l}\partial_{x}^{\ell}\mathcal{R}_{0}^{1}
\partial_{x}^{\ell}q(
\mathcal{R}^{1}_{1}+\mathcal{R}^{1}_{-1})dx\nonumber +\epsilon^{2}\mathcal{O}(\mathcal{E}_{s}+\epsilon^{3/2}\mathcal{E}_{s}^{3/2})\\
=&-\frac{\epsilon^{2}(\gamma-1)(2\ell+1)}{\gamma}\sum_{l\in\{\pm1\}}
\int\psi_{l}\phi_{2}\partial_{x}^{\ell}\mathcal{R}_{0}^{1}
\partial_{x}^{\ell+1}q(\mathcal{R}^{1}_{1}-\mathcal{R}^{1}_{-1})dx +\epsilon^{2}\mathcal{O}(\mathcal{E}_{s}+\epsilon^{3/2}\mathcal{E}_{s}^{3/2})\nonumber.
\end{align}
Using \eqref{1+-}-\eqref{1--}, the estimates \eqref{Aesti-1}, \eqref{Aesti-3} in Lemma \ref{L2} and \eqref{Lr}, we have for the term on the right-hand-side of \eqref{J112} that
\begin{equation}
\begin{split}\label{J121}
&\int\psi_{l}\phi_{2}\partial_{x}^{\ell}\mathcal{R}_{0}^{1}
\partial_{x}^{\ell+1}q(\mathcal{R}^{1}_{1}-\mathcal{R}^{1}_{-1})dx\\
=&\int\psi_{l}\phi_{2}\partial_{x}^{\ell}(\mathcal{R}_{0}^{1}
+\mathcal{R}^{1}_{1}+\mathcal{R}^{1}_{-1})
\partial_{x}^{\ell+1}q(\mathcal{R}^{1}_{1}-\mathcal{R}^{1}_{-1})dx\\
&-\int\psi_{l}\phi_{2}\partial_{x}^{\ell}(
\mathcal{R}^{1}_{1}+\mathcal{R}^{1}_{-1})
\partial_{x}^{\ell+1}q(\mathcal{R}^{1}_{1}-\mathcal{R}^{1}_{-1})dx\\
=&\int\frac{\psi_{l}\phi_{2}}{1+\epsilon\phi_{3}}\partial_{x}^{\ell}(\mathcal{R}_{0}^{1}
+\mathcal{R}^{1}_{1}+\mathcal{R}^{1}_{-1})\partial_{x}^{\ell}
\partial_{t}(\mathcal{R}^{1}_{0}
+\mathcal{R}^{1}_{1}+\mathcal{R}^{1}_{-1})dx\\
&-\epsilon\int\frac{\psi_{l}\phi_{2}\phi_{1}}{1+\epsilon\phi_{3}}\partial_{x}^{\ell}(\mathcal{R}_{0}^{1}
+\mathcal{R}^{1}_{1}+\mathcal{R}^{1}_{-1})\partial_{x}^{\ell+1}
(\mathcal{R}^{1}_{0}
+\mathcal{R}^{1}_{1}+\mathcal{R}^{1}_{-1})dx\\
&-\int\frac{\psi_{l}\phi_{2}}{1+\epsilon\phi_{4}}\partial_{x}^{\ell}(
\mathcal{R}^{1}_{1}+\mathcal{R}^{1}_{-1})
\partial_{x}^{\ell}\partial_{t}(\mathcal{R}^{1}_{1}+\mathcal{R}^{1}_{-1})dx\\
&
+\epsilon\int\frac{\psi_{l}\phi_{2}\phi_{1}}{1+\epsilon\phi_{4}}\partial_{x}^{\ell}(
\mathcal{R}^{1}_{1}+\mathcal{R}^{1}_{-1})
\partial_{x}^{\ell+1}(\mathcal{R}^{1}_{1}+\mathcal{R}^{1}_{-1})dx\\
&+\epsilon^{2}\mathcal{O}(\mathcal{E}_{s}+\epsilon^{3/2}\mathcal{E}_{s}^{3/2}
+\epsilon^{3}\mathcal{E}_{s}^{2}).
\end{split}
\end{equation}
Thus, we have
\begin{equation}
\begin{split}\label{J12}
J^{1}_{12}=&-\frac{\epsilon^{2}(\gamma-1)(2\ell+1)}{2\gamma}\frac{d}{dt}\sum_{l\in\{\pm1\}}\Bigg(
\int\frac{\psi_{l}\phi_{2}}{1+\epsilon\phi_{3}}(\partial_{x}^{\ell}(\mathcal{R}_{0}^{1}
+\mathcal{R}^{1}_{1}+\mathcal{R}^{1}_{-1}))^{2}dx\\
& \ \ \ \ \ \ \ \ \ \ \ \ \ \ \ \ \ \ \ \ \ \ \ \ \ \ \ \ \ \ \ \ \ \ \ \ \ \ \ \ \ \ \  -\int\frac{\psi_{l}\phi_{2}}{1+\epsilon\phi_{4}}(\partial_{x}^{\ell}(
\mathcal{R}^{1}_{1}+\mathcal{R}^{1}_{-1}))^{2}dx\Bigg)\\
&+\epsilon^{2}\mathcal{O}(\mathcal{E}_{s}+\epsilon^{3/2}\mathcal{E}_{s}^{3/2}
+\epsilon^{3}\mathcal{E}_{s}^{2}).
\end{split}
\end{equation}
For $J^{2}_{12}$, we have
\begin{align*}
J^{2}_{12}=&\epsilon^{2}\sum_{m,l\in\{\pm1\}}
\bigg(\int\partial_{x}^{\ell}\left(B^{1,1}_{0,l,m}(\psi_{l}, \mathcal{R}^{1}_{m})
+\epsilon\mathcal{F}_{0,m}^{1}\right)
\partial_{x}^{\ell}N_{0,l,0}^{1,1,2}(\psi_{l},\mathcal{R}_{0}^{1})dx\\
& \  \ \ \  \  \ \ \  \  \ \ \  \ \ \ \ \   +\int \partial_{x}^{\ell}\mathcal{R}_{0}^{1}
\partial_{x}^{\ell}N_{0,l,0}^{1,1,2}\left(\psi_{l},\left(B^{1,1}_{0,l,m}(\psi_{l}, \mathcal{R}^{1}_{m})
+\epsilon\mathcal{F}_{0,m}^{1}\right)\right)dx\bigg)\\
=&-\epsilon^{2}(\gamma-1)\sum_{l\in\{\pm1\}}
\bigg(\int\partial_{x}^{\ell}\left(\frac{1}{q^{2}}
\left(\phi_{2}\partial_{x}q(\mathcal{R}^{1}_{1}-\mathcal{R}^{1}_{-1})\right)\right)
\partial_{x}^{\ell}N_{0,l,0}^{1,1,2}(\psi_{l},\mathcal{R}_{0}^{1})dx\\
& \  \ \ \  \  \ \ \  \  \ \ \  \  \ \ \  \  \ \ \  \ \ \ \ \ \ \  +\int\partial_{x}^{\ell}\mathcal{R}_{0}^{1}
\partial_{x}^{\ell}N_{0,l,0}^{1,1,2}
\left(\psi_{l},\frac{1}{q^{2}}\left(\phi_{2}\partial_{x}q
(\mathcal{R}^{1}_{1}-\mathcal{R}^{1}_{-1})\right)\right)dx\bigg)\\
&+(\gamma-1)(\gamma-2)\epsilon^{2}\sum_{l\in\{\pm1\}}
\Bigg(\int\partial_{x}^{\ell}\left(\frac{1}{q^{2}}\left(\partial_{x}q\psi_{l}(
\mathcal{R}^{1}_{1}+\mathcal{R}^{1}_{-1})\right)\right)
\partial_{x}^{\ell}N_{0,l,0}^{1,1,2}(\psi_{l},\mathcal{R}_{0}^{1})dx\\
& \  \ \ \  \  \ \ \  \  \ \ \  \  \ \ \  \  \ \ \  \ \ \ \ \ \  \ \ \  \  \ \ \  \  \ \ \ +\int\partial_{x}^{\ell}\mathcal{R}_{0}^{1}
\partial_{x}^{\ell}N_{0,l,0}^{1,1,2}
\left(\psi_{l},\frac{1}{q^{2}}\left(\partial_{x}q\psi_{l}(
\mathcal{R}^{1}_{1}+\mathcal{R}^{1}_{-1})\right)\right)dx\Bigg).
\end{align*}
Using the asymptotic expansions \eqref{qomega} and \eqref{hfjj}, we have
\begin{align*}
J^{2}_{12}
=&4\epsilon^{2}(\frac{\gamma-1}{\gamma})^{2}\sum_{l\in\{\pm1\}}
\int\psi_{l}\phi_{2}\partial_{x}^{\ell}\mathcal{R}_{0}^{1}
\partial_{x}^{\ell+1}q(\mathcal{R}^{1}_{1}-\mathcal{R}^{1}_{-1})
dx
+\epsilon^{2}\mathcal{O}(\mathcal{E}_{s}+\epsilon^{3/2}\mathcal{E}_{s}^{3/2}).
\end{align*}
As was done for $J^{1}_{12}$, using \eqref{1+-}-\eqref{1--}, the estimates \eqref{Aesti-1}, \eqref{Aesti-3} in Lemma \ref{L2} and \eqref{Lr} again, we have
\begin{align}\label{J122}
J^{2}_{12}=&2\epsilon^{2}(\frac{\gamma-1}{\gamma})^{2}
\frac{d}{dt}\sum_{l\in\{\pm1\}}
\int\Big(\frac{\psi_{l}\phi_{2}}{1+\epsilon\phi_{3}}(\partial_{x}^{\ell}(\mathcal{R}_{0}^{1}
+\mathcal{R}^{1}_{1}+\mathcal{R}^{1}_{-1}))^{2} -\frac{\psi_{l}\phi_{2}}{1+\epsilon\phi_{4}}(\partial_{x}^{\ell}(
\mathcal{R}^{1}_{1}+\mathcal{R}^{1}_{-1}))^{2}\Big)dx\nonumber\\
&+\epsilon^{2}\mathcal{O}(\mathcal{E}_{s}+\epsilon^{3/2}\mathcal{E}_{s}^{3/2}
+\epsilon^{3}\mathcal{E}_{s}^{2}).
\end{align}
Therefore, combing \eqref{J111}, \eqref{J1121}, \eqref{J12} and \eqref{J122}, we have
\begin{align}\label{J1}
J_{1}=&\epsilon^{2}\frac{\gamma-1}{2\gamma}\left(\frac{4(\gamma-1)}{\gamma}-(2\ell+1)\right)
\frac{d}{dt}\sum_{l\in\{\pm1\}}\Bigg(
\int\frac{\psi_{l}\phi_{2}}{1+\epsilon\phi_{3}}(\partial_{x}^{\ell}(\mathcal{R}_{0}^{1}
+\mathcal{R}^{1}_{1}+\mathcal{R}^{1}_{-1}))^{2}dx\nonumber\\
& \ \ \ \ \ \ \ \ \ \ \ \ \ \ \ \ \ \ \ \ \ \ \ \ \ \ \ \ \ \ \ \ \ \ \ \ \ \ \ \ \ \ \ \ \ \ \ \ \ \ \ \ \ \ \ \ \ -\int\frac{\psi_{l}\phi_{2}}{1+\epsilon\phi_{4}}(\partial_{x}^{\ell}(
\mathcal{R}^{1}_{1}+\mathcal{R}^{1}_{-1}))^{2}dx\Bigg)\nonumber\\
&+\epsilon^{2}\mathcal{O}(\mathcal{E}_{s}+\epsilon^{3/2}\mathcal{E}_{s}^{3/2}
+\epsilon^{3}\mathcal{E}_{s}^{2}).
\end{align}

Now we turn to $J_2$ in the equation \eqref{ptE} for $\partial_{t}E_{\ell}$. We have
\begin{align*}
J_{2}=&\epsilon^{2}\sum_{\substack{m\in\{0,\pm1\}\\ j,l\in\{\pm1\}}}
\bigg(\int\partial_{x}^{\ell}\left(B^{1,1}_{j,l,m}(\psi_{l}, \mathcal{R}^{1}_{m})
+\epsilon\mathcal{F}_{j,m}^{1}\right)
\partial_{x}^{\ell}N_{j,l,0}^{1,1}(\psi_{l},\mathcal{R}_{0}^{1})dx\nonumber\\
&+\int \partial_{x}^{\ell}\mathcal{R}_{j}^{1}
\partial_{x}^{\ell}N_{j,l,0}^{1,1}\left(\psi_{l},\left(B^{1,1}_{0,l,m}(\psi_{l}, \mathcal{R}^{1}_{m})
+\epsilon\mathcal{F}_{0,m}^{1}\right)\right)dx\bigg)\nonumber\\
=&\epsilon^{2}\sum_{j,l\in\{\pm1\}}
\bigg(\int\partial_{x}^{\ell}\left(B^{1,1}_{j,l,0}(\psi_{l}, \mathcal{R}^{1}_{0})
+\epsilon\mathcal{F}_{j,0}^{1}\right)
\partial_{x}^{\ell}N_{j,l,0}^{1,1}(\psi_{l},\mathcal{R}_{0}^{1})dx\nonumber\\
&+\int \partial_{x}^{\ell}\mathcal{R}_{j}^{1}
\partial_{x}^{\ell}N_{j,l,0}^{1,1}\left(\psi_{l},\left(B^{1,1}_{0,l,0}(\psi_{l}, \mathcal{R}^{1}_{0})
+\epsilon\mathcal{F}_{0,0}^{1}\right)\right)dx\bigg)\nonumber\\
&+\epsilon^{2}\sum_{\substack{m\in\{\pm1\}\\ j,l\in\{\pm1\}}}
\bigg(\int\partial_{x}^{\ell}\left(B^{1,1}_{j,l,m}(\psi_{l}, \mathcal{R}^{1}_{m})
+\epsilon\mathcal{F}_{j,m}^{1}\right)
\partial_{x}^{\ell}N_{j,l,0}^{1,1}(\psi_{l},\mathcal{R}_{0}^{1})dx\nonumber\\
&+\int \partial_{x}^{\ell}\mathcal{R}_{j}^{1}
\partial_{x}^{\ell}N_{j,l,0}^{1,1}\left(\psi_{l},\left(B^{1,1}_{0,l,m}(\psi_{l}, \mathcal{R}^{1}_{m})
+\epsilon\mathcal{F}_{0,m}^{1}\right)\right)dx\bigg)\\
=&:J_{21}+J_{22}+J_{23}+J_{24}.
\end{align*}
Using \eqref{Bjn}, \eqref{F01} and the asymptotic expansion \eqref{qq2}, we extract the terms that lose derivatives
\begin{equation}
\begin{split}\label{F11j0}
B^{1,1}_{j,l,0}(\psi_{l}, \mathcal{R}^{1}_{0})
+\epsilon\mathcal{F}_{j,0}^{1}
=j\phi_{7}\partial_{x}\mathcal{R}_{0}^{1}+\mathcal{L}^{4},
\end{split}
\end{equation}
\begin{equation}
\begin{split}\label{Fjm0}
\sum_{m\in\{\pm1\}}(B^{1,1}_{j,l,m}(\psi_{l}, \mathcal{R}^{1}_{m})
+\epsilon\mathcal{F}_{j,m}^{1})
=&\phi_{8}\partial_{x}(\mathcal{R}_{1}^{1}+\mathcal{R}_{-1}^{1})
+\phi_{9}\partial_{x}q(\mathcal{R}_{1}^{1}-\mathcal{R}_{-1}^{1})+\mathcal{L}^{5},
\end{split}
\end{equation}
with
\begin{align*}
\phi_{7}:=&\frac{1}{2\sqrt{\gamma}}\phi_{2},\\
\phi_{8}:=&\frac{1}{2}\phi_{1}
+\frac{j}{2\sqrt{\gamma}}\phi_{2},\\
\phi_{9}:=&\frac{1}{2}\phi_{3}
+\frac{1}{2\gamma}\phi_{2}-\frac{j}{2\sqrt{\gamma}}\phi_{1}.
\end{align*}
Here $\phi_{1}, \phi_{2}, \phi_{3}$ are defined in \eqref{phi}, and $\mathcal{L}^{4}$ and $\mathcal{L}^{5}$ denote the terms that do not lose any derivatives and satisfy
\begin{align*}
\|\mathcal{L}^{r}\|_{H^{s}}=\mathcal{O}(1+\mathcal{E}_{s}+\epsilon^{3/2}\mathcal{E}_{s}^{3/2}
+\epsilon^{3}\mathcal{E}_{s}^{2}), \ \ \ \ \text{for} \ r\in\{4,5\}.
\end{align*}
For $J_{21}$, using \eqref{hfjj} and \eqref{F11j0}, we have
\begin{align}\label{J21}
J_{21}
=&\epsilon^{2}\sum_{j,l\in\{\pm1\}}
\int\partial_{x}^{\ell}\left(B^{1,1}_{j,l,0}(\psi_{l}, \mathcal{R}^{1}_{0})
+\epsilon\mathcal{F}_{j,0}^{1}\right)
N_{j,l,0}^{1,1}(\psi_{l},\mathcal{R}_{0}^{1})dx\nonumber\\
=&\epsilon^{2}\sum_{j,l\in\{\pm1\}}
\int\partial_{x}^{\ell}(j\phi_{7}\partial_{x}\mathcal{R}_{0}^{1})
\partial_{x}^{\ell}\left(\frac{\gamma-2}{2\gamma}\psi_{l}\mathcal{R}_{0}^{1}-j\frac{\gamma-2}{2\gamma}
\Omega\psi_{l}\frac{\mathcal{R}_{0}^{1}}{\Omega}
-\frac{\gamma+2}{\gamma}\partial_{x}\psi_{l}\partial_{x}^{-1}\mathcal{R}_{0}^{1}\right)dx\nonumber\\
&+\epsilon^{2}\mathcal{O}(1+\mathcal{E}_{s}+\epsilon^{3/2}\mathcal{E}_{s}^{3/2}
+\epsilon^{3}\mathcal{E}_{s}^{2})\nonumber\\
=&\epsilon^{2}\mathcal{O}(1+\mathcal{E}_{s}+\epsilon^{3/2}\mathcal{E}_{s}^{3/2}
+\epsilon^{3}\mathcal{E}_{s}^{2}),
\end{align}
where we have used integration by parts. For $J_{22}$, using \eqref{hfjj}, \eqref{B00} and integration by parts, we have
\begin{align*}
J_{22}
=&\epsilon^{2}\sum_{j,l\in\{\pm1\}}
\int\partial_{x}^{\ell}\mathcal{R}_{j}^{1}
\partial_{x}^{\ell}N_{j,l,0}^{1,1}\left(\psi_{l},B^{1,1}_{0,l,0}(\psi_{l}, \mathcal{R}^{1}_{0})
+\epsilon\mathcal{F}_{0,0}^{1}\right)dx\\
=&\epsilon^{2}\sum_{j,l\in\{\pm1\}}
\int \partial_{x}^{\ell}\mathcal{R}_{j}^{1}
\partial_{x}^{\ell}N_{j,l,0}^{1,1}\left(\psi_{l},
\frac{1}{q^{2}}(\phi_{1}\partial_{x}q^{2}\mathcal{R}^{1}_{0})\right)dx
+\epsilon^{2}\mathcal{O}(1+\mathcal{E}_{s})\\
=&\epsilon^{2}\frac{2-\gamma}{2\gamma}\sum_{l\in\{\pm1\}}
\int \psi_{l}\phi_{1}\partial_{x}^{\ell+1}(\mathcal{R}_{1}^{1}+\mathcal{R}_{-1}^{1})
\partial_{x}^{\ell}\mathcal{R}^{1}_{0}dx
+\epsilon^{2}\mathcal{O}(1+\mathcal{E}_{s}).
\end{align*}
Using \eqref{1++}, we have
\begin{align}\label{J222}
J_{22}
=&\epsilon^{2}\frac{2-\gamma}{2\gamma}\sum_{l\in\{\pm1\}}\int \frac{\psi_{l}\phi_{1}}{q+\epsilon\phi_{5}}\partial_{x}^{\ell}\mathcal{R}^{1}_{0}
\partial_{x}^{\ell}\left(\partial_{t}(\mathcal{R}_{1}^{1}-\mathcal{R}_{-1}^{1})
-\epsilon\phi_{1}\partial_{x}(\mathcal{R}_{1}^{1}-\mathcal{R}_{-1}^{1})
-\epsilon\phi_{5}\partial_{x}\mathcal{R}_{0}^{1}\right)dx\nonumber\\
&+\epsilon^{2}\mathcal{O}(1+\mathcal{E}_{s})\nonumber\\
=&\epsilon^{2}\frac{2-\gamma}{2\gamma}\sum_{l\in\{\pm1\}}\frac{d}{dt}\int
\frac{\psi_{l}\phi_{1}}{q+\epsilon\phi_{5}}\partial_{x}^{\ell}\mathcal{R}^{1}_{0}
\partial_{x}^{\ell}(\mathcal{R}_{1}^{1}-\mathcal{R}_{-1}^{1})dx\nonumber\\
&-\epsilon^{2}\frac{2-\gamma}{2\gamma}\sum_{l\in\{\pm1\}}\int
\frac{\psi_{l}\phi_{1}}{q+\epsilon\phi_{5}}\partial_{x}^{\ell}\partial_{t}\mathcal{R}^{1}_{0}
\partial_{x}^{\ell}(\mathcal{R}_{1}^{1}-\mathcal{R}_{-1}^{1})dx\nonumber\\
&-\epsilon^{3}\frac{2-\gamma}{2\gamma}\sum_{l\in\{\pm1\}}\int \frac{\psi_{l}\phi_{1}^{2}}{q+\epsilon\phi_{5}}\partial_{x}^{\ell}\mathcal{R}^{1}_{0}
\partial_{x}^{\ell+1}(\mathcal{R}_{1}^{1}-\mathcal{R}_{-1}^{1})
+\epsilon^{2}\mathcal{O}(1+\mathcal{E}_{s}).
\end{align}
According to \eqref{RJ} with \eqref{B00} and \eqref{B0m} for $\partial_{t}\mathcal{R}_{0}^{1}$, the second term on the right-hand-side of the above equation can be written as
\begin{align*}
&-\epsilon^{2}\frac{2-\gamma}{2\gamma}\sum_{l\in\{\pm1\}}\int
\frac{\psi_{l}\phi_{1}}{q+\epsilon\phi_{5}}\partial_{x}^{\ell}\partial_{t}\mathcal{R}^{1}_{0}
\partial_{x}^{\ell}(\mathcal{R}_{1}^{1}-\mathcal{R}_{-1}^{1})dx\\
=&-\epsilon^{2}\frac{2-\gamma}{2\gamma^{2}}\sum_{l\in\{\pm1\}}\int
\frac{\psi_{l}\phi_{1}}{q+\epsilon\phi_{5}}
\partial_{x}^{\ell}(\mathcal{R}_{1}^{1}-\mathcal{R}_{-1}^{1})\partial_{x}^{\ell}
(\epsilon\phi_{1}\partial_{x}\mathcal{R}^{1}_{0}
-\epsilon(\gamma-1)\phi_{2}\partial_{x}q(\mathcal{R}_{1}^{1}-\mathcal{R}_{-1}^{1})dx\\
&+\epsilon^{2}\mathcal{O}(1+\mathcal{E}_{s})\\
=&-\epsilon^{3}\frac{2-\gamma}{2\gamma^{2}}\sum_{l\in\{\pm1\}}\int
\frac{\psi_{l}\phi_{1}^{2}}{q+\epsilon\phi_{5}}
\partial_{x}^{\ell}(\mathcal{R}_{1}^{1}-\mathcal{R}_{-1}^{1})
\partial_{x}^{\ell+1}\mathcal{R}^{1}_{0}dx\\
&+\epsilon^{3}\frac{(2-\gamma)(\gamma-1)}{2\gamma^{2}}\sum_{l\in\{\pm1\}}
\int\frac{\psi_{l}\phi_{1}\phi_{2}}{q+\epsilon\phi_{5}}
\partial_{x}^{\ell}(\mathcal{R}_{1}^{1}-\mathcal{R}_{-1}^{1})
\partial_{x}^{\ell+1}q(\mathcal{R}_{1}^{1}-\mathcal{R}_{-1}^{1})dx
+\epsilon^{2}\mathcal{O}(1+\mathcal{E}_{s})\\
=&\epsilon^{3}\frac{2-\gamma}{2\gamma^{2}}\sum_{l\in\{\pm1\}}\int
\frac{\psi_{l}\phi_{1}^{2}}{q+\epsilon\phi_{5}}\partial_{x}^{\ell}\mathcal{R}^{1}_{0}
\partial_{x}^{\ell+1}(\mathcal{R}_{1}^{1}-\mathcal{R}_{-1}^{1})
dx+\epsilon^{2}\mathcal{O}(1+\mathcal{E}_{s}).
\end{align*}
As was done for \eqref{J112}, using \eqref{1+-} and \eqref{1--}, we have
\begin{align}\label{J22}
J_{22}
=&\epsilon^{2}\frac{2-\gamma}{2\gamma}\sum_{l\in\{\pm1\}}
\frac{d}{dt}\int\frac{\psi_{l}\phi_{1}}{q+\epsilon\phi_{5}}\partial_{x}^{\ell}\mathcal{R}^{1}_{0}
\partial_{x}^{\ell}(\mathcal{R}_{1}^{1}-\mathcal{R}_{-1}^{1})dx\nonumber\\
&+\epsilon^{3}\frac{(2-\gamma)(1-\gamma)\sqrt{\gamma}}{2\gamma^{2}}\sum_{l\in\{\pm1\}}
\frac{d}{dt}\int\frac{\psi_{l}\phi_{1}^{2}}{(q+\epsilon\phi_{5})}
\Big(\frac{1}{1+\epsilon\phi_{3}}(\partial_{x}^{\ell}
(\mathcal{R}_{0}^{1}+\mathcal{R}_{1}^{1}+\mathcal{R}_{-1}^{1}))^{2}\\
& \ \ \ \ \ \ \ \ \ \ \ \ \ \ \ \ \ \ \ \ \ \ \ \ \ \ \ \ \ \ \ \ \ \ \ \ \ \ \ \ \ \ \ \ \ \ \ \ \ \ \ \ \ \ \ \ \ \ \ \ -\frac{1}{1+\epsilon\phi_{4}}
(\partial_{x}^{\ell}(\mathcal{R}_{1}^{1}+\mathcal{R}_{-1}^{1}))^{2}\Big)dx\nonumber\\
&+\epsilon^{2}\mathcal{O}(1+\mathcal{E}_{s}).\nonumber
\end{align}
For $J_{23}$, using \eqref{hfjj}, \eqref{Fjm0} and integration by parts, we have
\begin{align*}
J_{23}
=&\epsilon^{2}\sum_{j,l,m\in\{\pm1\}}
\int\partial_{x}^{\ell}\left(B^{1,1}_{j,l,m}(\psi_{l}, \mathcal{R}^{1}_{m})
+\epsilon\mathcal{F}_{j,m}^{1}\right)
\partial_{x}^{\ell}N_{j,l,0}^{1,1}(\psi_{l},\mathcal{R}_{0}^{1})dx\nonumber\\
=&\epsilon^{2}\frac{\gamma-2}{2\gamma}\sum_{l\in\{\pm1\}}
\int\left(\psi_{l}\phi_{8}\partial_{x}^{\ell+1}(\mathcal{R}_{1}^{1}+\mathcal{R}_{-1}^{1})
+\psi_{l}\phi_{9}\partial_{x}^{\ell+1}q(\mathcal{R}_{1}^{1}-\mathcal{R}_{-1}^{1})\right)
\partial_{x}^{\ell}\mathcal{R}_{0}^{1}dx\nonumber\\
&+\epsilon^{2}\mathcal{O}(1+\mathcal{E}_{s}).
\end{align*}
As done in \eqref{J121} and \eqref{J222}, using \eqref{1+-}-\eqref{1++}, we have
\begin{align}\label{J23}
J_{23}
=&\epsilon^{2}\frac{\gamma-2}{2\gamma}\sum_{l\in\{\pm1\}}
\frac{d}{dt}\bigg(\int\frac{\psi_{l}\phi_{8}}{q+\epsilon\phi_{5}}
\partial_{x}^{\ell}(\mathcal{R}_{1}^{1}-\mathcal{R}_{-1}^{1})\partial_{x}^{\ell}\mathcal{R}_{0}^{1}\nonumber\\
&+\int(
\frac{\epsilon\psi_{l}\phi_{8}\phi_{1}}{\sqrt{\gamma}(q+\epsilon\phi_{5})}
+\psi_{l}\phi_{9})\Big(\frac{1}{1+\epsilon\phi_{3}}
(\partial_{x}^{\ell}(\mathcal{R}_{0}^{1}+\mathcal{R}_{1}^{1}+\mathcal{R}_{-1}^{1}))^{2}\\
&-\frac{1}{1+\epsilon\phi_{4}}
(\partial_{x}^{\ell}(\mathcal{R}_{1}^{1}+\mathcal{R}_{-1}^{1}))^{2}\Big)dx\bigg)
+\epsilon^{2}\mathcal{O}(1+\mathcal{E}_{s}).\nonumber
\end{align}
For $J_{24}$, using \eqref{hfjj}, \eqref{B0m} and \eqref{1++}, we have
\begin{align}\label{J24}
J_{24}=
&\epsilon^{2}\sum_{\substack{m\in\{\pm1\}\\ j,l\in\{\pm1\}}}
\bigg(\int \partial_{x}^{\ell}\mathcal{R}_{j}^{1}
\partial_{x}^{\ell}N_{j,l,0}^{1,1}\left(\psi_{l},\left(B^{1,1}_{0,l,m}(\psi_{l}, \mathcal{R}^{1}_{m})
+\epsilon\mathcal{F}_{0,m}^{1}\right)\right)dx\bigg)\nonumber\\
=&-\epsilon^{2}\frac{(\gamma-1)(\gamma-2)}{2\gamma^{2}}\sum_{l\in\{\pm1\}}
\bigg(\int\psi_{l}\phi_{2}\partial_{x}^{\ell}(\mathcal{R}_{1}^{1}+\mathcal{R}_{-1}^{1})
\partial_{x}^{\ell+1}q(\mathcal{R}_{1}^{1}-\mathcal{R}_{-1}^{1})dx\bigg)
+\epsilon^{2}\mathcal{O}(1+\mathcal{E}_{s})\nonumber\\
=&-\epsilon^{2}\frac{(\gamma-1)(\gamma-2)}{4\gamma^{2}}\sum_{l\in\{\pm1\}}\frac{d}{dt}
\int\frac{\psi_{l}\phi_{2}}{1+\epsilon\phi_{4}}
(\partial_{x}^{\ell}(\mathcal{R}_{1}^{1}+\mathcal{R}_{-1}^{1}))^{2}dx
+\epsilon^{2}\mathcal{O}(1+\mathcal{E}_{s}).
\end{align}
Plusing \eqref{J21}, \eqref{J22}, \eqref{J23} and \eqref{J24} together, we have
\begin{align}\label{J2}
J_{2} =& \epsilon^{2}\sum_{l\in\{\pm1\}}\frac{d}{dt}\Bigg( \int\frac{2-\gamma}{\gamma}\frac{\psi_{l}\phi_{1}}{q+\epsilon\phi_{5}}\partial_{x}^{\ell}\mathcal{R}^{1}_{0} \partial_{x}^{\ell}(\mathcal{R}_{1}^{1}-\mathcal{R}_{-1}^{1})dx\nonumber\\
&\ \ \ \ \ \ \ \ \ \ \ \ \ \ \ \ \  +\int\phi_{9}(\partial_{x}^{\ell}(\mathcal{R}_{0}^{1}+\mathcal{R}_{1}^{1}+\mathcal{R}_{-1}^{1}))^{2}dx +\int\phi_{10}(\partial_{x}^{\ell}(\mathcal{R}_{1}^{1}+\mathcal{R}_{-1}^{1}))^{2}dx\Bigg),
\end{align}
where $\phi_{9}$ and $\phi_{10}$ depend on $\gamma, \psi_{l}$ and $\phi_{r}$ with $r\in\{1\sim5,8,9\}$.

For $J_{3}$, we have
\begin{align*}
J_{3}=&\epsilon^{2}\sum_{\substack{m\in\{0,\pm1\}\\ j,l\in\{\pm1\}}}
\bigg(\int\partial_{x}^{\ell}\left(B^{1,1}_{0,l,m}(\psi_{l}, \mathcal{R}^{1}_{m})
+\epsilon\mathcal{F}_{0,m}^{1}\right)
\partial_{x}^{\ell}N_{0,l,j}^{1,1}(\psi_{l},\mathcal{R}_{j}^{1})dx\nonumber\\
& \ \ \ \ \ \ \ \ \ \ \ \ \ \ \ \ \ +\int \partial_{x}^{\ell}\mathcal{R}_{0}^{1}
\partial_{x}^{\ell}N_{0,l,j}^{1,1}\left(\psi_{l},\left(B^{1,1}_{j,l,m}(\psi_{l}, \mathcal{R}^{1}_{m})
+\epsilon\mathcal{F}_{j,m}^{1}\right)\right)dx\bigg)\nonumber\\
=&\epsilon^{2}\sum_{j,l\in\{\pm1\}}
\bigg(\int\partial_{x}^{\ell}\left(B^{1,1}_{0,l,0}(\psi_{l}, \mathcal{R}^{1}_{0})
+\epsilon\mathcal{F}_{0,0}^{1}\right)
\partial_{x}^{\ell}N_{0,l,j}^{1,1}(\psi_{l},\mathcal{R}_{j}^{1})dx\nonumber\\
& \ \ \ \ \ \ \ \ \ \ \ \ \ \ \ +\int \partial_{x}^{\ell}\mathcal{R}_{0}^{1}
\partial_{x}^{\ell}N_{0,l,j}^{1,1}\left(\psi_{l},\left(B^{1,1}_{j,l,0}(\psi_{l}, \mathcal{R}^{1}_{0})
+\epsilon\mathcal{F}_{j,0}^{1}\right)\right)dx\bigg)\nonumber\\
&+\epsilon^{2}\sum_{\substack{m\in\{\pm1\}\\ j,l\in\{\pm1\}}}
\bigg(\int\partial_{x}^{\ell}\left(B^{1,1}_{0,l,m}(\psi_{l}, \mathcal{R}^{1}_{m})
+\epsilon\mathcal{F}_{0,m}^{1}\right)
\partial_{x}^{\ell}N_{0,l,j}^{1,1}(\psi_{l},\mathcal{R}_{j}^{1})dx\nonumber\\
& \ \ \ \ \ \ \ \ \ \ \ \ \ \ \ \ \ \ \ +\int \partial_{x}^{\ell}\mathcal{R}_{0}^{1}
\partial_{x}^{\ell}N_{0,l,j}^{1,1}\left(\psi_{l},\left(B^{1,1}_{j,l,m}(\psi_{l}, \mathcal{R}^{1}_{m})
+\epsilon\mathcal{F}_{j,m}^{1}\right)\right)dx\bigg).
\end{align*}
As was done for $J_{2}$, using \eqref{hfjj} for the normal-form transform $N_{0,l,j}^{1,1}$ and \eqref{B00}, \eqref{B0m}, \eqref{F11j0}-\eqref{Fjm0}, we have
 \begin{align}\label{J3}
J_{3}
&=\epsilon^{2}\sum_{l\in\{\pm1\}}
\frac{d}{dt}\Bigg(
\int\phi_{11}\partial_{x}^{\ell}\mathcal{R}^{1}_{0}
\partial_{x}^{\ell}(\mathcal{R}_{1}^{1}-\mathcal{R}_{-1}^{1})dx\nonumber\\
& \ \ \ \ \ \ \ \ \ \ \ \ \  \ \ \ \ \ \ \ +\int\phi_{12}(\partial_{x}^{\ell}(\mathcal{R}_{0}^{1}+\mathcal{R}_{1}^{1}+\mathcal{R}_{-1}^{1}))^{2}dx
+\int\phi_{13}(\partial_{x}^{\ell}(\mathcal{R}_{1}^{1}+\mathcal{R}_{-1}^{1}))^{2}dx\Bigg),
\end{align}
where $\phi_{11}\sim\phi_{13}$ depend on $\gamma, \psi_{l}$ and $\phi_{r}$ with $r\in\{1\sim5,8,9\}$.

For $J_{4}$, we have
\begin{align*}
J_{4}=&
\epsilon^{2}\sum_{\substack{m\in\{0,\pm1\}\\ j,l\in\{\pm1\}}}
\bigg(\int\partial_{x}^{\ell}\left(B^{1,1}_{j,l,m}(\psi_{l}, \mathcal{R}^{1}_{m})
+\epsilon\mathcal{F}_{j,m}^{1}\right)
\partial_{x}^{\ell}N_{j,l,j}^{1,1}(\psi_{l},\mathcal{R}_{j}^{1})dx\nonumber\\
& \ \ \ \ \ \ \ \ \ \ \ \  \ \ \ \ \  \  +\int \partial_{x}^{\ell}\mathcal{R}_{j}^{1}
\partial_{x}^{\ell}N_{j,l,j}^{1,1}\left(\psi_{l},\left(B^{1,1}_{j,l,m}(\psi_{l}, \mathcal{R}^{1}_{m})
+\epsilon\mathcal{F}_{j,m}^{1}\right)\right)dx\bigg)\nonumber\\
=&\epsilon^{2}\sum_{j,l\in\{\pm1\}}\sum_{r=1}^{7}
\bigg(\int\partial_{x}^{\ell}\left(B^{1,1}_{j,l,0}(\psi_{l}, \mathcal{R}^{1}_{0})
+\epsilon\mathcal{F}_{j,0}^{1}\right)
\partial_{x}^{\ell}N_{j,l,j}^{1,1,r}(\psi_{l},\mathcal{R}_{j}^{1})dx\nonumber\\
& \ \ \ \ \ \ \ \ \ \ \ \  \ \ \ \ \ \ \ \ \ +\int \partial_{x}^{\ell}\mathcal{R}_{j}^{1}
\partial_{x}^{\ell}N_{j,l,j}^{1,1,r}\left(\psi_{l},\left(B^{1,1}_{j,l,0}(\psi_{l}, \mathcal{R}^{1}_{0})
+\epsilon\mathcal{F}_{j,0}^{1}\right)\right)dx\bigg)\nonumber\\
&+\epsilon^{2}\sum_{\substack{m\in\{\pm1\}\\ j,l\in\{\pm1\}}}\sum_{r=1}^{7}
\bigg(\int\partial_{x}^{\ell}\left(B^{1,1}_{j,l,m}(\psi_{l}, \mathcal{R}^{1}_{m})
+\epsilon\mathcal{F}_{j,m}^{1}\right)
\partial_{x}^{\ell}N_{j,l,j}^{1,1,r}(\psi_{l},\mathcal{R}_{j}^{1})dx\nonumber\\
& \ \ \ \ \ \ \ \ \ \ \ \ \ \ \ \ \ \ \ \ \ \ \ \ \ +\int \partial_{x}^{\ell}\mathcal{R}_{j}^{1}
\partial_{x}^{\ell}N_{j,l,j}^{1,1,r}\left(\psi_{l},\left(B^{1,1}_{j,l,m}(\psi_{l}, \mathcal{R}^{1}_{m})
+\epsilon\mathcal{F}_{j,m}^{1}\right)\right)dx\bigg)\nonumber\\
=&:\sum_{r=1}^{7}(J^{r}_{41}+J^{r}_{42}).
\end{align*}
Using Leibniz's rule and \eqref{17} to extract all terms with more than one spatial derivatives falling on $\mathcal{R}^{1}$, we have for $r\in\{1,2,3,4,5\}$ that
\begin{align*}
J^{r}_{41}
=&\epsilon^{2}\sum_{j,l\in\{\pm1\}}
\bigg(\int\partial_{x}^{\ell}\left(B^{1,1}_{j,l,0}(\psi_{l}, \mathcal{R}^{1}_{0})
+\epsilon\mathcal{F}_{j,0}^{1}\right)
\partial_{x}^{\ell}N_{j,l,j}^{1,1,r}(\psi_{l},\mathcal{R}_{j}^{1})dx\nonumber\\
& \ \ \ \ \ \ \ \ \ \ \ \ \ \ \ \ +\int \partial_{x}^{\ell}\mathcal{R}_{j}^{1}
\partial_{x}^{\ell}N_{j,l,j}^{1,1,r}\left(\psi_{l},B^{1,1}_{j,l,0}(\psi_{l}, \mathcal{R}^{1}_{0})
+\epsilon\mathcal{F}_{j,0}^{1}\right)dx\bigg)\nonumber\\
=&\epsilon^{2}\sum_{j,l\in\{\pm1\}}
\bigg(\int\partial_{x}^{\ell}\left(B^{1,1}_{j,l,0}(\psi_{l}, \mathcal{R}^{1}_{0})
+\epsilon\mathcal{F}_{j,0}^{1}\right)
N_{j,l,j}^{1,1,r}(\psi_{l},\partial_{x}^{\ell}\mathcal{R}_{j}^{1})dx\nonumber\\
& \ \ \ \ \ \ \ \ \ \ \ \ \ \ \  +\int \partial_{x}^{\ell}\mathcal{R}_{j}^{1}
N_{j,l,j}^{1,1,r}\left(\psi_{l},\partial_{x}^{\ell}(B^{1,1}_{j,l,0}(\psi_{l}, \mathcal{R}^{1}_{0})
+\epsilon\mathcal{F}_{j,0}^{1})\right)dx\bigg)\nonumber\\
& \ \ \ \ \ \ \ \ \ \ \ \ \ \ \ +\ell\int\partial_{x}^{\ell}\left(B^{1,1}_{j,l,0}(\psi_{l}, \mathcal{R}^{1}_{0})
+\epsilon\mathcal{F}_{j,0}^{1}\right)
N_{j,l,j}^{1,1,r}(\partial_{x}\psi_{l},\partial_{x}^{\ell-1}\mathcal{R}_{j}^{1})dx\nonumber\\
& \ \ \ \ \ \ \ \ \ \ \ \ \ \ \ +\ell\int \partial_{x}^{\ell}\mathcal{R}_{j}^{1}
N_{j,l,j}^{1,1,r}\left(\partial_{x}\psi_{l},\partial_{x}^{\ell-1}(B^{1,1}_{j,l,0}(\psi_{l}, \mathcal{R}^{1}_{0})
+\epsilon\mathcal{F}_{j,0}^{1})\right)dx\bigg)\nonumber\\
&+\epsilon^{2}\mathcal{O}(\mathcal{E}_{s}+\epsilon^{3/2}\mathcal{E}_{s}^{3/2})\\
=&\epsilon^{2}\sum_{j,l\in\{\pm1\}}\bigg(\int \partial_{x}^{\ell}\mathcal{R}_{j}^{1}
S_{j,j}^{r}\left(\partial_{x}\psi_{l},\partial_{x}^{\ell}\left(B^{1,1}_{j,l,0}(\psi_{l}, \mathcal{R}^{1}_{0})
+\epsilon\mathcal{F}_{j,0}^{1}\right)\right)dx\nonumber\\
& \ \ \ \ \ \ \ \ \ \ \ \ \ \ \ +2\ell\int \partial_{x}^{\ell}\mathcal{R}_{j}^{1}
N_{j,l,j}^{1,1,r}\left(\partial_{x}\psi_{l},\partial_{x}^{\ell-1}(B^{1,1}_{j,l,0}(\psi_{l}, \mathcal{R}^{1}_{0})
+\epsilon\mathcal{F}_{j,0}^{1})\right)dx\bigg)\nonumber\\
&+\epsilon^{2}\mathcal{O}(\mathcal{E}_{s}+\epsilon^{3/2}\mathcal{E}_{s}^{3/2}).
\end{align*}
Using \eqref{hfjj}, \eqref{Sjj} and \eqref{F11j0}, we have
\begin{align*}
J^{1}_{41}
=&\epsilon^{2}j\sum_{j,l\in\{\pm1\}}\bigg(\int \partial_{x}^{\ell}\mathcal{R}_{j}^{1}
G_{j,j}^{1}\partial_{x}q\psi_{l} \ \partial_{x}^{\ell}(\phi_{7}\partial_{x}\mathcal{R}_{0}^{1})dx\nonumber\\
& \ \ \ \ \ \ \ \ \ \ \ \ \ \ \ \ \ \ +2\ell\int \partial_{x}^{\ell}\mathcal{R}_{j}^{1}
\partial_{x}\left(G_{j,j}^{1}\partial_{x}q\psi_{l} \ \partial_{x}^{\ell-1}(\phi_{7}\partial_{x}\mathcal{R}_{0}^{1}) \right)dx\bigg)\nonumber +\epsilon^{2}\mathcal{O}(\mathcal{E}_{s}+\epsilon^{3/2}\mathcal{E}_{s}^{3/2})\\
=&-\epsilon^{2}j(2\ell+1)\sum_{j,l\in\{\pm1\}}\int (G_{j,j}^{1}\partial_{x}q\psi_{l})\phi_{7} \partial_{x}^{\ell+1}\mathcal{R}_{j}^{1}
\partial_{x}^{\ell}\mathcal{R}_{0}^{1}dx\nonumber +\epsilon^{2}\mathcal{O}(\mathcal{E}_{s}+\epsilon^{3/2}\mathcal{E}_{s}^{3/2})\\
=&\frac{-\epsilon^{2}(2\ell+1)}{2}\sum_{l\in\{\pm1\}}\int\Big((G_{1,1}^{1}+G_{-1,-1}^{1})\partial_{x}q\psi_{l}
\partial_{x}^{\ell+1}(\mathcal{R}_{1}^{1}-\mathcal{R}_{-1}^{1})\\
& \ \ \ \ \ \ \ \ \ \ \ \  \ \ \ \ \ \ \ \ \ \ \ \ \ \ \ +(G_{1,1}^{1}-G_{-1,-1}^{1})\partial_{x}q\psi_{l}
\partial_{x}^{\ell+1}(\mathcal{R}_{1}^{1}+\mathcal{R}_{-1}^{1})\Big)
\phi_{7}\partial_{x}^{\ell}\mathcal{R}_{0}^{1}dx\\
&+\epsilon^{2}\mathcal{O}(\mathcal{E}_{s}+\epsilon^{3/2}\mathcal{E}_{s}^{3/2}).
\end{align*}
As was done in \eqref{J121} and \eqref{J222}, using \eqref{1+-}-\eqref{1++}, \eqref{part4}-\eqref{part5} and the estimates \eqref{Aesti-3} in Lemma \ref{L2}, we have
 \begin{align}\label{J411}
J_{41}^{1}
=&-\frac{\epsilon^{2}(2\ell+1)}{4}\sum_{l\in\{\pm1\}}
\frac{d}{dt}
\int\frac{\big((G_{1,1}^{1}+G_{-1,-1}^{1})\partial_{x}q\psi_{l}\big)\phi_{7}}{q(1+\epsilon\phi_{3})}
(\partial_{x}^{\ell}
(\mathcal{R}_{0}^{1}+\mathcal{R}_{1}^{1}+\mathcal{R}_{-1}^{1}))^{2}dx\nonumber\\
&+\frac{\epsilon^{2}(2\ell+1)}{4}\sum_{l\in\{\pm1\}}
\frac{d}{dt}
\int\frac{\big((G_{1,1}^{1}+G_{-1,-1}^{1})\partial_{x}q\psi_{l}\big)\phi_{7}}{q(1+\epsilon\phi_{4})}
(\partial_{x}^{\ell}
(\mathcal{R}_{1}^{1}+\mathcal{R}_{-1}^{1}))^{2}dx\nonumber\\
&-\frac{\epsilon^{2}(2\ell+1)}{4}\sum_{l\in\{\pm1\}}
\frac{d}{dt}
\int\frac{\big((G_{1,1}^{1}-G_{-1,-1}^{1})\partial_{x}q\psi_{l}\big)\phi_{7}}{q+\epsilon\phi_{5}}
\partial_{x}^{\ell}\mathcal{R}^{1}_{0}
\partial_{x}^{\ell}(\mathcal{R}_{1}^{1}-\mathcal{R}_{-1}^{1})dx\nonumber\\
&-\frac{\epsilon^{3}(2\ell+1)}{2}\sum_{l\in\{\pm1\}}
\frac{d}{dt}
\int\frac{\big((G_{1,1}^{1}-G_{-1,-1}^{1})\partial_{x}q\psi_{l}\big)\phi_{7}\phi_{1}}
{(q+\epsilon\phi_{5})q(1+\epsilon\phi_{3})}
(\partial_{x}^{\ell}
(\mathcal{R}_{0}^{1}+\mathcal{R}_{1}^{1}+\mathcal{R}_{-1}^{1}))^{2}dx\nonumber\\
&+\frac{\epsilon^{3}(2\ell+1)}{2}\sum_{l\in\{\pm1\}}
\frac{d}{dt}
\int\frac{\big((G_{1,1}^{1}-G_{-1,-1}^{1})\partial_{x}q\psi_{l}\big)\phi_{7}\phi_{1}}
{(q+\epsilon\phi_{5})q(1+\epsilon\phi_{4})}(\partial_{x}^{\ell}
(\mathcal{R}_{1}^{1}+\mathcal{R}_{-1}^{1}))^{2}dx\nonumber\\
&+\epsilon^{2}\mathcal{O}(\mathcal{E}_{s}+\epsilon^{3/2}\mathcal{E}_{s}^{3/2}).
\end{align}
As was done for $J^{1}_{41}$, $J^{r}_{41}$ with $r\in\{2,3,4,5\}$ has similar estimates. In addition, we see that $N_{j,l,j}^{1,1,6}$ does not lose derivatives and $N_{j,l,j}^{1,1,7}$ even gains two derivatives in light of \eqref{hfjj}, so we finally obtain
 \begin{align}\label{J41}
\sum_{r=1}^{7}J^{r}_{41}
=& \epsilon^{2}\sum_{l\in\{\pm1\}}
\frac{d}{dt}\Bigg(
\int\phi_{14}\partial_{x}^{\ell}\mathcal{R}^{1}_{0}
\partial_{x}^{\ell}(\mathcal{R}_{1}^{1}-\mathcal{R}_{-1}^{1})dx +\int\phi_{15}(\partial_{x}^{\ell}(\mathcal{R}_{0}^{1}+\mathcal{R}_{1}^{1}+\mathcal{R}_{-1}^{1}))^{2}dx\nonumber\\
& \ \ \ \ \ \ \ \ \ \ \ \ \ \ \ \ \ \
+\int\phi_{16}(\partial_{x}^{\ell}(\mathcal{R}_{1}^{1}+\mathcal{R}_{-1}^{1}))^{2}dx\Bigg) +\epsilon^{2}\mathcal{O}(\mathcal{E}_{s}+\epsilon^{3/2}\mathcal{E}_{s}^{3/2}),
\end{align}
where $\phi_{14}\sim\phi_{16}$ depend on $\gamma,\psi_{l}$ and $\phi_{r}$ with $r\in\{1\sim5,8,9\}$.  For $J^{r}_{42}$ with $r\in\{1,2,3,4,5\}$, using Leibniz's rule and \eqref{17} again to extract all terms with more than one spatial derivatives falling on $\mathcal{R}^{1}$, we have
\begin{align*}
J^{r}_{42}=&\epsilon^{2}\sum_{\substack{m\in\{\pm1\}\\ j,l\in\{\pm1\}}}
\bigg(\int\partial_{x}^{\ell}\left(B^{1,1}_{j,l,m}(\psi_{l}, \mathcal{R}^{1}_{m})
+\epsilon\mathcal{F}_{j,m}^{1}\right)
\partial_{x}^{\ell}N_{j,l,j}^{1,1,r}(\psi_{l},\mathcal{R}_{j}^{1})dx\nonumber\\
& \ \ \ \ \ \ \ \ \ \ \ \ \ \ \ \ \ +\int \partial_{x}^{\ell}\mathcal{R}_{j}^{1}
\partial_{x}^{\ell}N_{j,l,j}^{1,1,r}\left(\psi_{l},\left(B^{1,1}_{j,l,m}(\psi_{l}, \mathcal{R}^{1}_{m})
+\epsilon\mathcal{F}_{j,m}^{1}\right)\right)dx\bigg)\nonumber\\
=&\epsilon^{2}\sum_{\substack{m\in\{\pm1\}\\ j,l\in\{\pm1\}}}
\bigg(\int\partial_{x}^{\ell}\left(B^{1,1}_{j,l,m}(\psi_{l}, \mathcal{R}^{1}_{m})
+\epsilon\mathcal{F}_{j,m}^{1}\right)
N_{j,l,j}^{1,1,r}(\psi_{l},\partial_{x}^{\ell}\mathcal{R}_{j}^{1})dx\nonumber\\
& \ \ \ \ \ \ \ \ \ \ \ \ \ \ \ \ \ +\int \partial_{x}^{\ell}\mathcal{R}_{j}^{1}
N_{j,l,j}^{1,1,r}\left(\psi_{l},\partial_{x}^{\ell}\left(B^{1,1}_{j,l,m}(\psi_{l}, \mathcal{R}^{1}_{m})
+\epsilon\mathcal{F}_{j,m}^{1}\right)\right)dx\nonumber\\
& \ \ \ \ \ \ \ \ \ \ \ \ \ \ \ \ \ +\ell
\int\partial_{x}^{\ell}\left(B^{1,1}_{j,l,m}(\psi_{l}, \mathcal{R}^{1}_{m})
+\epsilon\mathcal{F}_{j,m}^{1}\right)
N_{j,l,j}^{1,1,r}(\partial_{x}\psi_{l},\partial_{x}^{\ell-1}\mathcal{R}_{j}^{1})dx\nonumber\\
& \ \ \ \ \ \ \ \ \ \ \ \ \ \ \ \ \ +\ell\int \partial_{x}^{\ell}\mathcal{R}_{j}^{1}
N_{j,l,j}^{1,1,r}\left(\partial_{x}\psi_{l},\partial_{x}^{\ell-1}\left(B^{1,1}_{j,l,m}(\psi_{l}, \mathcal{R}^{1}_{m})
+\epsilon\mathcal{F}_{j,m}^{1}\right)\right)dx\bigg)\nonumber\\
&+\epsilon^{2}\mathcal{O}(\mathcal{E}_{s}+\epsilon^{3/2}\mathcal{E}_{s}^{3/2})\\
=&\epsilon^{2}\sum_{\substack{m\in\{\pm1\}\\ j,l\in\{\pm1\}}}\bigg(\int \partial_{x}^{\ell}\mathcal{R}_{j}^{1}
S_{j,j}^{r}\left(\partial_{x}\psi_{l},\partial_{x}^{\ell}\left(B^{1,1}_{j,l,m}(\psi_{l}, \mathcal{R}^{1}_{m})
+\epsilon\mathcal{F}_{j,m}^{1}\right)\right)dx\nonumber\\
& \ \ \ \ \ \ \ \ \ \ \ \ \ \ \ +2\ell\int \partial_{x}^{\ell}\mathcal{R}_{j}^{1}
N_{j,l,j}^{1,1,r}\left(\partial_{x}\psi_{l},\partial_{x}^{\ell-1}\left(B^{1,1}_{j,l,m}(\psi_{l}, \mathcal{R}^{1}_{m})
+\epsilon\mathcal{F}_{j,m}^{1}\right)\right)dx\bigg)\nonumber\\
&+\epsilon^{2}\mathcal{O}(\mathcal{E}_{s}+\epsilon^{3/2}\mathcal{E}_{s}^{3/2}).
\end{align*}
Using \eqref{hfjj}, \eqref{Sjj}, \eqref{Fjm0} and \eqref{part2} in Lemma \ref{L9}, we have
\begin{align*}
J^{1}_{42}=&\epsilon^{2}\sum_{\substack{m\in\{\pm1\}\\ j,l\in\{\pm1\}}}\Bigg(\int \partial_{x}^{\ell}\mathcal{R}_{j}^{1}
G_{j,j}^{1}\left(\partial_{x}q\psi_{l} \ \partial_{x}^{\ell}\left(B^{1,1}_{j,l,m}(\psi_{l}, \mathcal{R}^{1}_{m})
+\epsilon\mathcal{F}_{j,m}^{1}\right)\right)dx\nonumber\\
& \ \ \ \ \ \ \ \ \ \ \ \ \ \ \ +2\ell\int \partial_{x}^{\ell}\mathcal{R}_{j}^{1}
N_{j,l,j}^{1,1,r}\left(\partial_{x}\psi_{l},\partial_{x}^{\ell-1}\left(B^{1,1}_{j,l,m}(\psi_{l}, \mathcal{R}^{1}_{m})
+\epsilon\mathcal{F}_{j,m}^{1}\right)\right)dx\Bigg)\nonumber\\
&+\epsilon^{2}\mathcal{O}(\mathcal{E}_{s}+\epsilon^{3/2}\mathcal{E}_{s}^{3/2})\\
=&\epsilon^{2}(2\ell+1)\sum_{j,l\in\{\pm1\}}\int
\partial_{x}^{\ell}\mathcal{R}_{j}^{1}
(G_{j,j}^{1}\partial_{x}q\psi_{l})(\phi_{8}\partial_{x}^{\ell+1}(\mathcal{R}_{j}^{1}
+\mathcal{R}_{-j}^{1})
+j\phi_{9}\partial_{x}^{\ell+1}(\mathcal{R}_{j}^{1}-\mathcal{R}_{-j}^{1}))dx\\
&+\epsilon^{2}\mathcal{O}(\mathcal{E}_{s}+\epsilon^{3/2}\mathcal{E}_{s}^{3/2})\\
=&\epsilon^{2}(2\ell+1)\sum_{j,l\in\{\pm1\}}\int
\partial_{x}^{\ell}\mathcal{R}_{j}^{1}
(G_{j,j}^{1}\partial_{x}q\psi_{l})(\phi_{8}\partial_{x}^{\ell+1}
\mathcal{R}_{-j}^{1}
-j\phi_{9}\partial_{x}^{\ell+1}\mathcal{R}_{-j}^{1})dx\\
&+\epsilon^{2}\mathcal{O}(\mathcal{E}_{s}+\epsilon^{3/2}\mathcal{E}_{s}^{3/2})\\
=&\frac{\epsilon^{2}(2\ell+1)}{2}\sum_{l\in\{\pm1\}}\int
((G_{-1,-1}^{1}-G_{1,1}^{1})\partial_{x}q\psi_{l})\phi_{8}
+((G_{-1,-1}^{1}+G_{1,1}^{1})\partial_{x}q\psi_{l})\phi_{9})\\
& \ \ \ \ \ \ \ \ \ \ \ \ \ \ \ \ \ \ \ \ \  \ \ \ \times\partial_{x}^{\ell}(\mathcal{R}_{1}^{1}+\mathcal{R}_{-1}^{1})
\partial_{x}^{\ell+1}(\mathcal{R}_{1}^{1}-\mathcal{R}_{-1}^{1})dx +\epsilon^{2}\mathcal{O}(\mathcal{E}_{s}+\epsilon^{3/2}\mathcal{E}_{s}^{3/2}).
\end{align*}
Using \eqref{1--}, we have
\begin{align*}
J^{1}_{42}=&\epsilon^{2}\sum_{l\in\{\pm1\}}\frac{d}{dt}\int
((G_{-1,-1}^{1}-G_{1,1}^{1})\partial_{x}q\psi_{l})\phi_{8}
+((G_{-1,-1}^{1}+G_{1,1}^{1})\partial_{x}q\psi_{l})\phi_{9})\\
& \ \ \ \ \ \ \ \ \ \ \ \ \ \ \ \ \
\times(\partial_{x}^{\ell}(\mathcal{R}_{1}^{1}+\mathcal{R}_{-1}^{1}))^{2}dx +\epsilon^{2}\mathcal{O}(\mathcal{E}_{s}+\epsilon^{3/2}\mathcal{E}_{s}^{3/2}).
\end{align*}
As was done for $J^{1}_{42}$, $J^{r}_{42}$ with $r\in\{2,3,4,5\}$ has similar estimates. In addition, we see that $N_{j,l,j}^{1,1,6}$ does not lose derivatives and $N_{j,l,j}^{1,1,7}$ even gains two derivatives in light of \eqref{hfjj}, so we obtain
\begin{align}\label{J42}
\sum_{r=1}^{7}J^{r}_{42}=&\epsilon^{2}\sum_{l\in\{\pm1\}}\frac{d}{dt}\int
\phi_{17}
(\partial_{x}^{\ell}(\mathcal{R}_{1}^{1}+\mathcal{R}_{-1}^{1}))^{2}dx
+\epsilon^{2}\mathcal{O}(\mathcal{E}_{s}+\epsilon^{3/2}\mathcal{E}_{s}^{3/2}).
\end{align}
Adding \eqref{J41} and \eqref{J42} together, we have
 \begin{align}\label{J4}
J_{4}
=& \epsilon^{2}\sum_{l\in\{\pm1\}}
\frac{d}{dt}\Bigg(
\int\phi_{14}\partial_{x}^{\ell}\mathcal{R}^{1}_{0}
\partial_{x}^{\ell}(\mathcal{R}_{1}^{1}-\mathcal{R}_{-1}^{1})dx\nonumber\\
&+\int\phi_{15}(\partial_{x}^{\ell}(\mathcal{R}_{0}^{1}+\mathcal{R}_{1}^{1}+\mathcal{R}_{-1}^{1}))^{2}dx
+\int(\phi_{16}+\phi_{17})(\partial_{x}^{\ell}(\mathcal{R}_{1}^{1}+\mathcal{R}_{-1}^{1}))^{2}dx\Bigg)\\
&+\epsilon^{2}\mathcal{O}(\mathcal{E}_{s}+\epsilon^{3/2}\mathcal{E}_{s}^{3/2})\nonumber.
\end{align}

For $J_{5}$, we have
\begin{align*}
J_{5}=&\epsilon^{2}\sum_{\substack{m\in\{0,\pm1\}\\ j,l\in\{\pm1\}}}
\bigg(\int\partial_{x}^{\ell}\left(B^{1,1}_{j,l,m}(\psi_{l}, \mathcal{R}^{1}_{m})
+\epsilon\mathcal{F}_{j,m}^{1}\right)
\partial_{x}^{\ell}N_{j,l,-j}^{1,1}(\psi_{l},\mathcal{R}_{-j}^{1})dx\nonumber\\
& \ \ \ \ \ \ \ \ \ \ \ \ \ \ \ \ \ +\int \partial_{x}^{\ell}\mathcal{R}_{j}^{1}
\partial_{x}^{\ell}N_{j,l,-j}^{1,1}\left(\psi_{l},\left(B^{1,1}_{-j,l,m}(\psi_{l}, \mathcal{R}^{1}_{m})
+\epsilon\mathcal{F}_{-j,m}^{1}\right)\right)dx\bigg)\nonumber\\
=&\epsilon^{2}\sum_{j,l\in\{\pm1\}}\sum_{r=1}^{7}
\bigg(\int\partial_{x}^{\ell}\left(B^{1,1}_{j,l,0}(\psi_{l}, \mathcal{R}^{1}_{0})
+\epsilon\mathcal{F}_{j,0}^{1}\right)
\partial_{x}^{\ell}N_{j,l,-j}^{1,1,r}(\psi_{l},\mathcal{R}_{-j}^{1})dx\nonumber\\
& \ \ \ \ \ \ \ \ \ \ \ \ \ \ \ \ \ \ \ \ +\int \partial_{x}^{\ell}\mathcal{R}_{j}^{1}
\partial_{x}^{\ell}N_{j,l,-j}^{1,1,r}\left(\psi_{l},\left(B^{1,1}_{-j,l,0}(\psi_{l}, \mathcal{R}^{1}_{0})
+\epsilon\mathcal{F}_{-j,0}^{1}\right)\right)dx\bigg)\nonumber\\
&+\epsilon^{2}\sum_{\substack{m\in\{\pm1\}\\ j,l\in\{\pm1\}}}\sum_{r=1}^{7}
\bigg(\int\partial_{x}^{\ell}\left(B^{1,1}_{j,l,m}(\psi_{l}, \mathcal{R}^{1}_{m})
+\epsilon\mathcal{F}_{j,m}^{1}\right)
\partial_{x}^{\ell}N_{j,l,-j}^{1,1,r}(\psi_{l},\mathcal{R}_{-j}^{1})dx\nonumber\\
& \ \ \ \ \ \ \ \ \ \ \ \ \ \ \ \ \ \ \ \ \ \ \ \ +\int \partial_{x}^{\ell}\mathcal{R}_{j}^{1}
\partial_{x}^{\ell}N_{j,l,-j}^{1,1,r}\left(\psi_{l},\left(B^{1,1}_{-j,l,m}(\psi_{l}, \mathcal{R}^{1}_{m})
+\epsilon\mathcal{F}_{-j,m}^{1}\right)\right)dx\bigg)\\
=&:\sum_{r=1}^{7}J^{r}_{5}.
\end{align*}
Using again Leibniz's rule, \eqref{17}, \eqref{F11j0} as well as the asymptotic expansions \eqref{hfj-j} and \eqref{qomega} to extract in $J^{1}_{5}$ all integral terms containing factors with more than $\ell$ spatial derivatives falling on $\mathcal{R}^{1}$ we get
\begin{align*}
J^{1}_{5}=&\epsilon^{2}\sum_{j,l\in\{\pm1\}}
\bigg(\int\partial_{x}^{\ell}\left(B^{1,1}_{j,l,0}(\psi_{l}, \mathcal{R}^{1}_{0})
+\epsilon\mathcal{F}_{j,0}^{1}\right)
N_{j,l,-j}^{1,1,1}(\psi_{l},\partial_{x}^{\ell}\mathcal{R}_{-j}^{1})dx\nonumber\\
& \ \ \ \ \ \ \ \ \ \ \ \ \ \ \ \ \ +\int \partial_{x}^{\ell}\mathcal{R}_{j}^{1}
N_{j,l,-j}^{1,1,1}\left(\psi_{l},\partial_{x}^{\ell}\left(B^{1,1}_{-j,l,0}(\psi_{l}, \mathcal{R}^{1}_{0})
+\epsilon\mathcal{F}_{-j,0}^{1}\right)\right)dx\bigg)\nonumber\\
&+\epsilon^{2}\sum_{\substack{m\in\{\pm1\}\\ j,l\in\{\pm1\}}}
\bigg(\int\partial_{x}^{\ell}\left(B^{1,1}_{j,l,m}(\psi_{l}, \mathcal{R}^{1}_{m})
+\epsilon\mathcal{F}_{j,m}^{1}\right)
N_{j,l,-j}^{1,1,1}(\psi_{l},\partial_{x}^{\ell}\mathcal{R}_{-j}^{1})dx\nonumber\\
& \ \ \ \ \ \ \ \ \ \ \ \ \ \ \ \ \ \ \ \ \ +\int \partial_{x}^{\ell}\mathcal{R}_{j}^{1}
N_{j,l,-j}^{1,1,1}\left(\psi_{l},\partial_{x}^{\ell}\left(B^{1,1}_{-j,l,m}(\psi_{l}, \mathcal{R}^{1}_{m})
+\epsilon\mathcal{F}_{-j,m}^{1}\right)\right)dx\bigg)\\
&+\epsilon^{2}\mathcal{O}(\mathcal{E}_{s}+\epsilon^{3/2}\mathcal{E}_{s}^{3/2})\\
=&\epsilon^{2}\sum_{j,l\in\{\pm1\}}
\bigg(\int\partial_{x}^{\ell}\left(B^{1,1}_{j,l,0}(\psi_{l}, \mathcal{R}^{1}_{0})
+\epsilon\mathcal{F}_{j,0}^{1}\right)
N_{j,l,-j}^{1,1,1}(\psi_{l},\partial_{x}^{\ell}\mathcal{R}_{-j}^{1})dx\nonumber\\
& \ \ \ \ \ \ \ \ \ \ \ \ \ \ \ \ \ -\int \partial_{x}^{\ell}\left(B^{1,1}_{-j,l,0}(\psi_{l}, \mathcal{R}^{1}_{0})
+\epsilon\mathcal{F}_{-j,0}^{1}\right)
N_{-j,l,j}^{1,1,1}\left(\psi_{l},\partial_{x}^{\ell}\mathcal{R}_{j}^{1}\right)dx\\
& \ \ \ \ \ \ \ \ \ \ \ \ \ \ \ \ \ +\int \partial_{x}^{\ell}\left(B^{1,1}_{-j,l,0}(\psi_{l}, \mathcal{R}^{1}_{0})
+\epsilon\mathcal{F}_{-j,0}^{1}\right)
S_{-j,j}^{1}\left(\partial_{x}\psi_{l},\partial_{x}^{\ell}\mathcal{R}_{j}^{1}\right)dx\bigg)\nonumber\\
&+\epsilon^{2}\sum_{\substack{m\in\{\pm1\}\\ j,l\in\{\pm1\}}}
\bigg(\int\partial_{x}^{\ell}\left(B^{1,1}_{j,l,m}(\psi_{l}, \mathcal{R}^{1}_{m})
+\epsilon\mathcal{F}_{j,m}^{1}\right)
N_{j,l,-j}^{1,1,1}(\psi_{l},\partial_{x}^{\ell}\mathcal{R}_{-j}^{1})dx\nonumber\\
& \ \ \ \ \ \ \ \ \ \ \ \ \ \ \ \ \ \ \ \ -\int\partial_{x}^{\ell}\left(B^{1,1}_{-j,l,m}(\psi_{l}, \mathcal{R}^{1}_{m})
+\epsilon\mathcal{F}_{-j,m}^{1}\right)
N_{-j,l,j}^{1,1,1}\left(\psi_{l},\partial_{x}^{\ell}\mathcal{R}_{j}^{1}\right)dx\\
& \ \ \ \ \ \ \ \ \ \ \ \ \ \ \ \ \ \ \ \ +\int\partial_{x}^{\ell}\left(B^{1,1}_{-j,l,m}(\psi_{l}, \mathcal{R}^{1}_{m})
+\epsilon\mathcal{F}_{-j,m}^{1}\right)
S_{-j,j}^{1}\left(\partial_{x}\psi_{l},\partial_{x}^{\ell}\mathcal{R}_{j}^{1}\right)dx\bigg)\\
&+\epsilon^{2}\mathcal{O}(\mathcal{E}_{s}+\epsilon^{3/2}\mathcal{E}_{s}^{3/2})\\
=&-\epsilon^{2}\sum_{j,l\in\{\pm1\}}
\int\partial_{x}^{\ell-1}\left(B^{1,1}_{-j,l,0}(\psi_{l}, \mathcal{R}^{1}_{0})
+\epsilon\mathcal{F}_{-j,0}^{1}\right)
\partial_{x}S_{-j,j}^{1}\left(\partial_{x}\psi_{l}, \partial_{x}^{\ell}\mathcal{R}_{j}^{1}\right)dx\\
&+\epsilon^{2}\sum_{\substack{m\in\{\pm1\}\\ j,l\in\{\pm1\}}}\int\partial_{x}^{\ell-1} \left(B^{1,1}_{-j,l,m}(\psi_{l}, \mathcal{R}^{1}_{m})
+\epsilon\mathcal{F}_{-j,m}^{1}\right)
\partial_{x}S_{-j,j}^{1}\left(\partial_{x}\psi_{l}, \partial_{x}^{\ell}\mathcal{R}_{j}^{1}\right)dx\\
&+\epsilon^{2}\mathcal{O}(\mathcal{E}_{s} +\epsilon^{3/2}\mathcal{E}_{s}^{3/2})\\
=&\epsilon^{2}\mathcal{O}(\mathcal{E}_{s} +\epsilon^{3/2}\mathcal{E}_{s}^{3/2}).
\end{align*}
Similarly, for $J_{5}^{2}$, we have
\begin{align*}
J^{2}_{5}=&\epsilon^{2}\sum_{j,l\in\{\pm1\}}
\bigg(\int\partial_{x}^{\ell}\left(B^{1,1}_{j,l,0}(\psi_{l}, \mathcal{R}^{1}_{0})
+\epsilon\mathcal{F}_{j,0}^{1}\right)
N_{j,l,-j}^{1,1,2}(\psi_{l},\partial_{x}^{\ell}\mathcal{R}_{-j}^{1})dx\nonumber\\
& \ \ \ \ \ \ \ \ \ \ \ \ \ \ \ \ \ +\int \partial_{x}^{\ell}\left(B^{1,1}_{-j,l,0}(\psi_{l}, \mathcal{R}^{1}_{0})
+\epsilon\mathcal{F}_{-j,0}^{1}\right)
N_{-j,l,j}^{1,1,2}\left(\psi_{l},\partial_{x}^{\ell}\mathcal{R}_{j}^{1}\right)dx\\
& \ \ \ \ \ \ \ \ \ \ \ \ \ \ \ \ \ +\int \partial_{x}^{\ell}\left(B^{1,1}_{-j,l,0}(\psi_{l}, \mathcal{R}^{1}_{0})
+\epsilon\mathcal{F}_{-j,0}^{1}\right)
S_{-j,j}^{2}\left(\partial_{x}\psi_{l},\partial_{x}^{\ell}\mathcal{R}_{j}^{1}\right)dx\bigg)\nonumber\\
&+\epsilon^{2}\sum_{\substack{m\in\{\pm1\}\\ j,l\in\{\pm1\}}}
\bigg(\int\partial_{x}^{\ell}\left(B^{1,1}_{j,l,m}(\psi_{l}, \mathcal{R}^{1}_{m})
+\epsilon\mathcal{F}_{j,m}^{1}\right)
N_{j,l,-j}^{1,1,2}(\psi_{l},\partial_{x}^{\ell}\mathcal{R}_{-j}^{1})dx\nonumber\\
& \ \ \ \ \ \ \ \ \ \ \ \ \ \ \ \ \ \ \ \ +\int\partial_{x}^{\ell}\left(B^{1,1}_{-j,l,m}(\psi_{l}, \mathcal{R}^{1}_{m})
+\epsilon\mathcal{F}_{-j,m}^{1}\right)
N_{-j,l,j}^{1,1,2}\left(\psi_{l},\partial_{x}^{\ell}\mathcal{R}_{j}^{1}\right)dx\\
& \ \ \ \ \ \ \ \ \ \ \ \ \ \ \ \ \ \ \ \ +\int\partial_{x}^{\ell}\left(B^{1,1}_{-j,l,m}(\psi_{l}, \mathcal{R}^{1}_{m})
+\epsilon\mathcal{F}_{-j,m}^{1}\right)
S_{-j,j}^{2}\left(\partial_{x}\psi_{l},\partial_{x}^{\ell}\mathcal{R}_{j}^{1}\right)dx\bigg)\\
&+\epsilon^{2}\mathcal{O}(\mathcal{E}_{s}+\epsilon^{3/2}\mathcal{E}_{s}^{3/2})\\
=&2\epsilon^{2}\sum_{j,l\in\{\pm1\}}
\int\partial_{x}^{\ell}\left(B^{1,1}_{j,l,0}(\psi_{l}, \mathcal{R}^{1}_{0})
+\epsilon\mathcal{F}_{j,0}^{1}\right)
N_{j,l,-j}^{1,1,2}(\psi_{l},\partial_{x}^{\ell}\mathcal{R}_{-j}^{1})dx\nonumber\\
&+\epsilon^{2}\sum_{\substack{m\in\{\pm1\}\\ j,l\in\{\pm1\}}}\int\partial_{x}^{\ell}\left(B^{1,1}_{j,l,m}(\psi_{l}, \mathcal{R}^{1}_{m})
+\epsilon\mathcal{F}_{j,m}^{1}\right)
N_{j,l,-j}^{1,1,2}(\psi_{l},\partial_{x}^{\ell}\mathcal{R}_{-j}^{1})dx\nonumber\\
&+\epsilon^{2}\mathcal{O}(\mathcal{E}_{s}+\epsilon^{3/2}\mathcal{E}_{s}^{3/2})\\
=&\frac{\epsilon^{2}}{2}\sum_{l\in\{\pm1\}}
\bigg(\int\psi_{l}\phi_{7}\partial_{x}^{\ell}\mathcal{R}_{0}^{1}
\partial_{x}^{\ell+1}q(\mathcal{R}_{1}^{1}-\mathcal{R}_{-1}^{1})dx\\
& \ \ \ \ \ \ \ \ \ \ \ \ \  +\int(\psi_{l}\phi_{8}\partial_{x}^{\ell}(\mathcal{R}_{1}^{1}+\mathcal{R}_{-1}^{1})
+\psi_{l}\phi_{9}\partial_{x}^{\ell}q(\mathcal{R}_{1}^{1}-\mathcal{R}_{-1}^{1}))
\partial_{x}^{\ell+1}q(\mathcal{R}_{1}^{1}-\mathcal{R}_{-1}^{1})dx\bigg)\nonumber\\
&+\epsilon^{2}\mathcal{O}(\mathcal{E}_{s}+\epsilon^{3/2}\mathcal{E}_{s}^{3/2})\\
=&\frac{\epsilon^{2}}{4}\sum_{l\in\{\pm1\}}\frac{d}{dt}\bigg(\int\frac{\psi_{l}\phi_{7}}{1+\epsilon\phi_{3}}
(\partial_{x}^{\ell}(\mathcal{R}_{0}^{1}+\mathcal{R}_{1}^{1}+\mathcal{R}_{-1}^{1}))^{2}dx
-\int\frac{\psi_{l}\phi_{7}}{1+\epsilon\phi_{4}}
(\partial_{x}^{\ell}(\mathcal{R}_{1}^{1}+\mathcal{R}_{-1}^{1}))^{2}dx\\
& \ \ \ \ \ \ \ \ \ \ \ \ \ \ \ \ \ +\int\frac{\psi_{l}\phi_{8}}{1+\epsilon\phi_{4}}
(\partial_{x}^{\ell}(\mathcal{R}_{1}^{1}+\mathcal{R}_{-1}^{1}))^{2}dx\bigg) +\epsilon^{2}\mathcal{O}(\mathcal{E}_{s}+\epsilon^{3/2}\mathcal{E}_{s}^{3/2}).
\end{align*}
$J_{5}^{4}$ has similar estimates with the $J_{5}^{1}$; $J_{5}^{3}$ and $J_{5}^{5}$ have similar estimates with $J_{5}^{2}$. Since $N_{j,l,-j}^{1,1,\{6,7\}}$ even gains at least one derivative, $J_{5}^{\{6,7\}}$ can be bounded by $\epsilon^{2}\mathcal{O}(\mathcal{E}_{s}+\epsilon^{3/2}\mathcal{E}_{s}^{3/2})$ directly.
In conclusion, we have
 \begin{align}\label{J5}
J_{5}
=& \epsilon^{2}\sum_{l\in\{\pm1\}}
\frac{d}{dt}\Bigg(
\int\phi_{18}(\partial_{x}^{\ell}(\mathcal{R}_{0}^{1}+\mathcal{R}_{1}^{1}+\mathcal{R}_{-1}^{1}))^{2}dx
+\int\phi_{19}(\partial_{x}^{\ell}(\mathcal{R}_{1}^{1}+\mathcal{R}_{-1}^{1}))^{2}dx\Bigg)\\
&+\epsilon^{2}\mathcal{O}(\mathcal{E}_{s}+\epsilon^{3/2}\mathcal{E}_{s}^{3/2})\nonumber.
\end{align}

For $J_{6}$, we have
\begin{align*}
J_{6}=&\epsilon^{2}\sum_{j,n\in\{0,\pm1\}}
\int\partial_{x}^{\ell}\mathcal{R}_{j}^{1}\partial_{x}^{\ell}\mathcal{F}_{j,n}^{1}dx\nonumber\\
=&\epsilon^{2}
\int\partial_{x}^{\ell}\mathcal{R}_{0}^{1}\partial_{x}^{\ell}\mathcal{F}_{0,0}^{1}dx
+\epsilon^{2}\sum_{n\in\{\pm1\}}
\int\partial_{x}^{\ell}\mathcal{R}_{0}^{1}\partial_{x}^{\ell}\mathcal{F}_{0,n}^{1}dx\nonumber\\
&+\epsilon^{2}\sum_{j\in\{\pm1\}}
\int\partial_{x}^{\ell}\mathcal{R}_{j}^{1}\partial_{x}^{\ell}\mathcal{F}_{j,0}^{1}dx
+\epsilon^{2}\sum_{j,n\in\{\pm1\}}
\int\partial_{x}^{\ell}\mathcal{R}_{j}^{1}\partial_{x}^{\ell}\mathcal{F}_{j,n}^{1}dx.
\end{align*}
Using \eqref{F00}-\eqref{F01}, we extract in $J_{6}$ all integral terms containing factors with more than $\ell$ spatial derivatives falling on $\mathcal{R}^{1}$
\begin{align*}
J_{6}
=&\epsilon^{2}
\int\widetilde{\phi}_{1}\partial_{x}^{\ell}\mathcal{R}_{0}^{1}\partial_{x}^{\ell+1}\mathcal{R}_{0}^{1}dx
+\epsilon^{2}
\int\widetilde{\phi}_{2}\partial_{x}^{\ell}\mathcal{R}_{0}^{1}
\partial_{x}^{\ell+1}q(\mathcal{R}_{1}^{1}-\mathcal{R}_{-1}^{1})dx\nonumber\\
&+\frac{j\epsilon^{2}}{2\sqrt{\gamma}}\sum_{j\in\{\pm1\}}
\int\widetilde{\phi}_{2}\partial_{x}^{\ell}\mathcal{R}_{j}^{1}
\partial_{x}^{\ell+1}\mathcal{R}_{0}^{1}dx
+\frac{j\epsilon^{2}}{2}\sum_{j\in\{\pm1\}}
\int\widetilde{\phi}_{1}\partial_{x}^{\ell}\mathcal{R}_{j}^{1}
\partial_{x}^{\ell+1}(\mathcal{R}_{1}^{1}+\mathcal{R}_{-1}^{1})dx\nonumber\\
&+\epsilon^{2}\sum_{j\in\{\pm1\}}
\int(\widetilde{\phi}_{3}+\frac{j}{2\sqrt{\gamma}}\widetilde{\phi}_{1})\partial_{x}^{\ell}\mathcal{R}_{j}^{1}
\partial_{x}^{\ell+1}q(\mathcal{R}_{1}^{1}-\mathcal{R}_{-1}^{1})dx +\epsilon^{2}\mathcal{O}(\mathcal{E}_{s}+\epsilon^{3/2}\mathcal{E}_{s}^{3/2}),
\end{align*}
where
\begin{align*}
\widetilde{\phi}_{1}&=q(\varphi_{1}+\epsilon^{\beta-2}(\mathcal{R}_{1}^{1}-\mathcal{R}_{-1}^{1})),\\
\widetilde{\phi}_{2}&=-\frac{\gamma-1}{\gamma}(\varphi_{2}
+\epsilon^{\beta-2}((\gamma-2-q^{2})\mathcal{R}_{0}^{1}+(\gamma-2)\mathcal{R}_{1}^{1}
+(\gamma-2)\mathcal{R}_{-1}^{1})),\\
\widetilde{\phi}_{3}&=\frac{1}{2(\varphi_{3}+\epsilon^{\beta-2}(\mathcal{R}_{0}^{1}
+\mathcal{R}_{1}^{1}+\mathcal{R}_{-1}^{1}))}
+\frac{\gamma-1}{2\gamma}\widetilde{\phi}_{2}.
\end{align*}
Applying integration by parts and \eqref{1+-}-\eqref{1--}, we have
\begin{align}\label{J6}
J_{6}
=&(1-\frac{1}{2\gamma})\epsilon^{2}
\int\widetilde{\phi}_{2}\partial_{x}^{\ell}\mathcal{R}_{0}^{1}
\partial_{x}^{\ell+1}q(\mathcal{R}_{1}^{1}-\mathcal{R}_{-1}^{1})dx\nonumber\\
&+\epsilon^{2}
\int(\widetilde{\phi}_{3}-\frac{1}{2\sqrt{\gamma}}\widetilde{\phi}_{1})\partial_{x}^{\ell+1}q(\mathcal{R}_{1}^{1}-\mathcal{R}_{-1}^{1})
\partial_{x}^{\ell}(\mathcal{R}_{1}^{1}+\mathcal{R}_{-1}^{1})dx\nonumber\\
&+\epsilon^{2}\mathcal{O}(\mathcal{E}_{s}+\epsilon^{3/2}\mathcal{E}_{s}^{3/2})\nonumber\\
=&(\frac{1}{2}-\frac{1}{4\gamma})\epsilon^{2}\frac{d}{dt}
\int\left(\frac{\widetilde{\phi}_{2}}{1+\epsilon\phi_{3}}(\partial_{x}^{\ell}(\mathcal{R}_{0}^{1}
+\mathcal{R}_{1}^{1}+\mathcal{R}_{-1}^{1}))^{2}-
\frac{\widetilde{\phi}_{2}}{1+\epsilon\phi_{4}}(\partial_{x}^{\ell}(
\mathcal{R}_{1}^{1}+\mathcal{R}_{-1}^{1}))^{2}\right)dx\nonumber\\
&+\frac{\epsilon^{2}}{2}
\frac{d}{dt}\int\frac{(\widetilde{\phi}_{3}-\frac{1}{2\sqrt{\gamma}}\widetilde{\phi}_{1})}{1+\epsilon\phi_{4}}
(\partial_{x}^{\ell}(\mathcal{R}_{1}^{1}+\mathcal{R}_{-1}^{1}))^{2}dx +\epsilon^{2}\mathcal{O}(\mathcal{E}_{s} +\epsilon^{3/2}\mathcal{E}_{s}^{3/2}).
\end{align}

Now we can complete the proof. Define the modified energy $\widetilde{\mathcal{E}}_{s}$ as
\begin{align*}
\widetilde{\mathcal{E}}_{s}=\mathcal{E}_{s}+\epsilon^{2}\sum_{\ell=1}^{s}h_{\ell},
\end{align*}
with
\begin{align*}
h_{\ell}=\int\widetilde{\phi}_{4}\partial_{x}^{\ell}\mathcal{R}^{1}_{0}
\partial_{x}^{\ell}(\mathcal{R}_{1}^{1}-\mathcal{R}_{-1}^{1})dx
+\int\widetilde{\phi}_{5}(\partial_{x}^{\ell}(\mathcal{R}_{0}^{1}+\mathcal{R}_{1}^{1}+\mathcal{R}_{-1}^{1}))^{2}dx
+\int\widetilde{\phi}_{6}(\partial_{x}^{\ell}(\mathcal{R}_{1}^{1}+\mathcal{R}_{-1}^{1}))^{2},
\end{align*}
where $\widetilde{\phi}_{4},\widetilde{\phi}_{5}$ and $\widetilde{\phi}_{6}$ depend on $\ell,\gamma,\psi_{l},\phi_{r}$ with $r\in\{1\sim5,8,9\}$. Adding \eqref{J1}, \eqref{J2}, \eqref{J3}, \eqref{J4}, \eqref{J5} and \eqref{J6} together, we obtain
\begin{align*}
\partial_{t}\widetilde{\mathcal{E}}_{s}\lesssim \epsilon^{2}(1+\widetilde{\mathcal{E}}_{s}+\epsilon^{1/2}\widetilde{\mathcal{E}}_{s}^{3/2}
+\epsilon\widetilde{\mathcal{E}}_{s}^{2}).
\end{align*}
Consequently, Gronwall's inequality yields for sufficiently small $\epsilon>0$ the $\mathcal{O}(1)$-boundedness of $\widetilde{\mathcal{E}}_{s}$ for all $t\in[0,T_{0}/\epsilon^{2}]$. Because of $\|(R_{0},R_{1},R_{-1})\|_{H^s}\lesssim\sqrt{\widetilde{\mathcal{E}}_{s}}$ for sufficiently small $\epsilon>0$, Theorem \ref{Thm1} follows, thanks to \eqref{Aesti-2} .

\bigskip

\textbf{Conflict of interest:} The authors declare that they have no conflicts of interest.

\bigskip

\textbf{Data Availability:} The manuscript contains no associated data.

\end{document}